\documentclass[11pt,a4paper]{article}
\usepackage[english]{babel}

\usepackage[text={16.5cm, 23.0cm}, centering]{geometry}

\usepackage[T1]{fontenc}
\usepackage{lmodern}
\usepackage{microtype}
\usepackage{amssymb,amscd,amsmath,amsthm}
\usepackage{bbold}
\usepackage{color}
\usepackage{tikz,pgfplots}
\usepackage{hyperref}
\hypersetup{pdfborder=0 0 0}
\usepackage{color}
\usepackage{marginnote} % for \marginnote, a replacement for \marginpar that works everywhere
\usepackage{ifthen} % for \ifthenelse
\usepackage{paralist} % for compactitem, etc.
\usepackage{graphicx}

\usepackage{pgfplots}
\pgfplotsset{compat=1.14}

\newtheorem{theorem}{Theorem}[section]
\newtheorem{lemma}[theorem]{Lemma}
\newtheorem{remark}[theorem]{Remark}

\numberwithin{equation}{section}
\newtheorem{corollary}[theorem]{Corollary}
\newtheorem{proposition}[theorem]{Proposition}

\newtheorem{definition}[theorem]{Definition}

\newcommand{\bbC}{\mathbb{C}}
\newcommand{\bbH}{\mathbb{H}}
\newcommand{\R}{\mathbb{R}}

\newcommand{\X}{\mathbb{X}}
\newcommand{\Y}{\mathbb{Y}}
\newcommand{\rmP}{\mathrm{P}}
\newcommand{\cG}{\mathcal{G}}
\newcommand{\cH}{\mathcal{H}}
\newcommand{\cL}{\mathcal{L}}
\newcommand{\cM}{\mathcal{M}}
\newcommand{\cT}{\mathcal{T}}
\renewcommand{\sl}{\mathfrak{sl}}
\newcommand{\x}{\mathbf{x}}
\newcommand{\y}{\mathbf{y}}
\newcommand{\AdS}{\mathrm{AdS}}
\newcommand{\dS}{\mathrm{dS}}
\newcommand{\Mink}{\mathrm{Mink}}
\newcommand{\RP}{\R\rmP}
\newcommand{\CP}{\bbC\rmP}

  \DeclareFontFamily{U}{wncy}{}
    \DeclareFontShape{U}{wncy}{m}{n}{<->wncyr10}{}
    \DeclareSymbolFont{mcy}{U}{wncy}{m}{n}
    \DeclareMathSymbol{\Lob}{\mathord}{mcy}{"6C} 

\newcommand{\inv}[0]{{-1}}

\newcommand{\Rquotient}[2]{{\raisebox{.1em}{$#1$}\raisebox{-.1em}{$\Big/$}\raisebox{-.3em}{$#2$}}}
\newcommand{\rquotient}[2]{{\raisebox{.1em}{$#1$}\raisebox{-.1em}{$/$}\raisebox{-.3em}{$#2$}}}

\makeatletter
\newcommand\RedeclareMathOperator{%
  \@ifstar{\def\rmo@s{m}\rmo@redeclare}{\def\rmo@s{o}\rmo@redeclare}%
}
\newcommand\rmo@redeclare[2]{%
  \begingroup \escapechar\m@ne\xdef\@gtempa{{\string#1}}\endgroup
  \expandafter\@ifundefined\@gtempa
     {\@latex@error{\noexpand#1undefined}\@ehc}%
     \relax
  \expandafter\rmo@declmathop\rmo@s{#1}{#2}}
\newcommand\rmo@declmathop[3]{%
  \DeclareRobustCommand{#2}{\qopname\newmcodes@#1{#3}}%
}
\@onlypreamble\RedeclareMathOperator
\makeatother

\RedeclareMathOperator{\Re}{Re}
\RedeclareMathOperator{\Im}{Im}
\DeclareMathOperator{\sgn}{sgn}

\DeclareMathOperator{\tr}{tr}
\DeclareMathOperator{\Hom}{Hom}
\DeclareMathOperator{\im}{i}
\DeclareMathOperator{\Mat}{Mat}
\DeclareMathOperator{\PSL}{PSL}
\DeclareMathOperator{\PGL}{PGL}
\DeclareMathOperator{\PO}{PO}
\DeclareMathOperator{\PSU}{PSU}
\DeclareMathOperator{\Stab}{Stab}
\DeclareMathOperator{\Isom}{Isom}
\DeclareMathOperator{\vol}{vol}
\DeclareMathOperator{\Un}{U}
\DeclareMathOperator{\Cl}{Cl}

\begin{document}

\title{Lightlike and ideal tetrahedra}

\author{Catherine Meusburger\footnote{{\tt catherine.meusburger@math.uni-erlangen.de}}\\
{\it Department Mathematik},\\ {\it Friedrich-Alexander-Universit\"at Erlangen-N\"urnberg,} \\ {\it Cauerstra\ss e 11, 91058 Erlangen, Germany}\\
\\
Carlos Scarinci\footnote{{\tt cscarinci@ewha.ac.kr}} \\ {\it Institute of Mathematical Sciences,} \\ {\it Ewha Womans University,} \\ {\it 52 Ewhayeodae-gil, Seodaemun-gu, 03760 Seoul, Republic of Korea}
}

\maketitle

\abstract{We give a unified description of tetrahedra with lightlike faces in 3d anti-de Sitter, de Sitter and Minkowski spaces and of their duals in 3d anti-de Sitter, hyperbolic and half-pipe spaces. We show that both types of tetrahedra are determined by a generalized cross-ratio with values in a commutative 2d real algebra that generalizes the complex numbers. Equivalently, tetrahedra with lightlike faces are determined by a pair of edge lengths and their duals by a pair of dihedral angles. We prove that the dual tetrahedra are precisely the generalized ideal tetrahedra introduced by Danciger. Finally, we compute the volumes of both types of tetrahedra as functions of their edge lengths or dihedral angles, obtaining generalizations of the Milnor-Lobachevsky volume formula of ideal hyperbolic tetrahedra.}

\setlength{\parindent}{0pt}
\setlength{\parskip}{7pt plus 2pt minus 1pt}

\section{Introduction}

  \paragraph{\bf Ideal hyperbolic tetrahedra} Hyperbolic ideal tetrahedra are fundamental building blocks in 3d hyperbolic geometry. They are geodesic tetrahedra in $\bbH^3$ with vertices in the ideal boundary $\partial_\infty\bbH^3\cong\CP^1$. As they are determined by their vertices, they are parametrized, up to isometries, by a single complex parameter $z\in\bbC\backslash\{0,1\}$,  its shape parameter or cross-ratio.
  
  The general approach to the construction of 3d hyperbolic structures via hyperbolic ideal tetrahedra was introduced by Thurston in \cite{Thurston}. Starting with a topological 3-manifold $M$ with a topological ideal triangulation, one chooses hyperbolic structures on the tetrahedra that glue smoothly into a hyperbolic structure on $M$. The consistency conditions for the gluing determine a system of algebraic equations on the set of shape parameters. Under a few additional assumptions, solutions to these gluing equations define a smooth hyperbolic structure on $M$. 
  
  This construction is a powerful tool in 3d hyperbolic geometry. Given a hyperbolic 3-manifold $M$ with a geodesic ideal triangulation and solutions of Thurston's gluing equations, one can in principle compute many invariants of $M$. In particular, the hyperbolic volume of $M$ can be computed as the sum of volumes of each ideal tetrahedron \cite{Thurston}, see also \cite{NZ}, which is a well-know function of the shape parameter \cite{Mi}.

  \paragraph{\bf Generalized ideal tetrahedra}
  This description of hyperbolic 3-manifolds in terms of ideal hyperbolic tetrahedra can be generalized to other geometries. In \cite{DaPhD, Da2} Danciger introduced a generalized notion of ideal tetrahedra in 3d anti-de Sitter and 3d half-pipe spaces and studied a generalized version of Thurston's gluing equations.
  
  Denoting by $\Y_\Lambda$ the 3d hyperbolic space for $\Lambda>0$, the 3d anti-de Sitter space for $\Lambda<0$ and the 3d half-pipe space for $\Lambda=0$, one can describe these generalized ideal tetrahedra as geodesic tetrahedra in $\Y_\Lambda$ with vertices at the ideal boundary $\partial_\infty\Y_\Lambda$ and with {\em spacelike} edges. The additional condition that the edges are spacelike imposes restrictions on the relative position of the vertices at the asymptotic boundary. Nonetheless, generalized ideal tetrahedra are also parametrized, up to isometries, by a single shape parameter, now taking values in the ring of generalized complex numbers $\bbC_\Lambda$. See also \cite{Luo}, for a general discussion of gluing equations over commutative rings.
  
  Generalized ideal tetrahedra share many properties with their hyperbolic counterparts and thus offer the prospect to generalize results and constructions from hyperbolic geometry to 3d anti-de Sitter and half-pipe geometry. In particular, they were applied by Danciger in \cite{DaPhD,Da2} to construct geometric transitions between hyperbolic and anti-de Sitter structures, going through half-pipe structures, and were also used as building blocks for the study of more general polyhedra in \cite{DMS}. 
  
  A particularly interesting quantity in this respect is the hyperbolic volume. The volume of a generalized ideal tetrahedron can be defined as the integral of a 3-form invariant under the action of the isometry group, which is unique up to global rescaling. However, so far there is no anti-de Sitter or half-pipe analogue of the Milnor-Lobachevsky formula for the volume in this setting. This raises
  
  {\bf Question 1:} Is there a simple formula for the volume of a generalized ideal tetrahedron in $\Y_\Lambda$ as a function of its shape parameter and the parameter $\Lambda$ that controls the geometric transitions?

  \paragraph{\bf 3d Lorentzian geometry} Another strong motivation to investigate generalized ideal tetrahedra is the close relation between structures from 2d and 3d hyperbolic geometry and 3d Einstein geometry in Lorentzian signature. Every 3d Lorentzian Einstein manifold $M$ is locally isometric to a homogeneous and isotropic Lorentzian 3d manifold $\X_\Lambda$ of constant curvature $\Lambda$, namely the 3d de Sitter space for $\Lambda>0$, the 3d anti-de Sitter space for $\Lambda<0$ and the 3d Minkowski space for $\Lambda=0$. The geometry of  $M$ can then be described by geometric structures modeled on $\X_\Lambda$ and with structure group $G_\Lambda=\Isom_0(\X_\Lambda)$, that is, by an atlas of coordinate charts valued in $\X_\Lambda$ with isometric transition functions.
  
  Under additional assumptions on causality, namely maximal global hyperbolicity and the completeness of a Cauchy surface $S$, there is a full classification result \cite{Mess,Scannell,Barbot,Benedetti-Bonsante}, which characterizes the 3d Einstein manifolds in terms of structures from 2d and 3d hyperbolic geometry. More specifically, it identifies the moduli space $\cG\cH_\Lambda(M)$ of maximal globally hyperbolic Einstein metrics, modulo isotopy, on a 3-manifold $M=\R\times S$ with the bundle $\cM\cL(S)$ of bounded measured geodesic laminations over the Teichm\"uller space $\cT(S)$ of the Cauchy surface.
  
  For each value of $\Lambda$, this identification is given by a Lorentzian counterpart of the grafting construction from 3d hyperbolic geometry. Moreover, the Lorentzian grafting construction is directly related to hyperbolic grafting via the Wick-rotation and rescaling theory developed by Benedetti and Bonsante \cite{Benedetti-Bonsante}. It was also shown by the first author in \cite{Me} that these constructions admit a unified description via the ring of generalized complex numbers $\bbC_\Lambda$.

  \paragraph{\bf Symplectic structures and mapping class group actions}
  The moduli spaces $\cG\cH_\Lambda(M)$ admit a symplectic structure induced by Goldman's symplectic structure \cite{G1,G2} on the spaces of holonomies $\Hom(\pi_1(S), G_\Lambda)/G_\Lambda$. This is a natural Lorentzian generalization of (the imaginary part of) Goldman's symplectic structure on the moduli space of quasi-Fuchsian hyperbolic 3-manifolds or, more generally, the moduli space of hyperbolic end 3-manifolds. In fact, these structures are closely related via Wick-rotation and rescaling theory. More precisely, it was shown by the second author in joint work with Schlenker  \cite{SS}, that Wick rotations induce symplectic diffeomorphisms between the moduli spaces $\cG\cH_\Lambda(M)$ and the moduli space of hyperbolic end 3-manifolds for all values of $\Lambda$.
  
  In \cite{MSc} we showed that these symplectic structures can be given a unified description in terms of $\bbC_\Lambda$-valued shear coordinates associated with ideal triangulations of a punctured Cauchy surface.  This description generalizes the Weil-Petersson symplectic structure on Teichm\"uller space $\cT(S)$, and leads to a simple description of the mapping class group action in terms of 2d Whitehead moves. Interestingly, they involve $\bbC_\Lambda$-analytic continuations of classical dilogarithms, which suggests a close relation to the volumes of ideal hyperbolic tetrahedra.

  \paragraph{\bf Generalized ideal tetrahedra and their duals} 
  The role of hyperbolic structures in 3d Lorentzian geometry suggests that there should be a distinguished class of tetrahedra in 3d de Sitter, Minkowski and anti-de Sitter space with structural similarities to ideal tetrahedra, such as a simple description in terms of shape parameters.

  {\bf Question 2:} Are there analogues of generalized ideal tetrahedra in the spaces $\X_\Lambda$ with similar geometric properties?

  If the answer to this question is yes, one may generalize Question 1 to these tetrahedra and ask whether the geometry of these tetrahedra is simple enough to admit a volume formula in terms of simple quantities such as shape parameters and similar to the Milnor-Lobachevsky formula.

  {\bf Question 3:} Is there a simple volume formula for these tetrahedra in $\X_\Lambda$?

  In this article, we show that the answers to these three questions are positive.  More specifically, we show that the analogues of generalized ideal tetrahedra in the Lorentzian spaces $\X_\Lambda$ are the geodesic tetrahedra whose faces lie in  lightlike geodesic planes.

  We also find that they are related to Danciger's generalized ideal tetrahedra from \cite{Da2} via the projective duality between the spaces $\X_\Lambda$ and $\Y_\Lambda$ (Theorem \ref{prop:dualprop}).  This duality pairs points in one space with  (totally)  geodesic spacelike planes in the other.  It admits a natural extension to the ideal boundary, which assigns points in $\partial_\infty\Y_\Lambda$ to lightlike geodesic planes in $\X_\Lambda$, and hence pairs generalized ideal tetrahedra in $\Y_\Lambda$ and lightlike tetrahedra in $\X_\Lambda$.

  We achieve this via a unified description of the spaces $\X_\Lambda$ and $\Y_\Lambda$ in terms of $2\times 2$-matrices with entries in $\bbC_\Lambda$. This description leads to simple expressions for the geodesics, geodesic planes, metrics and isometry group actions on both spaces, and also for the ideal boundary of $\Y_\Lambda$. It allows us to parametrize both lightlike and ideal tetrahedra, to investigate their geometry in detail and to explicitly relate them.
  
  In particular, we show in Proposition \ref{prop:ltetrahedron} that lightlike tetrahedra are also parameterized by pair of real parameters $\alpha,\beta\in\R$ or, equivalently, by a generalized complex number $z\in \bbC_\Lambda$. These parameters have simple geometric interpretations, analogous to the ones for ideal tetrahedra.  For example, the parameters  $|\alpha|,|\beta|,|\alpha+\beta|$ represent edge lengths of the lightlike tetrahedron, with opposite edges having equal length. Under duality, these lengths correspond to the dihedral angles of the dual ideal tetrahedron.
  
  \paragraph{\bf Volumes of generalized ideal tetrahedra and their duals}
  We also apply the explicit parametrization of lightlike and ideal tetrahedra to derive a unified formula for their volumes as a function of the parameters $\alpha,\beta$. For a generalized ideal tetrahedron  $ I \subset \Y_\Lambda$ the resulting formula in Theorem \ref{prop:volideal} is a generalization of the Milnor-Lobachevsky volume formula for ideal hyperbolic tetrahedra, involving $\Lambda$ as a deformation parameter
  \begin{align*}
  \vol( I )=\frac{1}{2} \Big(\Cl_\Lambda(2\alpha)+\Cl_\Lambda(2\beta)+\Cl_\Lambda(2\gamma)\Big),\qquad \alpha+\beta+\gamma = 0.
  \end{align*}
  Here, $\Cl_\Lambda$ is a generalized Clausen function. It coincides with the usual Clausen function for $\Lambda>0$, the hyperbolic Clausen function for $\Lambda<0$ and the integral of a logarithmic function for $\Lambda=0$.
   The volume computation for a lightlike tetrahedron $ L \subset\X_\Lambda$ is more involved and is achieved in Theorem \ref{prop:dualvol}. The result is again a very simple expression involving $\Lambda$ as a deformation parameter 
  \begin{align*}
  \vol( L )=&\frac{1}{2\Lambda}\Big(\Cl_\Lambda(2\alpha)+\Cl_\Lambda(2\beta)+\Cl_\Lambda(2\gamma)\Big)
  \cr
  +&\;\frac 1 {\Lambda}\,\Big( \alpha\log|s_\Lambda(\alpha)|+\beta\log|s_\Lambda(\beta)|+\gamma \log|s_\Lambda(\gamma)|\Big),
  \end{align*}
  with $s_\Lambda$ given by the sine function for $\Lambda>0$ and the hyperbolic sine function for $\Lambda<0$. The volume for $\Lambda=0$  can be either computed directly or as the limit $\Lambda\to 0$ from a power series expansion around $\Lambda=0$  and reads $\vol( L )=-\alpha\beta\gamma/3$.

  \setlength{\parskip}{0pt}
  \tableofcontents
  
  \setlength{\parindent}{0pt}
  \setlength{\parskip}{7pt plus 2pt minus 1pt}

  \section{Lorentzian 3d geometries and their duals}
  
  \subsection{Projective models}
  
  In this section, we describe the 3d Lorentzian geometries considered in this article and their duals. We work with a projective formulation that identifies these spaces with subsets of $\RP^3$. 
  We denote by $\R^{p,q,r}$ the vector space $\R^{p+q+r}$ endowed with the symmetric bilinear form of signature $(p,q,r)$
  \begin{align}\label{eq:blf}
  \langle x,x\rangle_{p,q,r}=-(x_1)^2-\cdots-(x_p)^2+(x_{p+q+1})^2+\cdots+(x_{p+q+r})^2.
  \end{align}
  A vector $x\in \R^{p,q,r}$ is called timelike if $\langle x,x\rangle<0$, spacelike if $\langle x,x\rangle>0$ and lightlike if $x\neq 0$ and $\langle x,x\rangle=0$. 
  We call it a unit vector or normalized if $|\langle x,x\rangle|=1$ or if $|\langle x,x\rangle|=0$.

  \paragraph{\bf Anti-de Sitter space}
  
  The Klein model of 3d anti-de Sitter space can be defined as the space of timelike lines through the origin in $\R^{2,0,2}$
  \begin{align}\label{eq:ads}
  \AdS^3=\Rquotient{\Big\{x\in\R^{2,0,2} \mid \langle x,x\rangle_{2,0,2}<0\Big\}}{\R^\times}\subset \RP^3.
  \end{align}
  This can also be seen as the quotient of the hyperboloid of unit timelike vectors in $\R^{2,0,2}$ by the antipodal map and thus inherits a Lorentzian metric of constant sectional curvature $-1$.
  
  The group of orientation preserving isometries of $\AdS^3$ is   $\PO_0(2,2)\cong\PSL(2,\R)\times\PSL(2,\R)$. It acts transitively on $\AdS^3$.
  The full group of isometries of $\AdS^3$ is the group $\mathrm{PO}(2,2)\subset \PGL(2,\R)\times\PGL(2,\R)$. It is a double cover of $\PO_0(2,2)$ and is generated by $\PO_0(2,2)$ together with the isometry $[(x_1,x_2,x_3,x_4)]\mapsto [(x_1,-x_2,x_3, x_4)]$.

  \paragraph{\bf de Sitter space}
  The Klein model of 3d de Sitter space can be defined similarly as the space of spacelike lines through the origin in $\R^{1,0,3}$
  \begin{align}\label{eq:ds}
  \dS^3=\Rquotient{\Big\{x\in\R^{1,0,3} \mid \langle x,x\rangle_{1,0,3}>0\Big\}}{\R^\times}\subset \RP^3.
  \end{align}
  It is the quotient of the hyperboloid of unit spacelike vectors in $\R^{1,0,3}$ by the antipodal map and thus inherits a Lorentzian metric with sectional curvature $+1$. Note that with this definition $\dS^3$ is orientable, but not time orientable.
  
  The group of orientation preserving isometries is $\PO_0(1,3)\cong\PGL(2,\bbC)$. It acts transitively on $\dS^3$. The full isometry group is the group $\PO(1,3)$,  generated by $\PO_{0}(1,3)$ and $[(x_1,x_2,x_3,x_4)]\mapsto [(x_1,-x_2,x_3, x_4)]$.

  \paragraph{\bf Minkowski space}
  We also consider a Klein model of 3d Minkowski space. This is defined as the space of lines through the origin in $\R^{1,1,2}$ transversal to the hyperplane $x_2=0$
  \begin{align}\label{eq:mink}
  \Mink^3=\Rquotient{\Big\{x\in\R^{1,1,2} \mid (x_2)^2>0\Big\}}{\R^\times}\subset \RP^3.
  \end{align}
  As $\Mink^3$ can be identified with the hyperplane $H=\{x\in \R^{1,1,2}\mid x_2=1\}$, it inherits a Lorentzian metric of sectional curvature 0.
  The group of orientation preserving isometries of $\Mink^3$ is the Poincar\'e group in 3 dimensions $\PO_0(1,1,2)=\PO_0(1,2)\ltimes\R^{1,2}\cong\PSL(2,\R)\ltimes\sl(2,\R)$. It acts transitively on $\Mink^3$. 
  The full isometry group of $\Mink^3$  is the group $\PO(1,1,2)=\PO(1,2)\ltimes\R^{1,2}\cong \PGL(2,\R)\ltimes\sl(2,\R)$. It is generated by $\PO_0(1,1,2)$ and the isometry $[(x_1,x_2,x_3,x_4)]\mapsto [(x_1,-x_2,x_3, x_4)]$.
  
  In the following, we denote these three projective quadrics in $\RP^3$ by $\X_\Lambda$, where  $\Lambda\in\{-1,0,1\}$ is the sectional curvature of the quadric
  \begin{align*}
  \X_\Lambda=\begin{cases} \AdS^3, & \Lambda=-1,\\
  \dS^3, & \Lambda=1,\\
  \Mink^3, & \Lambda=0.
  \end{cases}
  \end{align*}

  \paragraph{\bf Dual models} 
   The projective quadrics $\X_\Lambda\subset\RP^3$  can also be characterized by their duality to three other projective quadrics  $\Y_\Lambda\subset\RP^3$ for $\Lambda=-1,0,1$. 
  The latter are defined
  as the spaces of timelike lines through the origin in $\R^4$  
  \begin{align}\label{eq:y}
  \Y_\Lambda=\Rquotient{\Big\{y\in\R^4 \mid \langle y,y\rangle_\Lambda<0\Big\}}{\R^\times}\subset \RP^3.
  \end{align}
  with respect to the 
   symmetric bilinear form
  \begin{align}\label{eq:yform}
  \langle y,y\rangle_{\Lambda}=-y_1^2+\Lambda y_2^2+y_3^2+y_4^2.
  \end{align}
  As  $\Y_\Lambda$ is the quotient of the set of timelike unit vectors for $\langle \cdot, \cdot\rangle_\Lambda$ by the antipodal map, it also inherits a constant curvature metric. For $\Lambda=-1$, this is again a Lorentzian metric of sectional curvature $-1$, and  $\Y_{-1}$ is identical to $\X_{-1}=\AdS^3$. For $\Lambda=1$ one obtains a Riemannian metric of sectional curvature $-1$, and $\Y_{1}$ is the Klein model of 3d hyperbolic space $\bbH^3$.
  For $\Lambda=0$ one has a degenerate metric of signature $(0,0,2)$, and  $\Y_{0}=\bbH^2\times\R$ is the product of 2d hyperbolic space with the real line, the so called co-Minkowski or half-pipe space, see for instance \cite{DaPhD,Da1,BF, Fillastre-Seppi}. Thus,
   \begin{align*}
  \Y_\Lambda=\begin{cases} 
  \AdS^3, & \Lambda=-1,\\
  \bbH^3, & \Lambda=1,\\
  \bbH^2\times\R, & \Lambda=0.
  \end{cases}
  \end{align*}
  For each value of $\Lambda$, the isometry group of $\Y_\Lambda$ agrees with the isometry group of  $\X_\Lambda$. The isotropy groups, however, are different.
  
  \subsection{Projective duality}\label{sec:projdual1}
  Geodesics lines and geodesic planes in $\X_\Lambda$ and $\Y_\Lambda$ are obtained as the intersections of $\X_\Lambda$ and $\Y_\Lambda$ with projective lines and with projective planes in $\RP^3$.  The latter are the projections of 2d and 3d linear subspaces of $\R^4$ to $\RP^3$.
  As usual, a geodesic in $\X_\Lambda$ or $\Y_\Lambda$  is called timelike, lightlike or spacelike if its tangent vectors are timelike, lightlike or spacelike.   A geodesic plane in $\X_\Lambda$ or $\Y_\Lambda$ is called timelike, if it contains a timelike geodesic, spacelike if all of its geodesics are spacelike, and lightlike, if it contains a lightlike but no timelike geodesics.
  
  The projective duality between $\X_\Lambda$ and $\Y_\Lambda$ is a bijection between points in one space and (totally) geodesic spacelike planes in the other. For $\Lambda \neq 0$, it is induced by orthogonality with respect to the ambient bilinear form $\langle\cdot,\cdot\rangle_\Lambda$ on $\R^4$ from \eqref{eq:yform}.
  To a point $[x]\in\X_\Lambda$ it assigns the spacelike plane $x^*\subset\Y_\Lambda$ and to a 
  point  $[y]\in \Y_\Lambda$  the spacelike plane $y^*\subset \X_\Lambda$  with
  \begin{align}\label{eq:duality}
  x^*:=\Big\{[y]\in\Y_\Lambda \mid \langle x,y\rangle_\Lambda=0\Big\},
  & &
  y^*:=\Big\{[x]\in\X_\Lambda \mid \langle x,y\rangle_\Lambda=0\Big\},
  \end{align}
  where $[x],[y]\in\RP^3$ denote the equivalence classes of  $x,y\in \R^4$ in $\RP^3$. 
  This duality also induces a bijection between spacelike geodesics in $\X_\Lambda$ and in $\Y_\Lambda$. It assigns to a spacelike geodesic $g$   the intersection $p^*\cap q^*$  for any two points $[p],[q] \in g$. This intersection is a spacelike geodesic and independent of the choice of $[p]$, $[q]$ in $g$.
  
 For $\Lambda=0$ the ambient bilinear form $\langle\cdot,\cdot\rangle_\Lambda$ becomes degenerate and the duality cannot be directly interpreted in terms of orthogonality. One can, however, understand the duality for $\Lambda=0$ as a limit of the other two cases via certain blow-up procedures, see \cite{Fillastre-Seppi}. The duality between points and geodesic planes in $\X_0$ and $\Y_0$ is then given by
    \begin{align}\label{eq:duality_minkowski}
      x^*:=\Big\{[y]\in\Y_0 \mid \langle x,y\rangle_0= x_2y_2\Big\},
      & &
      y^*:=\Big\{[x]\in\X_0 \mid \langle x,y\rangle_0=x_2y_2\Big\}.
  \end{align}
  The geometric interpretation of the duality is the following. Half-pipe space $\Y_0=\bbH^2\times\R$ can be identified with the set of spacelike affine planes in Minkowski space, whose normal vector is given by a point in $\bbH^2$ and whose offset in the direction of the normal vector by a real parameter. The duality sends a point in $\Y_0$ to the associated spacelike affine plane in $\Mink^3$. Conversely, a point $x\in \Mink^3$  is dual to the graph of the map $f: \bbH^2\to \R$, $n\mapsto\langle x,n\rangle_{1,1,2}$, which  
 defines a spacelike geodesic plane in half-pipe space.

  The duality between points and geodesic planes extends to more general convex subsets $\X_\Lambda$ and $\Y_\Lambda$.
    A set in $\R P^3$ is called \emph{convex} if it is the projection of a convex cone in $\R^{4}$ that contains no non-trivial linear subspace. The projective dual of a convex set is then defined as the projection of the corresponding dual cone.
    
  Convex sets in $\X_\Lambda$ and $\Y_\Lambda$ can then be defined as the restriction of convex set in $\R P^3$ to each of these projective quadrics. The projective duality can thus be defined with respect to the duality between of convex cones in $\R^4$. We refer the reader to \cite{Fillastre-Seppi} for more details.
Geometrically, the dual of a convex set can also be characterized as the set of spacelike geodesic planes which \emph{do not intersect} the convex set.

  \subsection{Ideal points and lightlike planes}\label{subsec:ideldual}
  
  The spaces $\Y_\Lambda$ admit a natural compactification in the projective quadric model. 
  Namely, we can consider the closure of $\Y_\Lambda$ in $\RP^3$, given by
  \begin{align*}
  \overline{\Y}_\Lambda=\Rquotient{\Big\{y\in \R^4\setminus\{0\} \mid \langle y,y\rangle_\Lambda\leq 0\Big\}}{\R^\times}.
  \end{align*}
  Its boundary in  $\RP^3$ is the projective lightcone
  \begin{align*}
  \partial_\infty\Y_\Lambda=\partial \overline{\Y}_\Lambda=\Rquotient{\Big\{y\in\R^4\setminus\{0\} \mid \langle y,y\rangle_\Lambda=0\Big\}}{\R^\times}.
  \end{align*}
  This can be viewed as the asymptotic ideal boundary of $\Y_\Lambda$. It generalizes the description of the boundary $\partial\mathbb H^3$ as the set of lightlike rays in $\R^{1,0,3}$.  We will see in Section \ref{sec:idealbound}
  that the ideal boundary $\partial_\infty\Y_\Lambda$  can be identified with $\RP^1\times\RP^1$ for $\Lambda=-1$, with $\CP^1$ for $\Lambda=1$ and with $\RP^1\times \R$ for $\Lambda=0$. 
  
  The projective duality \eqref{eq:duality} between points and spacelike planes in $\X_\Lambda$ and $\Y_\Lambda$ admits a natural extension to a duality between points $[y]\in\partial_\infty\Y_\Lambda$ and lightlike planes $y^*\subset\X_\Lambda$, given again by \eqref{eq:duality}.

  \section{3d geometries via generalized complex numbers}
  
  In this section, we give a unified description of the projective quadrics $\X_\Lambda$ and $\Y_\Lambda$ in terms of $2\times 2$-matrices with entries in a commutative real algebra $\bbC_\Lambda$, whose multiplication depends on  $\Lambda$. 
  For the spaces $\Y_\Lambda$, this description was introduced in \cite{DaPhD, Da1,Da2}. For the spaces $\X_\Lambda$ similar descriptions were considered by the first author in \cite{Me, MS} and by both authors in \cite{MSc}. In Sections \ref{subsec:genc} to  \ref{sec:tangent} we summarize the results from \cite{DaPhD,Da1,Da2}  and \cite{Me,MS,MSc} and combine both descriptions in a common framework. 
  In Section \ref{sec:planepar}  we derive simple parametrizations of geodesics and geodesic planes in these spaces, which are applied in Section \ref{sec:lightplane} to investigate the geometry of lightlike geodesic planes in $\X_\Lambda$. Section \ref{sec:idealbound} summarizes Danciger's description of the ideal boundary from \cite{DaPhD, Da1, Da2}   and interprets his results in terms of Lorentzian geometry by duality with the spaces $\X_\Lambda$.

  \subsection{Generalized complex numbers}\label{subsec:genc}
  For any   $\Lambda\in\R$ we define the ring of generalized complex numbers $\bbC_\Lambda$ as the quotient of the polynomial ring in one variable $\ell$ by the ideal generated by  $\ell^2+\Lambda$
  \begin{align*}
  \bbC_\Lambda=\rquotient{\R[\ell]}{(\ell^2+\Lambda)}.
  \end{align*}
  Elements in $\bbC_\Lambda$ can thus be parametrized uniquely as $z=x+\ell y$, with real $x,y$  and $\ell^2=-\Lambda$. 
  We write $x=\Re(z)$ and $y=\Im(z)$ and refer to $x$ and $y$ as the real and imaginary parts of $z\in\bbC_\Lambda$. We also define generalized complex conjugates by $\overline z=x-\ell y$ and the modulus
    $|z|^2=z\bar z$.
  
  Note that, up to isomorphisms, $\bbC_\Lambda$ only depends on the sign of $\Lambda$. 
  We therefore restrict attention to  $\Lambda=1,0,-1$.
  For $\Lambda=1$, this yields the field $\bbC$ of complex numbers, and for  $\Lambda=0,-1$ the dual numbers and hyperbolic numbers, respectively.
  Note that for $\Lambda=0, -1$ the ring  $\bbC_\Lambda$ is not a field, as there are nontrivial zero divisors. These are real multiples of $\ell$ for $\Lambda=0$ and real multiples of   $1\pm\ell$ for $\Lambda=-1$. 
  The group   of units in $\bbC_\Lambda$ is 
  \begin{align*}
  \bbC_\Lambda^\times=\Big\{z\in\bbC_\Lambda\mid \,|z|^2=z\overline z\neq 0\Big\}.
  \end{align*}
  The real algebra $\bbC_\Lambda$   becomes a 2d Banach algebra for all values of $\Lambda$ when equipped with an appropriate norm. This allows one to consider 
   power series and analytic functions on $\bbC_\Lambda$ and on the algebras  $\mathrm{Mat}(n,\bbC_\Lambda)$ of $n\times n$ matrices with entries in $\bbC_\Lambda$. 
  In particular, any real analytic function $f:I\to\R$ on an open interval $I\subset \R$ can be extended to a unique analytic function $F:\Omega\to\bbC_\Lambda$ on an appropriate open set $I\subset \Omega\subset \bbC_\Lambda$, via
  \begin{align*}
  F(x+\ell y)
  =\begin{cases}
  \tfrac{1+\ell}{2}f(x+y) +\tfrac{1-\ell}{2}f(x-y), & \Lambda=-1,\\
  f(x+\im y), & \Lambda=1,\\
  f(x)+\ell f'(x)y, & \Lambda=0.
  \end{cases}
  \end{align*}
  The analytic continuation $F$ satisfies a generalization of the Cauchy-Riemann equations on $\Omega$
  \begin{align*}
  \frac{\partial \Re F}{\partial x}=\frac{\partial \Im F}{\partial y},\qquad\qquad \frac{\partial \Re F}{\partial y}=-\Lambda \frac{\partial \Im F}{\partial x}.
  \end{align*}
  Using the exponential map, we define generalized trigonometric functions $c_\Lambda,s_\Lambda:\R\to \R$ by
  \begin{align}\label{eq:expellsincos}
  \exp(\ell \theta)=c_\Lambda(\theta)+\ell s_\Lambda(\theta),
  \end{align}
  which yields
  \begin{align*}
  c_\Lambda(\theta)
  =\begin{cases}
  \cosh(\theta), & \Lambda=-1,\\
  \cos(\theta), & \Lambda=1,\\
  1, & \Lambda=0,
  \cr
  \end{cases}
  \qquad\qquad
  s_\Lambda(\theta)
  =\begin{cases}
  \sinh(\theta), & \Lambda=-1,\\
  \sin(\theta), & \Lambda=1,\\
  \theta, & \Lambda=0.
  \cr
  \end{cases}
  \end{align*}
  They satisfy the following generalized trigonometric identities
  \begin{align}\label{eq:trigids}
  &c_\Lambda^2(\theta)+\Lambda s_\Lambda^2(\theta)=1,
  \cr
  &c_\Lambda(\theta)c_\Lambda(\varphi)-\Lambda s_\Lambda(\theta)s_\Lambda(\varphi)=c_\Lambda(\theta+\varphi),
  \cr
  &c_\Lambda(\theta)s_\Lambda(\varphi)+s_\Lambda(\theta)c_\Lambda(\varphi)=s_\Lambda(\theta+\varphi),
  \end{align}
  and their derivatives are given by
  \begin{align}\label{eq:trigdev}
  \dot c_\Lambda(\theta)=-\Lambda s_\Lambda(\theta),\qquad\qquad \dot s_\Lambda(\theta)=c_\Lambda(\theta).
  \end{align}
  We also introduce the generalized tangent and cotangent functions
  \begin{align}\label{eq:tandef}
  t_\Lambda(\theta)=\frac{s_\Lambda(\theta)}{c_\Lambda(\theta)}=\begin{cases}
  \tanh(\theta), & \Lambda=-1,\\
  \tan(\theta), & \Lambda=1,\\
  \theta, & \Lambda=0,
  \end{cases}
  \qquad\qquad
   ct_\Lambda(\theta)=\frac{1}{t_\Lambda(\theta)},
  \end{align}
  and denote by $t_\Lambda^\inv$ and $ct_\Lambda^\inv$ their inverse functions with $t_\Lambda^\inv(r)\in(-\tfrac \pi 2, \tfrac \pi 2)$ and $ct_\Lambda^\inv(r)\in(0,\pi)$  if $\Lambda=1$.

  \subsection{A unified description of \texorpdfstring{$\X_\Lambda$}{X Lambda} and \texorpdfstring{$\Y_\Lambda$}{Y Lambda} 
  } 
  To obtain a unified description of the quadrics $\X_\Lambda$ and $\Y_\Lambda$,  we consider the ring $\Mat(2,\bbC_\Lambda)$ of  $2\times 2$-matrices with entries in $\bbC_\Lambda$. 
  This allows one to identify the orientation preserving isometry groups of the projective quadrics $\X_\Lambda$ and $\Y_\Lambda$ 
  with the projective linear group over $\bbC_\Lambda$, see \cite{DaPhD}
  \begin{align*}
  \PGL^+(2,\bbC_\Lambda)&=\rquotient{\Big\{A\in\Mat(2,\bbC_\Lambda) \mid |\det A|^2>0\Big\}}{\bbC_\Lambda^\times}.
  \end{align*}

  More explicitly, the group isomorphisms between $\PGL^+(2,\bbC_\Lambda)$ and the orientation preserving isometry groups of $\X_\Lambda$ and $\Y_\Lambda$ are given by 
  \begin{align*}
  \PGL^+(2,\bbC_\Lambda)\to 
  \begin{cases}
  \PGL^{2,+}(2,\R)
  \cr
  \PSL(2,\bbC), 
  \cr
  \PGL(2,\R)\ltimes\sl(2,\R), 
  \end{cases}
  \quad\quad
  R+\ell I\mapsto \begin{cases}
  (R+I,R-I), 
  &\Lambda=-1.
  \cr
  R+\im I, 
  &\Lambda=1,
  \cr
  (R,R^\inv I), 
  &\Lambda=0,
  \end{cases}
  \end{align*}
   where $\PGL^{2,+}(2,\R)$ consists of pairs $(A,B)\in\PGL(2,\R)\times\PGL(2,\R)$ with $\det AB>0$.
  
  The description of   the projective quadrics $\X_\Lambda$ and $\Y_\Lambda$ in terms of matrices with entries in $\bbC_\Lambda$ is obtained from a pair of involutions $\circ,\dag: \Mat(2,\bbC_\Lambda)\to\Mat(2,\bbC_\Lambda)$, given by
  \begin{align*}\left(\begin{matrix} a & b \cr c & d\end{matrix}\right)^\circ=\left(\begin{matrix} \overline d & -\overline b \cr -\overline c & \overline a\end{matrix}\right),
  \qquad\qquad
  \left(\begin{matrix} a & b \cr c & d \end{matrix}\right)^\dag=\left(\begin{matrix} \overline a & \overline c \cr \overline b& \overline d\end{matrix}\right).
  \end{align*}
  The sets of fixed points under these involutions are four-dimensional real vector spaces.  The spaces $\X_\Lambda$ and $\Y_\Lambda$ can then be realized as
  their subsets of positive determinant matrices modulo rescaling
  \begin{align}\label{eq:xmatpar}
  &\X_\Lambda=\Rquotient{\Big\{x\in\Mat(2,\bbC_\Lambda) \mid x^\circ=x, \; \det(x)>0\Big\}}{\R^\times},\\
  \label{eq:ymatpar}
  &\Y_\Lambda=\Rquotient{\Big\{y\in\Mat(2,\bbC_\Lambda) \mid y^\dag=y, \; \det(y)>0\Big\}}{\R^\times}.
  \end{align}
  
  Explicitly, the identification of the quadrics $\X_\Lambda$  from \eqref{eq:ads}, \eqref{eq:ds} and \eqref{eq:mink} with
   \eqref{eq:xmatpar}  is given by the linear map
  \begin{align}\label{eq:xmatrixrp}
  &\phi_X:\R^4\to \mathrm{Mat}(2,\bbC_\Lambda),
  \qquad\qquad
  (x_1,x_2,x_3,x_4)\mapsto \begin{pmatrix}x_2+\ell x_4 & \ell(x_3-x_1)\\
  \ell(x_3+x_1) & x_2-\ell x_4
  \end{pmatrix},
  \end{align}
  and the identification of the quadrics $\Y_\Lambda$ from \eqref{eq:y} with  \eqref{eq:ymatpar} by 
  \begin{align}\label{eq:ymatrixrp}
  &\phi_Y:\R^4\to \mathrm{Mat}(2,\bbC_\Lambda),
  \qquad\qquad
  (y_1,y_2,y_3,y_4)\mapsto \begin{pmatrix} y_1+y_3 & y_4+\ell y_2\\  y_4-\ell y_2 & y_1-y_3\end{pmatrix}.
  \end{align}
  These maps identify $\R^4$ with the set of matrices $A,B\in \mathrm{Mat}(2,\bbC_\Lambda)$ satisfying $A=A^\circ$ and $B=B^\dag$, respectively.  With these identifications, the action of the group $\PGL^+(2,\bbC_\Lambda)$ on $\X_\Lambda$ and  $\Y_\Lambda$  takes the form
  \begin{align}\label{eq:actiony}
  &\rhd: \PGL^+(2,\bbC_\Lambda)\times \X_\Lambda\to \X_\Lambda,\qquad\qquad A\rhd x=A x A^\circ,
  \\
  &\rhd: \PGL^+(2,\bbC_\Lambda)\times \Y_\Lambda\to \Y_\Lambda,\qquad\qquad B\rhd y=B y B^\dag.\nonumber
  \end{align}
  The full isometry group of  $\X_\Lambda$ and $\Y_\Lambda$ is 
  generated by $\PGL^+(2,\bbC_\Lambda)$ together with generalized complex conjugation.

  The fact that $\PGL^+(2,\bbC_\Lambda)$  acts transitively on the spaces $\X_\Lambda$ and $\Y_\Lambda$ can then be seen as a consequence of the following lemma.

  \begin{lemma}\label{lem:unitpt}
  For any point  $x\in \X_\Lambda$ and  $y\in \Y_\Lambda$, there are isometries $A,B\in \PGL^+(2,\bbC_\Lambda)$ such that $x=A\rhd\mathbb{1}=AA^\circ$ and $y=B\rhd\mathbb{1}=BB^\dag$. They can be chosen to satisfy $A^\circ=A$ and $B^\dag=B$. 
  \end{lemma}
  
  \begin{proof}
  Given a point $x\in\X_\Lambda$ we can always choose a representative $x'\in\Mat(2,\bbC_\Lambda)$ with
  \begin{align*}
  (x')^\circ=x',\qquad &\det(x')=1,\qquad \tr(x')\geq 0.
  \end{align*}
  Then the matrix 
  $A'=\mathbb{1}+x'\in\Mat(2,\bbC_\Lambda)$
  satisfies
  \begin{align*}
  \det(A')=2+\tr(x')>0,\qquad (A')^2=\det(A')x',\qquad (A')^\circ=A',
  \end{align*}
  and thus define an element in $\PGL^+(2,\bbC_\Lambda)$ with the desired properties. The proof for $y\in\Y_\Lambda$ is analogous.
  \end{proof}
  
  The stabilizers of $\mathbb{1}$ in $\X_\Lambda$ and in $\Y_\Lambda$ are given by the projective unitary matrices with respect to $\circ$ and $\dag$
  \begin{align*}
  &\Stab(\mathbb{1},\X_\Lambda)=\Big\{U\in\PGL^+(2,\bbC_\Lambda)\mid U^\circ=U^\inv\Big\}, \\
  &\Stab(\mathbb{1},\Y_\Lambda)=\Big\{V\in \PGL^+(2,\bbC_\Lambda) \mid V^\dag=V^\inv\Big\}.\nonumber
  \end{align*}
  We denote by $\PSL(2,\R)_\Lambda$ and $\PSU(2)_\Lambda$ the identity components of these groups. They are isomorphic to the groups
  \begin{align*}
  \PSL(2,\R)_\Lambda\cong\begin{cases}
  \Delta \PSL(2,\R),
  \cr
  \PSL(2,\R),
  \cr
  \PSL(2,\R)\ltimes\{0\},
  \end{cases}
  \quad\quad
  \PSU(2)_\Lambda\cong\begin{cases}
  \overline{\Delta}\PSL(2,\R),
  \cr
  \PSU(2),
  \cr
  \Un(1)\ltimes\R^2,
  \end{cases}
  \end{align*}
  for $\Lambda=-1,1,0$, respectively.
  Here, $\Delta\PSL(2,\R),\overline{\Delta}\PSL(2,\R)\subset\PGL^{2,+}(2,\R)$ stand for the images of the diagonal and the anti-diagonal embeddings of $\PSL(2,\R)$ given by 
  $\Delta: U\mapsto (U,U)$ and $\overline \Delta: V\mapsto (V,(V^\inv)^{T})$.

  \subsection{Tangent vectors}\label{sec:tangent}
  The tangent spaces $T_x\X_\Lambda$ and $T_y\Y_\Lambda$ can also be given a simple matrix description \cite{DaPhD,MS}. 
  With Lemma \ref{lem:unitpt},  points in $\X_\Lambda$ and $\Y_\Lambda$ can be parametrized as
   $x=A\rhd \mathbb{1}$ and  $y=B\rhd \mathbb{1}$, with $A^\circ=A$ and $B^\dag=B$. 
   The tangent spaces $T_x\X_\Lambda$ and $T_y \Y_\Lambda$ can then be parametrized by
  \begin{align}\label{eq:tangent}
  &T_x\X_\Lambda=A\rhd\x_\Lambda,\qquad\qquad\x_\Lambda=\Big\{X\in\Mat(2,\bbC_\Lambda)\mid X^\circ=X,\quad \tr(X)=0\Big\},\\
  &T_y\Y_\Lambda=B\rhd\y_\Lambda,\qquad\qquad\y_\Lambda=\Big\{Y\in\Mat(2,\bbC_\Lambda)\mid Y^\dag=Y,\quad \tr(Y)=0\Big\}.\nonumber
  \end{align}
  
  The induced actions of 
  $\Stab(\mathbb 1,\X_\Lambda)$
  on $\x_\Lambda$ and of $\Stab(\mathbb 1,\Y_\Lambda)$ on $\y_\Lambda$ are given by
  \begin{align*}
  &\rhd: \Stab(\mathbb 1,\X_\Lambda)\times \x_\Lambda\to \x_\Lambda, \qquad\qquad U\rhd X=U X U^\inv,
  \\
  &\rhd: \Stab(\mathbb 1, \Y_\Lambda)\times \y_\Lambda\to \y_\Lambda, \qquad\qquad V\rhd Y=V Y V^\inv.\nonumber
  \end{align*}
  
  Note that $\x_\Lambda$ and $\y_\Lambda$ are endowed with invariant bilinear forms
  \begin{align}\label{eq:detmetric}
  \langle X,X\rangle_{\x_\Lambda}=-\det(\Im X),
  \qquad\qquad
  \langle Y,Y\rangle_{\y_\Lambda}=-\det(Y).
  \end{align}
  These are unique up to real rescaling and are transported to the tangent spaces at $x=A\rhd\mathbb{1}\in\X_\Lambda$ and at $y=B\rhd\mathbb{1}\in\Y_\Lambda$ via the  $\PGL^+(2,\bbC_\Lambda)$-action. More precisely, for  $X\in \x_\Lambda$,  $Y\in \y_\Lambda$   and $A,B\in \PGL^+(2,\bbC_\Lambda)$ the metrics on the tangent spaces at  $A\rhd\mathbb 1$ and $B\rhd\mathbb 1$ are defined by
  \begin{align}\label{eq:metric}
  \langle A\rhd X, A\rhd X\rangle=\langle X,X\rangle_{\x_\Lambda},
  \qquad\qquad
  \langle B\rhd Y, B\rhd Y\rangle=\langle Y,Y\rangle_{\y_\Lambda}.
  \end{align}
  Note also that $\x_\Lambda=\ell\,\mathfrak{sl}(2,\R)=\ell\,\mathrm{Lie}\PSL(2,\R)$ and that the bilinear form $\langle\cdot,\cdot \rangle_{\x_\Lambda}$ is proportional to the Killing form on $\mathfrak{sl}(2,\R)$. This shows that the tangent space $T_x\X_\Lambda$ with the metric from  \eqref{eq:detmetric} and \eqref{eq:metric}   is isometric to 3d Minkowski space for all values of $\Lambda$. We therefore call a matrix $X\in \x_\Lambda$ timelike, lightlike or spacelike, if $\langle X,X\rangle<0$, $\langle X,X\rangle=0$ or $\langle X,X\rangle>0$, respectively. This is equivalent to the statement that the matrix $\exp(\Im X)\in\PSL(2,\R)$  is elliptic, parabolic or hyperbolic, respectively.
  
  To simplify notation later, we define  $\sigma:\x_\Lambda\to\{-1,0,1\}$ with 
  \begin{align*}
  \sigma(X)=\begin{cases}-1, & \text{if $X$ is timelike,} \cr 0, & \text{if $X$ is lightlike,} \cr 1, & \text{if $X$ is spacelike}.\end{cases}
  \end{align*}
  For each $X\in\x_\Lambda$, we denote by $\hat X\in\x_\Lambda$ its normalization, given by
  \begin{align*}
  \hat X=\begin{cases}\frac{X}{\sqrt{|\langle X,X\rangle|}}, & \text{if $X$ is timelike or spacelike,} \cr X, & \text{if $X$ is lightlike.}\end{cases}
  \end{align*}
  The bilinear form $\langle\cdot,\cdot\rangle_{\y_\Lambda}$ on  $\y_\Lambda$ has different signatures for different values of $\Lambda$. It is Lorentzian for  $\Lambda=-1$, Riemannian for $\Lambda=1$, and degenerate with signature $(0,1,2)$  for $\Lambda=0$. We define timelike, lightlike and spacelike matrices and normalization for matrices in $\y_\Lambda$ analogously.
  Note that timelike vectors in $\y_\Lambda$ arise only for $\Lambda=-1$ and lightlike ones only for $\Lambda=0, -1$.
  
  These conventions allow one to refine Lemma \ref{lem:unitpt} and to parametrize points $x\in\X_\Lambda$ and $y\in \Y_\Lambda$ in terms of exponentials of unit tangent vectors.

  \begin{lemma}\label{lem:expara} Any point $x\in\X_\Lambda$ or $y\in\Y_\Lambda$ can be expressed as
  \begin{align*}
  &x=\exp(\tfrac \theta 2 X)\rhd \mathbb 1=\Big(c_{\Lambda\sigma(X)}\big(\tfrac \theta 2\big)+ s_{\Lambda \sigma(X)}\big(\tfrac \theta 2\big) X\Big)\rhd \mathbb 1,\\
  &y=\exp(\tfrac \theta 2 Y)\rhd \mathbb 1=\Big(c_{\sigma(Y)}\big(\tfrac \theta 2\big)+ s_{ \sigma(Y)}\big(\tfrac \theta 2\big) Y\Big)\rhd \mathbb 1,
  \end{align*}
  with unit vectors $X\in\x_\Lambda$, $Y\in\y_\Lambda$, $\theta\geq 0$ and with  $\theta<2\pi$ for $\Lambda\sigma(X)<0$ or $\sigma(Y)<0$. This parametrization is unique for $x, y \neq \mathbb 1$.
  \end{lemma}
  
  \begin{proof}
  By Lemma \ref{lem:unitpt} there are matrices $A,B\in \PGL^+(2,\bbC_\Lambda)$ with $A^\circ=A$, $B^\dag=B$ such that $x=A\rhd\mathbb 1$ and $y=B\rhd\mathbb 1$. By rescaling $A$ and $B$ we can achieve $\det(A)=\det(B)=1$ and $\tr(A), \tr(B)\geq 0$. Using the parametrizations \eqref{eq:xmatrixrp}, \eqref{eq:ymatrixrp} and \eqref{eq:tangent}, we can express them as
  $$
  A=a\mathbb 1+b X,\qquad\qquad B=c\mathbb 1+d Y,
  $$ 
  with $a,b,c,d\geq 0$  and unit matrices $X\in\x_\Lambda$ and $Y\in\y_\Lambda$. The condition $\det(A)=\det(B)=1$ then read $a^2+\Lambda\sigma(X) b^2=1$ and $c^2+\sigma(Y)d^2=1$. We can thus parametrize 
  $$a=c_{\Lambda\sigma(X)}(\tfrac \theta 2), \qquad b=s_{\Lambda\sigma(X)}(\tfrac \theta 2), \qquad c=c_{\sigma(Y)}(\tfrac \theta 2),\qquad d=s_{\sigma(Y)}(\tfrac \theta 2),$$ 
  with $\theta\geq 0$ and $\theta<2\pi$ for $\Lambda\sigma(X)<0$ or $\sigma(Y)<0$. A direct matrix computation using the definition of $x_\Lambda$ and $y_\Lambda$ in \eqref{eq:tangent} then shows that these expressions for $A,B$ coincide with $\exp(\tfrac \theta 2 X)$ and $\exp(\tfrac \theta 2 Y)$.
  \end{proof}

  \begin{proposition}\label{prop:stabilisers}
  The subgroups of 
  $\Stab(\mathbb 1,\X_\Lambda)$ and $\Stab(\mathbb 1,\Y_\Lambda)$ that stabilize a spacelike or timelike vector $X\in\x_\Lambda$ or $Y\in\y_\Lambda$ are 
  \begin{align*}
  \Stab(X)&=\Rquotient{\Big\{a\mathbb{1}+b\Im \hat X \mid a,b\in\R,\, a^2\neq \sigma(X)b^2 \Big\}}{\R^\times},
  \cr
  \Stab(Y)&=\Rquotient{\Big\{a\mathbb{1}+\ell b \hat Y \mid a,b\in\R,\, a^2\neq -\Lambda\sigma(Y) b^2 \Big\}}{\R^\times}.
  \end{align*}
  \end{proposition}
  \begin{proof}

  The conditions $|\det(U)|^2>0$ and $U\rhd \mathbb 1=UU^\circ=\mathbb 1$ for an element $U\in\PSL(2,\bbC_\Lambda)$
  imply
  \begin{align*}
  U=a\mathbb{1}+b\Im\hat X_U,
  \end{align*}
  for some $a,b\in\R$ and $X_U\in\x_\Lambda$ with $a^2-\sigma(\hat X_U)b^2\neq 0$. Furthermore, the condition
  $UXU^\inv=X$ for
  a spacelike or timelike vector $X\in\x_\Lambda$ implies $X_U=X$, up to rescaling. The proof for $\Y_\Lambda$ is analogous.
  
  \end{proof}

  \subsection{Geodesics and geodesic planes}\label{sec:planepar} The description of the spaces $\X_\Lambda$ and $\Y_\Lambda$ in terms of generalized complex matrices allows one to parametrize their geodesics in terms of the matrix exponential. As the isometry group $\PGL^+(2,\bbC_\Lambda)$ acts transitively on these spaces, all geodesics are obtained from geodesics through $\mathbb 1$  via the action of the isometry group.  Geodesics through $\mathbb 1$ are obtained by exponentiating matrices in $\x_\Lambda$ and $\y_\Lambda$.

  \begin{proposition} \label{prop:geodesics} Let  $x\in\X_\Lambda$, $y\in\Y_\Lambda$ and $A,B\in \PGL^+(2,\bbC_\Lambda)$ be as in Lemma \ref{lem:unitpt}. Then for any unit tangent vector $A\rhd X\in T_{x}\X_\Lambda$ at $x=A\rhd\mathbb 1$ the geodesic $x:\R\to\X_\Lambda$ with  $x(0)=x$ and $\dot x(0)=A\rhd X$ is given by
  \begin{align}\label{eq:xpar}
  &x(t)=A\rhd\exp(t X)=A\rhd\Big(c_{\Lambda\sigma(X)}(t)\mathbb{1}+s_{\Lambda\sigma(X)}(t) X\Big),
  \end{align}
  and for any unit tangent vector $B\rhd Y\in T_y\Y_\Lambda$ at $y=B\rhd\mathbb{1}$ the geodesic $y:\R\to\Y_\Lambda$ with $y(0)=y$ and $\dot y(0)=B\rhd Y$ is given by
  \begin{align}
  \label{eq:ypar}
  &y(t)=B\rhd\exp(t Y)=B\rhd \Big(c_{-\sigma(Y)}(t)\mathbb{1}+s_{-\sigma(Y)}(t) Y\Big).
  \end{align}
  \end{proposition}
  \begin{proof} As the expressions for $x=A\rhd \mathbb 1$ and $y=B\rhd \mathbb 1$ are obtained from the ones for $x=\mathbb 1$ and $y=\mathbb 1$ via the action of the isometry group, it is sufficient to consider the cases $A=B=\mathbb 1$. 
    
  Geodesics in $\X_\Lambda$ or $\Y_\Lambda$ are obtained by projecting planes in $\R^4$.  The identifications \eqref{eq:xmatrixrp} and \eqref{eq:ymatrixrp} of $\R^4$  with the sets of hermitian matrices for $\circ$ and $\dag$  then shows that their image is contained in $\mathrm{Span}(\{\mathbb 1, X\})$ or $\mathrm{Span}(\{\mathbb 1, Y\})$ for a vector $X\in\x_\Lambda$ or $Y\in\y_\Lambda$. They are characterized uniquely by the conditions $x(0)=\mathbb 1$, $\dot x(0)=X$,   $\langle \dot x(t),\dot x(t)\rangle$ constant or $y(0)=\mathbb 1$ $\dot y(0)=Y$ and $\langle \dot y(t),\dot y(t)\rangle$ constant. The first two conditions follow directly from \eqref{eq:xpar} and \eqref{eq:ypar}, the last conditions from the identities
  \begin{align*}
  \dot x(t)=(A\exp(\tfrac{t}{2} X))\rhd X,\qquad\qquad \dot y(t)=(B\exp(\tfrac{t}{2}Y))\rhd Y,
  \end{align*}
  which are obtained using \eqref{eq:trigids} and \eqref{eq:trigdev}.
  \end{proof}

  Note that a geodesic $x:\R\to\X_\Lambda$ or $y:\R\to \Y_\Lambda$ is timelike, lightlike or spacelike, respectively, if the vectors $X\in\x_\Lambda$ or $Y\in\y_\Lambda$ from Proposition \ref{prop:geodesics} are timelike, lightlike or spacelike.  
  Equation \eqref{eq:ypar} implies that a geodesic in $\Y_\Lambda$ is closed if and only if it is timelike, which is possible only for $\Lambda=-1$.  By equation \eqref{eq:xpar} a geodesic in $\X_\Lambda$ is closed if and only if it is spacelike and $\Lambda=1$ or timelike and $\Lambda=-1$.
  
  The parameter $t\in\R$ in \eqref{eq:xpar} and \eqref{eq:ypar} can be readily identified as the arc length parameter of a spacelike or timelike geodesic. By an abuse of notation, we write $d(x,x')$ and $d(y,y')$ for the arc length of a geodesic segment with endpoints $x,x'\in \X_\Lambda$ or $y,y'\in\Y_\Lambda$. This segment is of course non-unique whenever there is a closed geodesic containing $x,x'$ or $y,y'$. In this case, any identity stated for $d(x,x')$ and $d(y,y')$ is understood to hold for all such choices.

  \begin{proposition} \label{prop:geodesdist}
  Let $x,x'\in\X_\Lambda$ and $y,y'\in\Y_\Lambda$. Then the arc lengths $d(x,x')$, $d(y,y')$ satisfy
  \begin{align*}
  |c_{\Lambda\sigma}(d(x,x'))|=\tfrac{1}{2}|\tr(\bar x'\cdot \bar x^\inv)|,\qquad\qquad |c_\sigma(d(y,y'))|=\tfrac{1}{2}|\tr(\bar y'\cdot \bar y^\inv)|,
  \end{align*}
  where $\sigma=-1,0,1$, respectively, if the geodesic segment connecting $x,x'$ or $y,y'$ is timelike, lightlike or spacelike and $\bar x,\bar x', \bar y,\bar y'$ are matrices of unit determinant representing $x,x',y,y'$. 
  
  For $\Lambda=0$  one also has
  \begin{align*}
  \sigma d(x,x')^2=-\det \Im(\bar x'-\bar x),
  \end{align*}
  where $\bar x',\bar x$  are matrices with traces of equal sign representing $x',x$.
  \end{proposition}
  
  \begin{proof} Let  $x:\R\to\X_\Lambda$ be a spacelike or timelike geodesic parametrized as in \eqref{eq:xpar} with  $x(0)=\bar x=A\rhd\mathbb 1$ and  $t\geq 0$ such that $x(t)=\bar x'$. Then the arc length between $x$ and $x'$ is $d(x,x')=t$, and from \eqref{eq:xpar} one has
  \begin{align*}
  |\tr(\bar x'\cdot \bar x^\inv)|=\Big|\tr\left(A\cdot (c_{\Lambda\sigma(X)}(t)\mathbb 1+s_{\Lambda\sigma(X)}(t) X)\cdot A^\inv\right)\Big|=2|c_{\Lambda\sigma(X)}(t)|.
  \end{align*}
  The proof for points $y,y'\in\Y_\Lambda$ is analogous.
  
  For $\Lambda=0$ and $x,x'\in\X_\Lambda$ the geodesic with $x(0)=A\rhd\mathbb 1=\bar x$ and $x(t)=\bar x'$ is given by  $x(t)=A\rhd(\mathbb{1}+  t X)$ with a unit vector $ X\in\x_\Lambda$. And we have
  \begin{align*}
  \det \Im(\bar x'-\bar x)=t^2 \det ( \Im(A\rhd X))=t^2\det (X)=t^2\sigma(X)=\sigma(X)d(x,x')^2,
  \end{align*}
  where we used that $X$ is a unit vector and that $A\rhd X=\Re(A)\rhd X=\Re(A) X \Re(A)^\inv$ for all $ X\in \x_\Lambda$ and  $A\in \PGL(2,\bbC_\Lambda)=\PGL(2,\R)\ltimes\sl(2,\R)$ if $\Lambda=0$.
  \end{proof}

   The explicit description of geodesics in Proposition \ref{prop:geodesics} also allows one to compute their stabilizer groups.

  \begin{proposition}\label{prop:stabilizer}
  For a spacelike or timelike geodesic $x:\R\to\X_\Lambda$, parametrized as in \eqref{eq:xpar}, the subgroup of $\PGL^+(2,\bbC_\Lambda)$ stabilizing $x(\R)$ and preserving its orientation is given by
  \begin{align*}
  \Stab(x(\R))=\Big\{A\exp(\tfrac{\theta}{2} X)U A^\inv\mid \theta\in\R,\, U\in\Stab(X)\Big\}.
  \end{align*}
  Similarly, for a spacelike or timelike geodesic $y:\R\to\Y_\Lambda$, parametrized as in \eqref{eq:ypar}, the subgroup of $\PGL^+(2,\bbC_\Lambda)$ stabilizing $y(\R)$ and preserving its orientation, is given by
  \begin{align*}
  \Stab(y(\R))=\Big\{B\exp(\tfrac{\theta}{2}Y)VB^\inv \mid \theta\in \R,\, V\in\Stab(Y)\Big\}.
  \end{align*}
  \end{proposition}
  
  \begin{proof} 
  As all geodesics are obtained from geodesics through $\mathbb 1$ by the action of the isometry groups, we can assume $A=B=\mathbb 1$.
  For any isometry $T\in \Stab(x(\R))$ there is $\theta\in \R$ with  $T\rhd \mathbb 1=x(\theta)$. This implies  $T=\exp(\tfrac {\theta} 2 X)U$, where $U\in \Stab (\mathbb 1)$ with $U\rhd  \R X=\R X$. Due to invariance of the bilinear form \eqref{eq:detmetric} on $\x_\Lambda$, because $x$ is spacelike or timelike and because  $T$ preserves the orientation of $x$,  we have
   $U\rhd X=X$, and the claim follows from Proposition \ref{prop:stabilizer}. The proof for geodesics in $\Y_\Lambda$ is analogous.
  \end{proof}
  
%   Expressions for the stabilizer of a lightlike geodesic in $\X_\Lambda$ or $\Y_\Lambda$ can be obtained analogously, but in that case the stabilizer group is not an abelian two-parameter subgroup, and we will not need it in the following.
  
  The parameter $\theta$ in Proposition \ref{prop:stabilizer} describes a translation along the geodesics  $x: \R\to \X_\Lambda$ and  $y: \R\to \Y_\Lambda$, which corresponds to a shift $t\mapsto t+\theta$ in the parametrization in Proposition \ref{prop:geodesics}.   It is the arc length of the geodesic segment between a point on  $x$ or $y$ and its image. 
  The parameters $a$ and $b$ that define the elements $U=a\mathbb{1}+ b\Im X\in\Stab(X)$ and $V=a\mathbb{1}+\ell bY\in\Stab(Y)$ via Proposition \ref{prop:stabilisers}  describe generalized angles between geodesic planes through $x$ and $y$. More precisely, these angles are given by
  $$\varphi=2ct_{-\sigma(X)}^{-1}\left(\frac a b\right),\qquad\qquad \varphi=2ct_{\Lambda\sigma(Y)}^\inv\left(\frac a b\right),$$
  for geodesics $x:\R\to \X_\Lambda$ and $y: \R\to \Y_\Lambda$, respectively. 
  In the first case, the parameter $\varphi$ is the rapidity of a Lorentzian boost or the angle of a rotation around the geodesic  $x:\R\to\X_\Lambda$.
  In
  hyperbolic geometry, which corresponds to $\Y_\Lambda$ for $\Lambda=1$, the parameter $\varphi$ 
  describes the angle between a plane containing the geodesic $y$ and its image. We will use the nomenclature derived from hyperbolic geometry and call $\theta$ and $\phi$ the shearing and bending parameters along $x$ and $y$, respectively. 
  
  The parametrization of geodesics in terms of the matrix exponential in Proposition \ref{prop:geodesics} also gives rise to a parametrization of the geodesic planes in $\X_\Lambda$.  As the isometry group $\PGL^+(2,\bbC_\Lambda)$ acts transitively on $\X_\Lambda$,  the geodesic planes through $x=A\rhd \mathbb{1}$ are obtained from the geodesic planes containing $\mathbb{1}$  by the action of isometries.  Using the parametrization of the geodesics in Proposition \ref{prop:geodesics} and 
  the non-degenerate bilinear form on $\x_\Lambda$ from \eqref{eq:metric}, one then obtains

  \begin{proposition}\label{prop:geodplane}
  For every point $x\in\X_\Lambda$ and tangent vector $X\in T_x\X_\Lambda$, there is a unique geodesic plane 
  $P$ with $x\in P$  such that the tangent vectors of geodesics in $P$  at $x$ span $X^\perp$.
  If we parametrize  $x=A\rhd\mathbb 1$ and $X=A\rhd N$ with $A\in \PGL^+(2,\bbC_\Lambda)$ and $N\in\x_\Lambda$, then
  \begin{align*}
  P=\Big\{ A\rhd \exp \Big( t_1X_1+t_2X_2\Big)\mid t_1,t_2\in\R\Big\}.
  \end{align*}
  for any linearly independent pair  $X_1,X_2\in N^\perp$. We call $X$ a normal vector to  $P$ based at $x$.
  \end{proposition}

  \subsection{Lightlike geodesic planes in \texorpdfstring{$\X_\Lambda$}{X Lambda}}\label{sec:lightplane}
  In this section, we derive some elementary properties of lightlike planes in $\X_\Lambda$ that will be identified as the duals of certain statements about the ideal boundary $\partial_\infty \Y_\Lambda$ in the next section. 
   Recall that a geodesic plane in $\X_\Lambda$ is called lightlike, if it contains a lightlike geodesic, but no timelike geodesics. This is equivalent to its normal vector from Proposition \ref{prop:geodplane} being lightlike.

It follows directly that  two distinct  lightlike  planes in $\X_\Lambda$ that intersect always intersect in a spacelike geodesic.  Conversely, for any spacelike geodesic in  $\X_\Lambda$, there is a unique  pair of lightlike planes that intersect in this geodesic. In the following we often need an explicit parametrization of his intersection geodesic.

  \begin{lemma}\label{lem:lplaneint}
If two distinct lightlike planes $P_1,P_2$ in $\X_\Lambda$ intersect, then for any point $x\in P_1\cap P_2$,  there is an isometry that sends $x$ to $\mathbb 1$, their intersection to the spacelike geodesic
      \begin{align}\label{eq:xmatrixe1}
&g(t)=\exp(tX)\qquad   X=\ell\left(\begin{matrix} 1 & 0 \cr 0 & -1\end{matrix}\right)
  \end{align}
  and their normal vectors in $x$ to
    \begin{align}\label{eq:normalvec}
  N_1=\ell \left(\begin{matrix} 0 & 0 \cr 1 & 0 \end{matrix}\right),\qquad\qquad N_2=\ell \left(\begin{matrix} 0 & -1 \cr 0 & 0 \end{matrix}\right).
  \end{align}
 
  \end{lemma}
  \begin{proof}
By applying isometries, we can assume $x=\mathbb{1}$.
   The action of 
  $\PSL(2,\R)_\Lambda \subset \Stab(\mathbb 1,\X_\Lambda)$ on $\x_\Lambda=\ell\mathfrak{sl}(2,\R)$ then coincides with the action of $\PSL(2,\R)$ on Minkowski space, and the action on the normal vectors of these planes with the $\PSL(2,\R)$-action on the set of lightlike rays in 3d Minkowski space. 
  This can be identified with the $\PSL(2,\R)$-action on $\partial \mathbb H^2$, which is known to be 3-transitive.
 Hence, there is an isometry in $\PSL(2,\mathbb R)$  that sends 
  the normal vectors of the planes to \eqref{eq:normalvec}. 
  Then we have $N_1^\bot\cap N_2^\bot=\R X$ with $X$  unique up to real rescaling and given by
\eqref{eq:xmatrixe1}. 
  \end{proof}

An analogous parametrization exists  for triples of lightlike planes that intersect in a common point. 
 In this case, the 3-transitivity of the $\PSL(2,\mathbb R)$-action on $\partial \mathbb H^2$ implies uniqueness up to permutations.

  \begin{lemma}\label{cor:planeintersect}
  If three distinct lightlike planes in $\X_\Lambda$ intersect in a  common point $x$, then there is an isometry that sends $x$ to $\mathbb{1}$ and their normal vectors in $x$ to
  \begin{align*}
  N_1=\ell\left(\begin{matrix} 0 & 0 \cr 1 & 0 \end{matrix}\right),\qquad N_2=\ell\left(\begin{matrix} 0 & -1 \cr 0 & 0 \end{matrix}\right),\qquad N_3=\ell\left(\begin{matrix} 1 & -1 \cr 1 & -1 \end{matrix}\right).
  \end{align*}
  This isometry is unique up to isometries permuting the three planes.
  \end{lemma}
  
  \begin{proof}  By applying isometries, we can assume that the intersection point of these planes is $\mathbb 1$.  
  The action of 
  $\PSL(2,\R)_\Lambda \subset \Stab(\mathbb 1,\X_\Lambda)$ on $\x_\Lambda=\ell\mathfrak{sl}(2,\R)$ then coincides with the action of $\PSL(2,\R)$ on Minkowski space, and the action on the normal vectors of these planes on the $\PSL(2,\R)$-action on the set of lightlike rays in 3d Minkowski space. This can be identified with the $\PSL(2,\R)$-action on  $\partial \bbH^2$, which is known to be 3-transitive.
  \end{proof}

  \subsection{The ideal boundary of \texorpdfstring{$\Y_\Lambda$}{Y Lambda}} 
  \label{sec:idealbound}
   Under the duality between $\X_\Lambda$ and $\Y_\Lambda$ from Sections \ref{sec:projdual1} and  \ref{subsec:ideldual}, lightlike geodesic planes in $\X_\Lambda$ are dual to points on the ideal boundary of $\Y_\Lambda$. We thus summarize the properties of the ideal boundary $\partial_\infty \Y_\Lambda$ from \cite{DaPhD,Da1,Da2}. To make the paper self-contained, and because details will be needed in the following, we also include proofs, adapted from \cite{Da2}. We also 
   point out their duality with results on lightlike planes and show that in some cases this provides an additional geometric interpretation.
  
  In the matrix parametrization of $\Y_\Lambda$, the ideal boundary $\partial_\infty\Y_\Lambda$  becomes the set of rank 1 matrices modulo real rescaling
  \begin{align*}
  \partial_\infty\Y_\Lambda
  &=\Rquotient{\Big\{vv^\dag\in\Mat(2,\bbC_\Lambda) \mid v\in\bbC_\Lambda^2,\; vv^\dag\neq 0 \Big\}}{\R^\times}.
  \end{align*}
  This identifies $\partial_\infty\Y_\Lambda$ with the generalized complex projective line
  \begin{align}\label{eq:cproj}\bbC_\Lambda\mathrm{P}^1=\Rquotient{\Big\{v\in\bbC_\Lambda^2\mid vv^\dag\neq 0\Big\}}{\bbC_\Lambda^\times}=\begin{cases}
  \R \rmP^1\times \R \rmP^1,&\Lambda=-1,
  \cr
  \bbC \rmP^1, &\Lambda=1,
  \cr
  \R \rmP^1\times\R, &\Lambda=0.
  \end{cases}
  \end{align}
It should be mentioned, however, that the topology induced by this identification does not coincide with the one induced by $\RP^3$ for $\Lambda=0$. 

For $\Lambda=1$ this holds by definition. For $\Lambda=-1$ the identification is given by the map
  \begin{align}
 \R \rmP^1\times \R \rmP^1\to \bbC_\Lambda\mathrm{P}^1,\qquad  \left(\left[\begin{matrix} u\\ v\end{matrix}\right] , \left[\begin{matrix} x\\ y\end{matrix}\right]\right)\mapsto\left[\begin{matrix}u+x+\ell(u-x)\\ v+y+\ell(v-y)\end{matrix}\right]\qquad u,v,x,y\in\R,
  \end{align}
  and for $\Lambda=0$ by  the map
    \begin{align}
 \R \rmP^1\times\R\to  \bbC_\Lambda\mathrm{P}^1,\qquad   \left( \left[\begin{matrix} x\\ y\end{matrix}\right], u\right)\mapsto\left[\begin{matrix}x+\ell yu\\ y+\ell xu\end{matrix}\right]\qquad u,x,y\in\R.
  \end{align}
  The action \eqref{eq:actiony} of 
   $\PGL^+(2,\bbC_\Lambda)$
  on $\Y_\Lambda$ extends to a
   $\PGL^+(2,\bbC_\Lambda)$-action on $\partial_\infty \Y_\Lambda$ 
  \begin{align*}
  \rhd: \PGL^+(2,\bbC_\Lambda)\times\partial_\infty \Y_\Lambda\to\partial_\infty \Y_\Lambda,\qquad\qquad B\rhd Y=BYB^\dag.
  \end{align*}
  Under the identification of $\partial_\infty \Y_\Lambda$ with $\bbC_\Lambda \rmP^1$, this action becomes the standard action of
   $\PGL^+(2,\bbC_\Lambda)$
   on $\bbC_\Lambda\rmP^1$ via projective transformations
  \begin{align*}
  \rhd:\PGL^+(2,\bbC_\Lambda)\times\bbC_\Lambda\rmP^1\to\bbC_\Lambda\rmP^1,\qquad\qquad B\rhd [v]=[B\cdot v].
  \end{align*}
  Note that for $\Lambda=1$ this coincides with the action of M\"obius transformations on the Riemann sphere $\bbC\rmP^1=\partial_\infty \bbH^3$. In this case, the condition $vv^\dag\neq 0$ in \eqref{eq:cproj} simply states that $v\neq 0$. By rescaling representatives of points in  $\bbC\rmP^1$ such that their second entry is $1$, one obtains: 
  \begin{align}\label{eq:mts}
  \begin{pmatrix} a & b\\ c & d\end{pmatrix}\rhd \left[ \begin{matrix}z \\ 1\end{matrix}\right]=\left[ \begin{matrix} az+b \\ cz+d\end{matrix}\right]=\left[ \begin{matrix}\frac{az+b}{cz+d} \\ 1\end{matrix}\right]\qquad\text{with}\qquad  \left[ \begin{matrix}\infty \\ 1\end{matrix}\right]=\left[ \begin{matrix}1 \\ 0\end{matrix}\right].
  \end{align}
  In the following,   the action of  $\PGL^+(2,\bbC_\Lambda)$
   on $\bbC_\Lambda \rmP^1$  is often described with respect to three fixed reference points  $v_1,v_2,v_3\in \bbC_\Lambda P^1$
   \begin{align}\label{eq:refvecs}
   v_1=\left[ \begin{matrix}1 \\ 0\end{matrix}\right]=\infty,\qquad v_2=\left[ \begin{matrix}0 \\ 1\end{matrix}\right]=0,\qquad v_3=\left[ \begin{matrix}1 \\ 1\end{matrix}\right]=1, 
    \end{align}
  which correspond to the points  $\infty,0,1\in\CP^1=\bbC\cup\{\infty\}$ for $\Lambda=1$. We also write $v_1=\infty$, $v_2=0$ and $v_3=1$ to denote the points
  $v_1,v_2,v_3\in\bbC_\Lambda\rmP^1$ in \eqref{eq:refvecs} for $\Lambda\neq 1$.

  The subgroup of $\PGL^+(2,\bbC_\Lambda)$ that permutes $v_1,v_2,v_3$
  is the group of order six generated by the classes of 
  \begin{align}\label{eq:permutmatrix}
  T=\begin{pmatrix} 0 & 1\\ -1 & 1\end{pmatrix},\qquad\qquad I=\begin{pmatrix}0 & 1\\ 1 & 0\end{pmatrix}.
  \end{align}
  It permutes the points  $v_1,v_2,v_3$ 
  according to
  \begin{align*}
  &T: (v_1,v_2,v_3)\mapsto (v_2,v_3,v_1), \qquad\qquad I: (v_1,v_2,v_3)\mapsto (v_2,v_1,v_3).
  \end{align*}
  
  Spacelike geodesics in $\Y_\Lambda$  have two endpoints in  $\partial_\infty\Y_\Lambda$,  obtained from their parametrization  \eqref{eq:ypar} 
  as the limits $t\to\pm \infty$. These endpoints are the duals of the two unique lightlike planes  that intersect in the dual spacelike geodesic in $\X_\Lambda$.
  The action of the isometry group $\PGL^+(2,\bbC_\Lambda)$ on $\partial_\infty \Y_\Lambda$ allows one to map these endpoints to fixed reference points, namely the points $v_1v_1^\dag$ and $v_2v_2^\dag$ for $v_1$, $v_2$ given in \eqref{eq:refvecs}.  This is dual to the statement in  Lemma \ref{lem:lplaneint} that by acting with isometries, one can transform the normal vectors of the lightlike planes into   \eqref{eq:normalvec}.

  \begin{lemma}\label{lem:2ptbound} Let $y_+,y_-\in\partial_\infty\Y_\Lambda$ be endpoints of a spacelike geodesic in $\Y_\Lambda$. Then there is an isometry  $B\in \PGL^+(2,\bbC_\Lambda)$ such that 
  \begin{align}\label{eq:2ptgoedspl}
  B\rhd y_+=v_1v_1^\dag=\left(\begin{matrix} 1 & 0\\ 0 & 0\end{matrix}\right),\qquad\qquad B\rhd y_-=v_2v_2^\dag=\left(\begin{matrix}0 & 0\\ 0 & 1\end{matrix}\right).
  \end{align}
  \end{lemma}
  
  \begin{proof}Using  \eqref{eq:ypar} we can parametrize any spacelike geodesic $y$ in  $\Y_\Lambda$ as
  \begin{align*}
  y(t)=A\rhd\Big(\cosh(t)\mathbb{1}+\sinh(t) Y\Big),
  \end{align*}
  with $A\in\PGL^+(2,\bbC_\Lambda)$ and a spacelike unit matrix $Y\in\y_\Lambda$. Any normalized spacelike matrix in $\y_\Lambda$  can be written as
  \begin{align*}
  Y=\left(\begin{matrix}a & b+\ell c \cr b-\ell c & -a\end{matrix}\right),\qquad \text{with}\qquad \langle Y,Y\rangle_{\y_\Lambda}=a^2+b^2+\Lambda c^2=1.
  \end{align*}
  The endpoints of the geodesic $y$ are then represented by the matrices
  \begin{align}\label{eq:endpoints}
  y_\pm=y(\pm\infty)=A\rhd(\mathbb{1}\pm Y)\in\partial_\infty\Y_\Lambda.
  \end{align}
  Using the identification of the boundary $\partial_\infty \Y_\Lambda$ with the complex projective line $\bbC_\Lambda\rmP^1$ from \eqref{eq:cproj},  one can parametrize the endpoints as $y_\pm= v_\pm\cdot v_\pm^\dag$ with 
   \begin{align*}
   &v_+=A\cdot\left[\begin{matrix} 1+a \cr b-\ell c\end{matrix}\right], & &v_-=A\cdot\left[\begin{matrix}-b-\ell c \cr 1+a \end{matrix}\right], & &\text{ for } a\neq -1,\nonumber\\
   & v_+=A\cdot\left[\begin{matrix} b+\ell c \cr 1-a\end{matrix}\right], & &v_-=A\cdot\left[\begin{matrix} 1-a\cr -b+\ell c\end{matrix}\right],& &\text{ for } a\neq 1.
   \end{align*}
  A direct computation then shows that 
  \eqref{eq:2ptgoedspl} is satisfied for the projective matrices
  \begin{align*}
  &B=\left(\begin{matrix} 1+a & b+\ell c \cr -b+\ell c & 1+a\end{matrix}\right) A^\inv,\qquad\qquad
  B=\left(\begin{matrix} b+\ell c & 1-a \cr 1-a & -b+\ell c\end{matrix}\right) A^\inv.
  \end{align*}
  for $a\neq -1$ and $a\neq 1$, respectively.
  \end{proof}
  
  \medskip
  It is shown in \cite[Proposition 2]{Da2}, see also the remark after \cite[Proposition 3]{Da2}, that this result extends to triples of points in $\partial_\infty \Y_\Lambda$, provided that they 
are contained in a common spacelike plane.
  In this case one can take the three reference points $v_1=\infty,v_2=0,v_3=1$ in \eqref{eq:refvecs}.

  \begin{proposition}[{\cite[Proposition 2]{Da2}}] \label{lem3pt}Let $y_1, y_2, y_3\in \partial_\infty\Y_\Lambda$ be distinct points on a common spacelike plane. Then there is a unique isometry $B\in \PGL^+(2,\bbC_\Lambda)$  such that
   $B\rhd y_1=\infty$, $B\rhd y_2=0$ and $B\rhd y_3=1$.
  \end{proposition}
  \begin{proof}
  By Lemma \ref{lem:2ptbound}, one can assume that $y_1=v_1v_1^\dag$ and $y_2=v_2v_2^\dag$. As $y_3$ is connected to $y_1$ and  $y_2$ by spacelike geodesics, by \eqref{eq:endpoints} there are isometries
  $A_i\in \PGL^+(2,\bbC_\Lambda)$ and vectors $Y_i\in \y_\Lambda$  for  $i=1,2$ such that
  \begin{align}\label{eq:matrixboundarypar}
  A_i(\mathbb{1}-Y_i)A_i^\dag= v_iv_i^\dag,\qquad\qquad A_i(\mathbb{1}+Y_i) A_i^\dag= y_3.
  \end{align}
  Using the identification of $\partial_\infty \Y_\Lambda$ with $\bbC_\Lambda\rmP^1$, we can parametrize $y_3=w_3w_3^\dag$ with $w_3\in\bbC_\Lambda\rmP^1$.
  The condition $|\det(A_i)|^2>0$ together with \eqref{eq:matrixboundarypar} then implies that both entries of $w_3$ are units in $\bbC_\Lambda$, and by rescaling it, we can achieve that its second entry is 1 and its first entry is a unit $z\in\bbC_\Lambda^\times$, as in \eqref{eq:mts}. The condition that $y_1,y_2, y_3$ lie on a common spacelike plane  implies $|z|^2>0$ and that
  \begin{align*}
  B=\left(\begin{array}{cc} 1 & 0\\ 0 & z\end{array}\right)\in \PGL^+(2,\bbC_\Lambda)
  \end{align*}
  is an isometry with $B\rhd v_1=v_1$, $B\rhd v_2=v_2$ and $B\rhd w_3=v_3$.
  \end{proof}
  
  Note that for $\Lambda=1$  Proposition \ref{lem3pt}  is the well-known 3-transitivity of the action of $\PGL(2,\bbC)$ on the Riemann sphere $\CP^1$. However, for $\Lambda=0$ and $\Lambda=-1$ the action of $\PGL^+(2,\bbC_\Lambda)$ on $\bbC_\Lambda \rmP^1$ is in general not 3-transitive, even if one allows for permutations of the three points. In particular, the proof of Proposition \ref{lem3pt} shows that an element of $\PGL^+(2,\bbC_\Lambda)$ that stabilizes or exchanges $v_1=\infty$ and $v_2=0$ cannot map a general point $v\in \bbC_\Lambda \rmP^1$ to $v_3=1$.
  
Proposition \ref{lem3pt} can be viewed as the dual of Lemma \ref{cor:planeintersect}.
The dual of the spacelike plane in $\Y_\Lambda$ containing the points $y_1,y_2,y_3\in \partial_\infty\Y_\Lambda$ is a point in $\X_\Lambda$ that lies on the dual planes to $y_1,y_2, y_3$ and hence in their intersection. The normal vectors of the lightlike planes in Lemma \ref{cor:planeintersect} are thus given by the points  $y_1,y_2, y_3$ in Proposition \ref{lem3pt}.
  
  Given four distinct points in $\partial_\infty \Y_\Lambda$ such that any three of them lie on a common spacelike plane,
  one can apply an isometry to send three of them to the points $v_1v_1^\dag,v_2v_2^\dag,v_3v_3^\dag$, as in Proposition \ref{lem3pt}. As the fourth point is on a spacelike plane through  $v_1v_1^\dag$ and $v_2v_2^\dag$, it is represented by an element  $v_4\in\bbC_\Lambda\rmP^1$ whose entries are units in $\bbC_\Lambda$ by the proof of Proposition \ref{lem3pt}. Rescaling this element, one obtains
  \begin{align}\label{eq:yparamcp}
  v_4=\left[\begin{matrix} z\\ 1\end{matrix}\right],\qquad\text{with}\qquad z\in\bbC_\Lambda^\times\setminus\{1\}.
  \end{align} 
  Hence, up to isometries, the four points are characterized uniquely, by an element in $\bbC^\times_\Lambda\setminus\{1\}$,  the shape parameter introduced in \cite[Section 3.1]{Da2}, which can be viewed as a generalized cross-ratio.

  \begin{definition} \label{def:crossrat}Let $y_1,y_2,y_3,y_4$ be four distinct points on $\partial_\infty \Y_\Lambda$ such that any three of them lie on a spacelike plane. Let
   $B\in\PGL^+(2,\bbC_\Lambda)$ be an isometry such that  $B\rhd y_i=v_iv_i^\dag$ for $i=1,2,3$ 
   and $B\rhd y_4$ is parametrized as in \eqref{eq:yparamcp}.
   Then their cross-ratio is 
  \begin{align*}
  \mathrm{cr}(y_1,y_2, y_3,y_4)=\mathrm{cr}(\infty, 0, 1, z)=z\in\bbC_\Lambda^\times\setminus\{1\}.
  \end{align*}
  \end{definition}
  Note that the orbit of the cross-ratio $z=\mathrm{cr}(\infty, 0,1,z)$ under the action of the subgroup \eqref{eq:permutmatrix} of $\PGL^+(2,\bbC_\Lambda)$ permuting $v_1,v_2,v_3$
  is given by
  \begin{align*}
  &z, & &\frac 1 {1-z}, & &\frac{z-1} z, & &\frac 1 z, & &1-z,& &\frac z {z-1}.
  \end{align*}
  These are the familiar expressions for the transformation of a cross-ratio in $\bbC\rmP^1$ under the subgroup of M\"obius transformations that permute $\infty, 0,1$.
  Indeed, for  $\Lambda=1$, any point $y\in \bbC\rmP^1$ can be parametrized as in \eqref{eq:yparamcp} and the cross-ratio coincides with the usual cross-ratio on $\bbC\rmP^1$ defined by
  \begin{align}\label{eq:crossdef}
  \mathrm{cr}(z_1,z_2,z_3,z_4)=\frac{(z_3-z_1)(z_4-z_2)}{(z_3-z_2)(z_4-z_1)}.
  \end{align}
  This is a consequence of formula \eqref{eq:mts} for the $\PGL(2,\bbC)$-action on $\bbC\rmP^1$ and the invariance of the cross-ratio under isometries.
  Note, however, that for $\Lambda=0$ and $\Lambda=-1$ the cross-ratio cannot defined globally by  \eqref{eq:crossdef}, since 
  $z_3-z_2$ or $z_4-z_1$ need not be  units in $\bbC_\Lambda$.
  
  We remark that cross-ratios for $\Lambda=-1$ can be viewed as a pair of real cross-ratios on $\R P^1$
  \begin{align*}
    \mathrm{cr}(z) = \tfrac{1+\ell}{2}\mathrm{cr}(u) +\tfrac{1-\ell}{2}\mathrm{cr}(v)
    &&\text{ for }
    z =(z_1,z_2,z_3,z_4)= \tfrac{1+\ell}{2}u+\tfrac{1-\ell}{2}v.
  \end{align*}
For $\Lambda=0$, we have a real cross-ratio on $\R P^1$ together with an infinitesimal cross-ratio 
  \begin{align*}
    \mathrm{cr}(z) = \mathrm{cr}(x) + \ell d_x\mathrm{cr}(y)
    &&
    \text{ for }
    z =(z_1,z_2,z_3,z_4)= x + \ell y.
  \end{align*}

  \section{Lightlike and ideal tetrahedra}
  
  In this section we investigate the geometric properties of tetrahedra with lightlike faces in $\X_\Lambda$ and their duals in $\Y_\Lambda$. We then show that the latter are precisely the generalized ideal tetrahedra introduced by Danciger in \cite{Da2}.  
  
In the following, we denote by $x_i$ and $y_i$ the vertices of tetrahedra in $\X_\Lambda$ and $\Y_\Lambda$, respectively, and by $x_{ij}$ or $y_{ij}$ the geodesic through the vertices $x_i,x_j$ or $y_i,y_j$. In both cases, we write $e_{ij}$ for the edge of the tetrahedron through the vertices $x_i, x_j$ or $y_i,y_j$, the geodesic segment of $x_{ij}$ or $y_{ij}$ that is part of the tetrahedron.

  \subsection{Lightlike tetrahedra} 
  
  We start by considering tetrahedra in $\X_\Lambda$ whose faces are all contained in lightlike planes. We will also require that these tetrahedra are  (i) convex, i.~e.~obtained as projections of convex cones in $\R^4$, (ii) non-degenerate, i.~e.~not contained in a single geodesic plane, and (iii) that their \emph{internal} geodesics at each vertex, the geodesics that intersect the interior of the tetrahedron, are all spacelike. The last condition is relevant mainly for $\Lambda=1$.
  
  \medskip
  \begin{definition}\label{def:lightlike}
  A lightlike tetrahedron in $\X_\Lambda$ is a non-degenerate convex geodesic 3-simplex in $\X_\Lambda$ with lightlike faces such that all internal geodesics starting at its vertices are spacelike.
  \end{definition}
  
  Note that this definition implies with Lemma \ref{lem:lplaneint} that all edges of a lightlike tetrahedron are spacelike geodesic segments. The two faces containing an edge of a lightlike tetrahedron then lie on the two unique lightlike planes that intersect along this spacelike geodesic. Each vertex is the unique intersection point of the three lightlike planes containing the adjacent faces. 
  
  By applying isometries we can relate any lightlike tetrahedron to one in standard position. By this, we mean a lightlike tetrahedron with one of its vertices at $x=\mathbb 1$ and the three lightlike normal vectors at this vertex given as in Lemma \ref{cor:planeintersect}. The vertices of the
  lightlike tetrahedron can then characterized uniquely by its fourth lightlike normal vector, up to rescaling, and hence by a pair of real parameters.

  \begin{proposition}\label{prop:ltetrahedron}
  Let $L$ be a lightlike tetrahedron in $\X_\Lambda$ with vertices $x_1, x_2, x_3, x_4$. Then there is a unique isometry $A\in \PGL^+(2,\bbC_\Lambda)$ and parameters $\alpha,\beta,\gamma\in\R$ with $\alpha+\beta+\gamma=0$, such that
  \begin{align}
  &A\rhd x_1=\left(\begin{matrix} e^{\ell\alpha} & -2\ell s_\Lambda(\alpha) \cr 0 & e^{-\ell\alpha}\end{matrix}\right), 
  & &A\rhd x_2=\left(\begin{matrix} e^{\ell\beta} & 0 \cr 2\ell s_\Lambda(\beta) &  e^{-\ell\beta} \end{matrix}\right),
  \cr
  &A\rhd x_3=\left(\begin{matrix} e^{-\ell\gamma} & 0 \cr 0& e^{\ell\gamma} \end{matrix}\right), 
  & &A\rhd x_4=\left(\begin{matrix} 1 & 0 \cr 0& 1\end{matrix}\right). \label{eq:ltet}
  \end{align}
  For $\Lambda=1$ one can choose $0<|\alpha|,|\beta|,|\gamma|<\pi$.
  \end{proposition}
  
  \begin{proof} 
  Let $A_i\in \PGL^+(2,\bbC_\Lambda)$ an isometry with $A_i^\circ=A_i$ and $A_i\rhd\mathbb{1}=x_i$, as in Lemma \ref{lem:unitpt}.
  Denote by $A_i\rhd N_{ij}$ the normal vector of the face $f_j$ at the vertex  $x_i$ from Proposition \ref{prop:geodplane}.
  Then by Lemma \ref{cor:planeintersect} we can assume that $x_4=\mathbb{1}$ and 
  \begin{align}\label{eq:nexpr}
  N_{41}=\ell\left(\begin{matrix} 0 & 0 \cr 1 & 0 \end{matrix}\right),
  \qquad
  N_{42}=\ell\left(\begin{matrix} 0 & -1 \cr 0 & 0 \end{matrix}\right),
  \qquad
  N_{43}=\ell\left(\begin{matrix} 1 & -1 \cr 1 & -1 \end{matrix}\right).
  \end{align}
  Denote by $x_{ij}$  a spacelike geodesic through $x_i$ and $x_j$ with $x_{ij}(0)=x_i$. Then, by Proposition \ref{prop:geodesics}, the geodesic  $x_{ij}$ can be parametrized as
  \begin{align}\label{eq:xijparam}
  x_{ij}(t)=A_i\rhd\exp(t X_{ij}),
  \end{align}
  where $X_{ij}\in\x_\Lambda$ is a spacelike unit vector, unique up to a sign, that is orthogonal to both $N_{ik}$ and $N_{il}$ with respect to the bilinear form \eqref{eq:detmetric}  for distinct $i,j,k,l\in\{1,2,3,4\}$. By Lemma \ref{lem:expara} the remaining vertices can be expressed as
  \begin{align}\label{eq:xexps}
  x_i=A_i\rhd\mathbb 1=\exp(\alpha_i X_{4i}),
  \end{align}
  where $i=1,2,3$, $\alpha_i\in\R$.
  
  With  \eqref{eq:nexpr}  and expression \eqref{eq:detmetric} for the bilinear form on $\x_\Lambda$, one computes
  \begin{align}\label{eq:xijsimple}
  X_{41}
  =\ell\left(\begin{matrix} 1 & -2 \cr 0 & -1\end{matrix}\right),
  \qquad
  X_{42}
  =\ell\left(\begin{matrix} 1 & 0 \cr 2 & -1\end{matrix}\right),
  \qquad
  X_{43}
  =\ell\left(\begin{matrix} -1 & 0 \cr 0 & 1\end{matrix}\right).
  \end{align}
  Inserting these matrices in formula \eqref{eq:xexps} and computing the exponential with formula \eqref{eq:xpar}, one finds that $x_1,x_2,x_3$ are indeed given by the matrices in \eqref{eq:ltet}, if $\alpha=\alpha_1$, $\beta=\alpha_2$ and the parameters $\alpha_i$ satisfying $\alpha_1+\alpha_2+\alpha_3=0$ ($\bmod{\pi}$ for $\Lambda=1$).
  
  To obtain the relation between these parameters we now compute the remaining vectors $X_{ij}$ and $N_{ij}$. For the former, note that  \eqref{eq:xijparam} and \eqref{eq:xexps} imply
  \begin{align*}
  \exp(t_{ij} X_{ij}) =\exp(-\tfrac{\alpha_i} 2 X_{4i})\cdot \exp(\alpha_j X_{4j}) \cdot \exp(-\tfrac{\alpha_i} 2 X_{4i}),
  \end{align*}
  where $t_{ij}\in\R$ is given by the condition $x_j=x_{ij}(t_{ij})$. Using this identity with expression \eqref{eq:xpar} for the exponential and the identities
   \begin{align}\label{eq:xijtriple}
   \Im(X_{4i})\Im(X_{4j}) \Im(X_{4i})=-2 \Im(X_{4i})- \Im(X_{4j}),
   \end{align} which follow from \eqref{eq:xijsimple}, one obtains 
  \begin{align}\label{eq:xijhelp}
  X_{ij}=X_{4i}-\frac{s_\Lambda(\alpha_j)}{s_\Lambda(\alpha_i+\alpha_j)}(X_{4i}+X_{4j}),\qquad\qquad t_{ij}=-\alpha_i-\alpha_j,
  \end{align}
   for all distinct $i,j\in\{1,2,3\}$. Using again relation  \eqref{eq:xijtriple}, expression \eqref{eq:xexps} for the matrices $A_i$ and the parametrization \eqref{eq:xpar} of the matrix exponential implies for all distinct $i,j\in \{1,2,3,4\}$
  \begin{align}\label{eq:relationij}
  X_{ji}=-A_jA_i^\inv X_{ij}A_iA_j^\inv.
  \end{align}
  The matrices $N_{ij}\in\x_\Lambda$ can then be computed from the condition that $N_{ij}$ is orthogonal to $X_{ik}$ for all distinct $i,j,k$, and normalized such that $\langle N_{ij},X_{ij}\rangle_{\x_\Lambda}=-1$. Note that this last condition is also satisfied by the matrices $N_{4i}$ and $X_{4i}$ from \eqref{eq:nexpr} and \eqref{eq:xijsimple}. A direct computation with expression \eqref{eq:detmetric} for the bilinear form on $\x_\Lambda$ shows that $\langle X_{4i}, X_{4i}\rangle_{\x_\Lambda}=1$ and $\langle X_{4i}, X_{4j}\rangle_{\x_\Lambda}=-1$ for distinct $i,j\in\{1,2,3\}$. Equations \eqref{eq:xexps} and \eqref{eq:relationij} imply $X_{4i}=-X_{i4}$. Together with \eqref{eq:xijhelp}, these identities imply that
  \begin{align}\label{eq:nmat}
  &N_{ij}=-\frac{s_\Lambda(\alpha_i+\alpha_j)}{s_\Lambda(\alpha_j)}N_{4j},
  \cr
  &N_{i4}=X_{4i}-\frac{s_\Lambda(\alpha_i+\alpha_j)}{s_\Lambda(\alpha_j)}N_{4j}-\frac{s_\Lambda(\alpha_i+\alpha_k)}{s_\Lambda(\alpha_k)}N_{4k},
  \end{align}
  for all distinct $i,j,k\in\{1,2,3\}$.  A short computation using \eqref{eq:xijsimple} and  \eqref{eq:detmetric}  finally shows that they are all lightlike 
  if and only if $\alpha_1+\alpha_2+\alpha_3=0$ ($\bmod\pi$ for $\Lambda=1$).
  \end{proof}
  
  By applying isometries to a lightlike tetrahedron in $\X_\Lambda$, we may assume that its vertices are in the standard position given in Proposition \ref{prop:ltetrahedron}. Then, the group of isometries which fixes the vertex $x_4=\mathbb 1$ and permutes the  lightlike planes intersecting at this vertex is precisely the subgroup of $\PGL^+(2,\bbC_\Lambda)$ that permutes the reference points
  $v_1=\infty,v_2=0,v_3=1\in\bbC_\Lambda\rmP^1$ in \eqref{eq:refvecs}.

  \begin{corollary}
  For a lightlike tetrahedron with vertices as in Proposition \ref{prop:ltetrahedron} the isometry $T$ in \eqref{eq:permutmatrix} fixes $A\rhd x_4$ and cyclically permutes the lightlike vectors $N_{41},N_{42},N_{43}$ in \eqref{eq:nexpr} and the spacelike vectors $X_{41},X_{42},X_{43}$ in \eqref{eq:xijsimple}. The isometry $I$ in \eqref{eq:permutmatrix} fixes $A\rhd x_4$, $N_{43}$ and $X_{43}$, exchanges $N_{41}$ and $N_{42}$ and $X_{41}$ and $X_{42}$ and changes the signs of  $N_{41}$, $N_{42}$, $N_{43}$ and $X_{41}$, $X_{42}$ and $X_{43}$.
  \end{corollary}
  
  Using these symmetries we may always choose two of the parameters $\alpha,\beta,\gamma$ in Proposition \ref{prop:ltetrahedron} to be positive. For $\Lambda=1$, due to periodicity of spacelike geodesics, we can further choose $0<|\alpha|,|\beta|,|\gamma|<\pi$. The description of $\X_\Lambda$ as a projective quadric in $\RP^3$ then shows that the vertices in Proposition \ref{prop:ltetrahedron} always define a lightlike tetrahedron. It also gives rise to an explicit parametrization of lightlike tetrahedra.

  \begin{proposition} \label{prop:tetpara}
  The vertices in Proposition \ref{prop:ltetrahedron} define a lightlike tetrahedron in $\X_\Lambda$ if and only if $\alpha+\beta+\gamma=0$, with $0<|\alpha|,|\beta|,|\gamma|<\pi$ if $\Lambda=1$. Up to isometries, any lightlike tetrahedron $L\subset\X_\Lambda$ admits a global parametrization
  \begin{align}\label{eq:lighttetparnew}
L=\Big\{x=\exp(r \hat X(A,B)) \mid 0\leq r\leq r(A,B)\leq \pi,\; 0\leq A,B, 1-A-B \Big\},
  \end{align}
  with
  \begin{align*}
  &X(A,B)=\ell \left(\begin{matrix} 1 & -2A \\ 2B & -1\end{matrix}\right),
  \qquad \qquad
   r(A,B)=ct_\Lambda^\inv\Bigg(\frac{\frac A {t_\Lambda(\alpha)}+ \frac B {t_\Lambda(\beta)}+\frac {A+B-1} {t_\Lambda(\gamma)}}{|X(A,B)|}\Bigg),
  \end{align*}
  where $\alpha,\beta>0$, with $\alpha+\beta<\pi$ if $\Lambda=1$.
  \end{proposition}
  
  \begin{proof} 
  Let $x_1,x_2,x_3,x_4\in\X_\Lambda$ be given as in Proposition \ref{prop:ltetrahedron} and choose lifts $x_1',x_2',x_3',x_4'\in\R^4\subset\Mat(2,\bbC_\Lambda)$. Up to isometries, and an overall change of signs of $x'_i$, we can assume
  \begin{align}\label{eq:xprimedef}
  &x_1'=\left(\begin{matrix} e^{\ell\alpha} & -2\ell s_\Lambda(\alpha) \cr 0 & e^{-\ell\alpha}\end{matrix}\right), 
  & &
  x_2'=\left(\begin{matrix} e^{\ell\beta} & 0 \cr 2\ell s_\Lambda(\beta) &  e^{-\ell\beta} \end{matrix}\right),
  \cr
  &x_3'=\left(\begin{matrix} e^{-\ell\gamma} & 0 \cr 0& e^{\ell\gamma} \end{matrix}\right), 
  & &x_4'=\left(\begin{matrix} 1 & 0 \cr 0& 1\end{matrix}\right),
  \end{align}
  with $\alpha,\beta>0$. For $\Lambda=1$ we can further assume $0<\alpha,\beta<\pi$ and $-\pi<\gamma<\pi$.
  
  Consider the convex cone spanned by these lifted vertices
  \begin{align}\label{eq:convcomb}
L'=\Big\{x'=\sum_{i=1}^4 a_ix_i'\mid a_i\geq 0,\;\sum_{i=1}^4 a_i\neq 0\Big\}\subset\R^4.
  \end{align}
  This cone projects to $\X_\Lambda$ if and only if every $x'\in L'$ satisfies $\det(x')>0$ for $\Lambda=-1,1$ and $\tr(x')\neq0$ for $\Lambda=0$. A direct computation shows that this is always satisfied for $\Lambda=-1,0$, without any additional requirements on $\alpha,\beta,\gamma$. For $\Lambda=1$ the condition becomes
  \begin{align*}
  \langle x',x'\rangle=\Big(a_4&+a_i\cos(\alpha_i)+a_j\cos(\alpha_j)+a_k\cos(\alpha_k)\Big)^2
  \cr
  &+\Big(a_i\sin(\alpha_i)-a_j\sin(\alpha_j)-a_k\sin(\alpha_k)\Big)^2-4a_ja_k\sin(\alpha_j)\sin(\alpha_k)>0,
  \end{align*}
  for all $a_l$ as in \eqref{eq:convcomb}, with distinct $i,j,k\in\{1,2,3\}$ and $\alpha_1=\alpha$, $\alpha_2=\beta$ and $\alpha_3=\gamma$.
  
  Note that this imposes restrictions on the possible values of $\alpha,\beta,\gamma$, but does not determine $\gamma$ uniquely as a function of $\alpha,\beta$. The condition that the internal geodesics starting at each vertex are spacelike imposes further restrictions, namely
  \begin{align*}
  \Big(a_is_\Lambda(\alpha_i)-a_js_\Lambda(\alpha_j)-a_ks_\Lambda(\alpha_k)\Big)^2-4a_ja_ks_\Lambda(\alpha_j)s_\Lambda(\alpha_k)>0,
  \cr
  \Big(a_4s_\Lambda(\alpha_i)+a_js_\Lambda(\alpha_i+\alpha_j)+a_ks_\Lambda(\alpha_i+\alpha_k)\Big)^2-4a_ja_ks_\Lambda(\alpha_j)s_\Lambda(\alpha_k)>0,
  \end{align*}
  and these are satisfied for all $a_l$ as in \eqref{eq:convcomb} if and only if $\alpha_1+\alpha_2+\alpha_3=0$.
  
  A global parametrization of the coefficients $a_l$ in \eqref{eq:convcomb} can then be obtained via
  \begin{align*}
  &\lambda a_1=\frac A {s_\Lambda(\alpha)},
  \quad
  \lambda a_2=\frac B {s_\Lambda(\beta)},
  \quad
  \lambda a_3=\frac {1-A-B} {s_\Lambda(\alpha+\beta)},
  \\
  &\lambda a_4=\frac{\sqrt{1-4AB}}{t_\Lambda(r)}-\left(\frac A {t_\Lambda(\alpha)}+ \frac B {t_\Lambda(\beta)}+\frac {1-A-B} {t_\Lambda(\alpha+\beta)}\right),
  \end{align*}
  where $A,B$ and $r$ satisfy the conditions in \eqref{eq:lighttetparnew} and $\lambda\in\R_+$. By comparison with \eqref{eq:lighttetparnew} we find
  \begin{align*}
  x'=\sum_{i=1}^4 a_ix_i'=\lambda \exp\Big(r\hat X(A,B)\Big).
  \end{align*}
  \end{proof}
  
  This proposition gives a geometric interpretation of the parameters $\alpha,\beta,\gamma$ as the edge lengths of the lightlike tetrahedron. The vertices of the tetrahedron are given by $r=0$ and by  $r=r(A,B)$ for  $(A,B)=(0,0)$, $(1,0)$ and $(0,1)$ in the parametrization \eqref{eq:lighttetparnew}.  With the formulas for arc lengths in Proposition \ref{prop:geodesdist} one obtains 
  
  \begin{corollary} \label{cor:edgelength}
  A lightlike tetrahedron $L$ is determined up to isometries by its edge lengths. If $L$ is parametrized as in Proposition \ref{prop:tetpara}, its edge lengths are  $\alpha$, $\beta$ and $\alpha+\beta$, with opposite edges having equal lengths.
  \end{corollary}
  
  Using the parametrization   in Proposition \ref{prop:tetpara} and the formulas for arc lengths in Proposition \ref{prop:geodesdist},
   we obtain more general expressions for the arc lengths of geodesic segments between points on opposite edges.

  \begin{proposition}\label{prop:geoddltet}  Let $L$ be a lightlike tetrahedron in $\X_\Lambda$ parametrized as in Proposition \ref{prop:tetpara} and with edge geodesics $x_{ij}$  as in  \eqref{eq:xijparam}. Then the arc length $d_{4i,jk}(s,t)$  of a geodesic segment between points $x_{4i}(\frac{\alpha_i}{2}+s)$ and $x_{jk}(\frac{\alpha_i}{2}+t)$ on opposite edges $e_{4i},e_{jk}$ satisfies
  \begin{align}\label{eq:geoddistconc}
  |c_{\sigma\Lambda}(d_{4i,jk}(s,t))|=\Big|\frac{c_\Lambda(s+t)s_\Lambda(\alpha_j)+c_\Lambda(s-t)s_\Lambda(\alpha_k)}{s_\Lambda(\alpha_j+\alpha_k)}\Big|,\qquad \Lambda\neq 0,
  \cr 
  \sigma d_{4i,jk}(s,t)^2=\frac{(s+t)^2\alpha_j+(s-t)^2\alpha_k}{\alpha_j+\alpha_k}-\alpha_j\alpha_k,\qquad \Lambda=0,
  \end{align}
  with $\alpha_1=\alpha$, $\alpha_2=\beta$ and $\alpha_3=\gamma$, $s,t\in (-\frac {|\alpha_i|} 2,\frac {|\alpha_i|} 2)$  and $\sigma=-1,0,+1$ if the geodesic segment between them is timelike, lightlike or spacelike, respectively. 
  \end{proposition}

  Note that the formulas for $\Lambda=0$ in \eqref{eq:geoddistconc} are obtained from the ones for $\Lambda\neq 0$ by expanding the latter as a power series in $\alpha,\beta$ and $\Lambda$. Expression \eqref{eq:expellsincos}  for the generalized trigonometric functions in terms of the exponential map extends to general $\Lambda=-\ell^2\in\R$  and defines $s_\Lambda$ and $c_\Lambda$  as power series in $\Lambda$. One can thus expand the left- and right-hand side of the equations for $\Lambda\neq 0$ in \eqref{eq:geoddistconc} as a power series in $\Lambda$. To zero-th order in $\Lambda$ these equations are satisfied trivially, and at first order one obtains the equations for $\Lambda=0$.
  
  Proposition \ref{prop:tetpara} and Corollary \ref{cor:edgelength}  show that 
  for all admissible values of the edge lengths $\alpha,\beta$, the lightlike tetrahedron has a distinguished pair of opposite edges, namely its longest edge pair of edge length $\alpha+\beta$.
  Proposition \ref{prop:geoddltet} 
  implies that this edge pair also plays a distinguished role with respect to the causal structure. 
  The longest edge pair is the only pair of opposite edges that are connected by {\em timelike} geodesic segments.

  \begin{corollary}\label{cor:timelikeseg}
  There is a timelike geodesic segment between two opposite edges of a lightlike tetrahedron if and only if these are its longest edges. The arc length of such timelike geodesic segments is maximized at the midpoints of the longest edges.
  \end{corollary}
  
  \begin{proof}
  The functions $d_{4i,jk}(s,t)$ have a single critical point for $s,t$ in $ (-|\alpha_i|/2,|\alpha_i|/2)$, namely at  $(0,0)$. If one chooses $\alpha_1=\alpha$, $\alpha_2=\beta$ and $\alpha_3=-\alpha-\beta$, with $\alpha,\beta>0$, as in Proposition \ref{prop:tetpara}, the longest edges are $e_{12}$ and $e_{43}$ and $(0,0)$ is a local maximum for $d_{43,12}$. By inspection of the formulas \eqref{eq:geoddistconc}, one finds that 
  $c_{\Lambda\sigma}(d_{43,12}(0,0))>1$ for $\Lambda=1$,  $c_{\Lambda\sigma}(d_{43,12}(0,0))<1$ for $\Lambda=-1$ and $\sigma d_{43,12}(0,0)^2=-\alpha\beta$ for $\Lambda=0$. This shows in all cases that $\sigma=-1$ and hence the geodesic segments between the midpoints of $e_{43}$ and $e_{12}$ are timelike. For $d_{42,13}$ and $d_{41,23}$, the point $(0,0)$ is a saddle point. By investigating the boundary values of these functions, one finds that all geodesics connecting points on $e_{42}$ and $e_{13}$ or points on $e_{41}$ and $e_{23}$ are spacelike.
  \end{proof}

  \begin{remark} Corollary \ref{cor:timelikeseg} shows that for $\Lambda=-1,0$ a lightlike tetrahedron $L$ is the intersection of the past of the geodesic containing one of the two longest edges with the future of the geodesic containing the other.  For $\Lambda=1$, the space $\X_{1}=\mathrm{dS}^3$ is not time orientable, but it still holds that any point in $L$ is connected to each of two longest edges by a timelike geodesic segment in $L$ that ends on a face through the opposite edge. 
  \end{remark}

  Instead of using geodesics through the midpoints of its edges, we
   can also characterize the geometry of a lightlike tetrahedron in terms of lightlike geodesics. For this, we consider lightlike geodesics in the geodesic planes defined by its faces and through one of its vertices. The longest edges of a lightlike tetrahedron are then distinguished by the fact that such lightlike geodesics through their endpoints intersect the opposite face.

  \begin{corollary}\label{cor:lightint}
  Let $L$ be a lightlike tetrahedron in $\X_\Lambda$ with vertices $x_i$ and  $n_{ij}$ the unique lightlike geodesic through $x_i$ in the geodesic plane containing the face opposite $x_j$. 
  
  Then $n_{ij}$ intersects the edge geodesic $x_{kl}$ if and only if $i=k$, $i=l$ or $i,j,k,l\in\{1,2,3,4\}$ are all distinct.
  The intersection points are given by
  \begin{align*}
  &n_{ij}\cap x_{il}=n_{ij}\cap x_{ki}=x_i, & &
  n_{ij}\cap x_{4l}=x_{4l}(-\alpha_i),\\
  &n_{4j}\cap x_{kl}=x_{kl}(-\alpha_k), & & n_{i4}\cap x_{kl}=x_{kl}(-\alpha_l),
  \end{align*}
  where $\alpha_1=\alpha$, $\alpha_2=\beta$, $\alpha_3=\gamma$ and the edge geodesics $x_{ij}$ are parametrized as in \eqref{eq:xijparam}. 
  
  In particular, $n_{ij}$ intersects the tetrahedron $L$ outside $x_i$ if and only if $x_{ij}$ contains one of the longest edges of $L$.
  \end{corollary}
  
  \begin{proof}
  If $n_{ij}$ and $x_{kl}$ intersect, then $x_i,x_k,x_l$ lie on a common lightlike plane. Since $n_{ij}$ lies on the lightlike plane opposite $x_j$, the only edge geodesic containing $x_j$ which intersects $n_{ij}$ is $x_{ij}$, with the intersection point given by $x_i$. Furthermore, the edge geodesics $x_{kl}$ opposite to $x_j$ (that is, with $k,l\neq j$) intersect $n_{ij}$ at a single point. This is given by $x_i$, if $k=i$ or $l=i$. For $k,l\neq i$, the intersection point can be computed solving
  \begin{align*}
  n_{ij}(\theta_{ij})=A_i\rhd\exp(\theta_{ij} N_{ij})=A_k\rhd\exp(t_{kl} X_{kl})=x_{kl}(t_{kl})
  \end{align*}
  for $\theta_{ij}$ and $t_{kl}$, where $N_{ij}$ and $X_{kl}$ are given by \eqref{eq:nexpr}, \eqref{eq:nmat}, and \eqref{eq:xijsimple}, \eqref{eq:xijhelp}.
  \end{proof}

   Corollary \ref{cor:lightint} defines canonical projections of each vertex $x_i$ on each of the geodesics $x_{kl}$ containing its opposite edge $e_{kl}$. We will call these \emph{null projections} in the following.  Thus, given a vertex $x_i$, we define the point $\pi_{kl}(x_i)$ on the geodesic $x_{kl}$ as the unique intersection point between $x_{kl}$ and the lightlike geodesic $n_{ij}$, as shown in Figure \ref{fig:orthproj}. It should be emphasized that $\pi_{kl}(x_i)$ may lie outside of the corresponding edge $e_{kl}$.
  
  \begin{figure}[t]
  \begin{center}
  \begin{tikzpicture}[scale=.45]
  \filldraw[line width=0pt, fill=gray, opacity=.2] (-4,-3)--(3,3)--(-3,3);
  \draw[line width=1pt, color=black] (4,-3)--(3,3);
  \draw[line width=1pt, color=black] (-4,-3)--(-3,3);
  \draw[line width=1pt, color=black] (4,-3)--(-4,-3);
  \draw[line width=1pt, color=black] (5,3)--(-5,3);
  \draw[line width=1pt, color=black] (-4,-3)--(3,3);
  \draw[line width=1pt, color=black, style=dashed] (4,-3)--(0.21,0.25);
  \draw[line width=1pt, color=black, style=dashed] (-3,3)--(-0.1,0.5);
  \node at (-4,-3) [color=black, anchor=east] {$x_i$};
  \node at (4,-3) [color=black, anchor=west] {$x_j$};
  \node at (-4,3) [color=black, anchor=south] {$x_k$};
  \node at (3,3) [color=black, anchor=south] {$x_l$};
  \node at (5,3) [color=black, anchor=west] {$x_{kl}$};
  \draw[line width=1.5pt, color=darkgray] (-4.5,-4)--(0,5);
  \node at (-1.5,4.5) [color=darkgray, anchor=south west] {$n_{ij}$};
  \draw[color=black, fill=black] (-4,-3) circle(.15);
  \draw[color=black, fill=black] (4,-3) circle(.15);
  \draw[color=black, fill=black] (-3,3) circle(.15);
  \draw[color=black, fill=black] (3, 3) circle(.15);
  \draw[color=black, fill=gray] (-1, 3) circle(.15);
  \node at (-0.65,3)[anchor=south west, color=darkgray] {$\pi_{kl}(x_i)$};
  \end{tikzpicture}
  \end{center}
  \caption{Null projection of the vertex $x_i$ on the opposite edge $e_{kl}$.}
  \label{fig:orthproj}
  \end{figure}
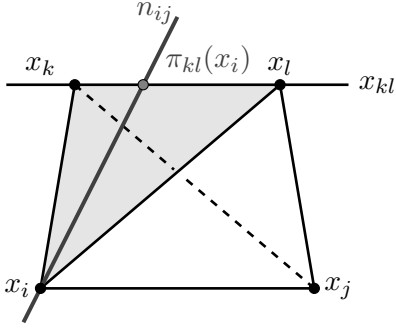
  
  Each geodesic $x_{kl}$ contains exactly two such projections, namely  $\pi_{kl}(x_i)$ and $\pi_{kl}(x_j)$ for the two vertices $x_i$ and $x_j$ opposite  $x_{kl}$. For each edge geodesic $x_{ij}$,  this defines two geodesic planes that intersect in $x_{ij}$, the planes through $x_i,x_j,\pi_{kl}(x_i)$ and through $x_i,x_j,\pi_{kl}(x_j)$, as shown in Figure \ref{fig:intplanes}. We call them the {\em internal planes} of the lightlike tetrahedron at the edge $e_{ij}$. 
  The angles between these planes are given by the ratios of the generalized sine functions of the edge lengths.

  \begin{proposition}\label{prop:lorentz_angle}Let $L$ be a lightlike tetrahedron in $\X_\Lambda$ with vertices $x_i$ as in Proposition \ref{prop:ltetrahedron}. Then the Lorentzian angle $\varphi_{ij}$ between the internal planes at the edge $e_{ij}$  is given by
  \begin{align*}
  2\cosh(\varphi_{ij})=|z_{ij}|+|z_{ij}|^{-1},
  \end{align*}
  where $|z_{ij}|=|z_{ji}|$ and
  \begin{align*}
  &|z_{12}|=|z_{34}|=\Big|\frac{s_\Lambda(\beta)}{s_\Lambda(\alpha)}\Big|,
  & & |z_{31}|=|z_{24}|=\Big|\frac{s_\Lambda(\alpha)}{s_\Lambda(\gamma)}\Big|,
  & & |z_{23}|=|z_{14}|=\Big|\frac{s_\Lambda(\gamma)}{s_\Lambda(\beta)}\Big|.
  \end{align*}
  \end{proposition}

  \begin{proof}
  Denote by $x_{ij}$ and $x_{kl}$ the geodesics through $x_i$, $x_j$ and through $x_k$, $x_l$, parametrized as in \eqref{eq:xijparam}.
  For any point $x_{kl}(t)$ on the geodesic $x_{kl}$ we can parametrize the plane through $x_i,x_j,x_{kl}(t)$ as
  \begin{align*}
  P_{ij,kl}(t)=\Big\{A_i\rhd \exp\Big(r X_{ij}+s X_{ij,kl}(t)\Big)  \mid r,s\in \R\Big\},
  \end{align*}
  where $X_{ij}\in\x_\Lambda$ is a unit vector parameterizing the geodesic $x_{ij}$  as in \eqref{eq:xijparam} and $X_{ij, kl}(t)\in\x_\Lambda$ is the unit vector parameterizing the geodesic
  $x_{ij,kl}$ through $x_i$ and $x_{kl}(t)$ via
  $$
  x_{ij,kl}(s)=A_i\rhd \exp(s X_{ij,kl}(t)).
  $$
  These vectors can be computed directly as the normalized trace-free parts of $A_i^{-1}\rhd x_j$ and $A_i^{-1}\rhd x_{kl}(t)$, respectively.
  
  We can then compute the normal vector  $A_i\rhd N_{ij,kl}(t)$ of $P_{ij,kl}(t)$ at $x_i$ from the conditions
  \begin{align*}
  \langle N_{ij,kl}(t),X_{ij}\rangle=0,\qquad\qquad \langle N_{ij,kl}(t),X_{ij,kl}(t)\rangle=0,
  \end{align*}
  where $X_{ij}$ are the matrices from \eqref{eq:xijsimple} and \eqref{eq:xijhelp}.
  This yields for all distinct $i,j,k\in\{1,2,3\}$
  \begin{align*}
  N_{ij,4k}(t)=\tfrac{1}{|r_{ijk}(t)|^{1/2}}\Big(N_{i4}-r_{ijk}(t)N_{ik}\Big),
  \qquad\qquad
  N_{4k,ij}(t)=\tfrac{1}{|r_{ijk}(t)|^{1/2}}\Big(N_{4i}-r_{ijk}(t)N_{4j}\Big),
  \end{align*}
  with
  $r_{ijk}(t)=\tfrac{s_\Lambda(\alpha_k-t)}{s_\Lambda(t)}\tfrac{s_\Lambda(\alpha_i)}{s_\Lambda(\alpha_j)}$ and $N_{ik}$ and $N_{i4}$ given by \eqref{eq:nexpr} and \eqref{eq:nmat}.
  
  Corollary \ref{cor:lightint} gives the null projections of $x_i$ and $x_j$ on the opposite edge geodesic $x_{4k}$
  \begin{align*}
  \pi_{4k}(x_i)=x_{ij}(-\alpha_i),\qquad \pi_{4k}(x_j)=x_{ij}(-\alpha_j),
  \end{align*}
  and the null projections of $x_4$ and $x_k$ on $x_{ij}$
  \begin{align*}
  \pi_{ij}(x_4)=x_{ij}(-\alpha_i),\qquad \pi_{ij}(x_k)=x_{ij}(-\alpha_j).
  \end{align*}
  In particular, the normal vectors at $x_i$  of the plane $P_{ij,4k}(-\alpha_i)$ through $x_i,x_j,\pi_{4k}(x_i)$ and of the plane $P_{ij,4k}(-\alpha_j)$ through $x_i,x_j,\pi_{4k}(x_j)$ are given by
  \begin{align*}
  N_{ij,4k}(-\alpha_i)=N_{i4}-N_{ik},
  \qquad
  N_{ij,4k}(-\alpha_j)=\frac{|s_\Lambda(\alpha_j)|}{|s_\Lambda(\alpha_i)|}N_{i4}-\frac{|s_\Lambda(\alpha_i)|}{|s_\Lambda(\alpha_j)|}N_{ik}.
  \end{align*}
  Similarly, the normal vectors at $x_4$  to the planes $P_{4k,ij}(-\alpha_i)$  through $x_4,x_k,\pi_{ij}(x_4)$ and $P_{ij,4k}(-\alpha_j)$  through $x_4,x_k,\pi_{ij}(x_i)$ are given by
  \begin{align*}
  N_{4k,ij}(-\alpha_i)=N_{4i}-N_{4j},
  \qquad
  N_{4k,ij}(-\alpha_j)=\frac{|s_\Lambda(\alpha_j)|}{|s_\Lambda(\alpha_i)|}N_{4i}-\frac{|s_\Lambda(\alpha_i)|}{|s_\Lambda(\alpha_j)|}N_{4j}.
  \end{align*}
  In both cases, we find that the Lorentzian angle between the two planes is given by
  \begin{align*}
  2\cosh(\varphi_{ij})=2\cosh(\varphi_{4k})=\frac{|s_\Lambda(\alpha_i)|}{|s_\Lambda(\alpha_j)|}+\frac{|s_\Lambda(\alpha_j)|}{|s_\Lambda(\alpha_i)|}.
  \end{align*}
  The claim then follows by setting $\alpha_1=\alpha$, $\alpha_2=\beta$, $\alpha_3=\gamma=-\alpha-\beta$.
  \end{proof}
  
  \begin{figure}[t]
  \begin{center}
  \begin{tikzpicture}[scale=.45]
  \draw[line width=1pt, color=black, style=dashed] (4,-3)--(0.21,0.25);
  \draw[line width=1pt, color=black, style=dashed] (-3,3)--(-0.1,0.5);
  \filldraw[line width=0pt, fill=blue, opacity=.2] (-4,-3)--(-1,3)--(4,-3);
  \filldraw[line width=0pt, fill=red, opacity=.2] (-4,-3)--(1,3)--(4,-3);
  \draw[line width=1pt, color=blue] (-4,-3)--(-1,3);
  \draw[line width=1pt, color=blue] (4,-3)--(-1,3);
  \draw[line width=1pt, color=red] (-4,-3)--(1,3);
  \draw[line width=1pt, color=red] (4,-3)--(1,3);
  \draw[line width=1pt, color=black] (4,-3)--(3,3);
  \draw[line width=1pt, color=black] (-4,-3)--(-3,3);
  \draw[line width=1pt, color=black] (4,-3)--(-4,-3);
  \draw[line width=1pt, color=black] (5,3)--(-5,3);
  \draw[line width=1pt, color=black] (-4,-3)--(3,3);
  \draw[color=black, fill=black] (-4,-3) circle(.15);
  \draw[color=black, fill=black] (4,-3) circle(.15);
  \draw[color=black, fill=black] (-3,3) circle(.15);
  \draw[color=black, fill=black] (3, 3) circle(.15);
  \draw[color=red, fill=red] (1, 3) circle(.15);
  \draw[color=blue, fill=blue] (-1, 3) circle(.15);
  \node at (-4,-3) [color=black, anchor=north east] {$x_i$};
  \node at (4,-3) [color=black, anchor=north west] {$x_j$};
  \node at (-3,3) [color=black, anchor=south east] {$x_k$};
  \node at (3,3) [color=black, anchor=south west] {$x_l$};
  \node at (-1.5,3) [color=blue, anchor=south] {$\pi_{kl}(x_i)$};
  \node at (1.5,3) [color=red, anchor=south] {$\pi_{kl}(x_j)$};
  \node at (5,3)[color=black, anchor=west]{$x_{kl}$};
  \end{tikzpicture}
  \caption{Internal planes of a lightlike tetrahedron at the edge $e_{ij}$.}
  \end{center}
  \label{fig:intplanes}
  \end{figure}
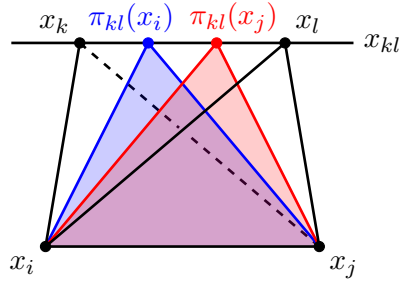
  
  Proposition \ref{prop:lorentz_angle} associates to each edge of a lightlike tetrahedron $L$ a Lorentzian angle that is given by the ratios of generalized sine functions of the edge lengths $\alpha,\beta,\gamma$. Combining these with the corresponding edge lengths, we may define a generalized complex parameters $z_{ij}=z_{ji}\in\bbC_\Lambda^\times$ for each edge $e_{ij}$ of $L$, namely
  \begin{align}\label{eq:zdefflight}
  &z_{12}=z_{34}=-\frac{s_\Lambda(\beta)}{s_\Lambda(\alpha)}e^{\ell\gamma},
  & & z_{31}=z_{24}=-\frac{s_\Lambda(\alpha)}{s_\Lambda(\gamma)}e^{\ell\beta},
  & & z_{23}=z_{14}=-\frac{s_\Lambda(\gamma)}{s_\Lambda(\beta)}e^{\ell\alpha}.
  \end{align}
  These are the shape parameters of the lightlike tetrahedron $L$.
  Note that opposite edges have equal shape parameters, while the shape parameters of adjacent edges satisfy the cross-ratio relations
  $$z'=\frac{1}{1-z},\qquad\qquad z''=\frac{z-1}{z}.$$
  
   Corollary \ref{cor:edgelength} and Proposition \ref{prop:lorentz_angle} show that the arguments of the shape parameters determine the edge lengths of a lightlike tetrahedron, while their moduli determine the angles between its internal planes. We will show in Section \ref{sec:idealtet} that they play a similar role to the classical shape parameters of ideal hyperbolic tetrahedra.
  In particular, the shape parameter of a single edge uniquely determines the geometry of a lightlike tetrahedron. 
  
  The shape parameter can also be characterized in terms of the symmetries of a lightlike tetrahedron.

  \begin{proposition}\label{eq:geomlightlike} Let $L\subset\X_\Lambda$ be a lightlike tetrahedron with vertices $x_1,x_2,x_3,x_4$ and  $x_{ij}$ the geodesic through $x_i$ and $x_j$, oriented from $x_i$ to $x_j$. 
  
  Then there is a unique isometry $T_{ij}\in\PGL^+(2,\bbC_\Lambda)$ that stabilizes $x_{ij}$, with its orientation and its adjacent null planes, which maps $x_i$ to $x_j$ and the normal vector $A_i\rhd N_{ik}$ to $A_j\rhd N_{jk}$, up to a sign.
  With the parametrization from Proposition \ref{prop:ltetrahedron} one has
  \begin{align*}
  T_{ij}=A_i\Big(\frac{z_{ij}}{2}(\mathbb{1}+\Im X_{ij})-\frac{\sigma_{ij}}{2}(\mathbb{1}-\Im X_{ij})\Big)A_i^\inv,
  \end{align*}
  where $A_i\in \PGL^+(2,\bbC_\Lambda)$ with $A_i\rhd \mathbb{1}=x_i$,   the tangent vector $X_{ij}$  of $x_{ij}$ is given by \eqref{eq:xijsimple}, \eqref{eq:xijtriple}, the shape parameter $z_{ij}=z_{ji}\in\bbC_\Lambda^\times$  by \eqref{eq:zdefflight}
  and $\sigma_{ij}=\sigma_{ji}\in\{\pm1\}$ by
  \begin{align*}
  &\sigma_{12}=\sigma_{34}=\sgn\Big(\frac{s_\Lambda(\beta)}{s_\Lambda(\alpha)}\Big),
  & & \sigma_{31}=\sigma_{24}=\sgn\Big(\frac{s_\Lambda(\alpha)}{s_\Lambda(\gamma)}\Big),
  & & \sigma_{23}=\sigma_{14}=\sgn\Big(\frac{s_\Lambda(\gamma)}{s_\Lambda(\beta)}\Big).
  \end{align*}
  \end{proposition}
  
  \begin{figure}[t]
  \begin{center}
  \begin{tikzpicture}[scale=.45]
  \draw[color=black, fill=black] (-6,0) circle(.15);
  \draw[color=black, fill=black] (6,0) circle(.15);
  \draw[color=black, fill=black] (0,4) circle(.15);
  \draw[color=black, fill=black] (0, -4) circle(.15);
  \draw[line width=1pt, color=black, ->,>=stealth] (0,-4)--(0,3.8);
  \draw[line width=.5pt, color=black,  ->,>=stealth, ] (0,-4)--(5.8,-.1);
  \draw[line width=.5pt, color=black,  ->,>=stealth, ] (0,-4)--(-5.8,-.1);
  \draw[line width=.5pt, color=black,  ->,>=stealth, ] (0,4)--(5.8,.1);
  \draw[line width=.5pt, color=black,  ->,>=stealth, ] (0,4)--(-5.8,.1);
  \draw[line width=.5pt, color=black, style=dashed,   ] (-6,0)--(-1.5,0);
  \draw[line width=.5pt, color=black, style=dashed,] (6,0)--(2.5,0);
  \node at (3,-2.2)[anchor=west, color=black, ]{$x_{il}$};
  \node at (3,2.2)[anchor=west, color=black, ]{$x_{jl}$};
  \node at (-3,-2.2)[anchor=east, color=black, ]{$x_{ik}$};
  \node at (-3,2.2)[anchor=east, color=black, ]{$x_{jk}$};
  \node at (0,0)[anchor=east, color=black, ]{$x_{ij}$};
  \draw[line width=.5pt, ->,>=stealth,color=black,  ] (1,-2)--(1,2);
  \node at (1.1,0)[anchor=west, color=black, ]{$T_{ij}$};
  \node at (0,-4.2)[anchor=north, color=black, ]{$x_i$};
  \node at (0,4.2)[anchor=south, color=black, ]{$x_j$};
  \node at (6.2,0)[anchor=west, color=black, ]{$x_l$};
  \node at (-6.2,0)[anchor=east, color=black, ]{$x_k$};
  \node at (2.5,-4)[anchor=north west, color=black, ]{$A_i\rhd N_{ik}$};
  \draw[line width=1pt, ->,>=stealth, color=black, ] (.5,-3)--(2.5, -4);
  \draw[line width=1pt, ->,>=stealth, color=black, ] (.5,3)--(2.5, 4);
  \node at (2.5, 4)[anchor=south west, color=black, ]{$A_j\rhd N_{jk}$};
  \end{tikzpicture}
  
  \caption{The isometries and normal vectors from Proposition \ref{eq:geomlightlike}.}
  \end{center}
  \end{figure}
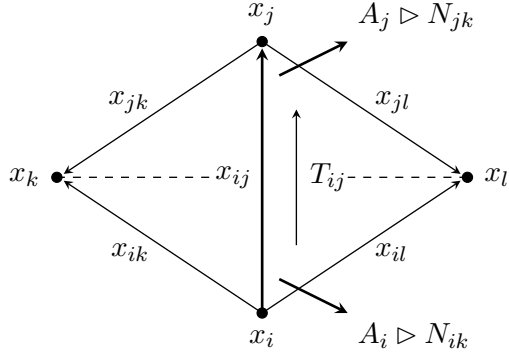
  
  \begin{proof}
  This follows from the expressions \eqref{eq:nexpr}, \eqref{eq:xijsimple}, \eqref{eq:xijhelp}, \eqref{eq:nmat} for the normal and tangent vectors derived in the proof of Proposition \ref{prop:ltetrahedron}.
  
  By Proposition \ref{prop:stabilizer} and equation \eqref{eq:xijparam}, we can parametrize $T_{ij}$ as
  \begin{align*}
  T_{ij}&=A_i\exp(\frac{\ell\theta_{ij}}{2}\Im X_{ij})\Big(a_{ij}\mathbb{1}+b_{ij}\Im X_{ij}\Big)A_i^\inv
  \cr
  &=A_i\Big(\tfrac{a_{ij}+b_{ij}}{2}e^{\ell\theta_{ij}/2}(\mathbb{1}+\Im X_{ij})+\tfrac{a_{ij}-b_{ij}}{2}e^{-\ell\theta_{ij}/2}(\mathbb{1}-\Im X_{ij})\Big)A_i^\inv.
  \end{align*}
  
  The requirement that $T_{ij}$ maps $x_i$ to $x_j$ determines the parameter $\theta_{ij}$ as follows. Using equation \eqref{eq:xexps}, we can rewrite this requirement as
  \begin{align*}
  T_{ij}\rhd x_i=A_i\exp\Big(\frac{\ell\theta_{ij}}{2}\Im X_{ij}\Big)\rhd\mathbb{1}=x_j=A_j\rhd\mathbb{1},
  \end{align*}
  which is equivalent to
  \begin{align*}
  R_{ij}:=\exp\Big(-\frac{\alpha_j}{2}X_{4j}\Big)\exp\Big(\frac{\alpha_i}{2}X_{4i}\Big)\exp\Big(\frac{\theta_{ij}}{2}X_{ij}\Big)\in \PSL(2,\R)_\Lambda.
  \end{align*}
  By an explicit computation of the matrices $R_{ij}$, one finds that this is satisfied if and only if
  \begin{align*}
  \theta_{12}=\theta_{34}=\alpha_3,\qquad \theta_{31}=\theta_{24}=\alpha_2,\qquad \theta_{23}=\theta_{14}=\alpha_1,
  \end{align*}
  with $\theta_{ij}=\theta_{ji}$. In the case $\Lambda=1$, this holds up to multiples of $\pi$.
  
  To investigate the action of $T_{ij}$ on the normal vectors of the adjacent faces, denote by $A_i\rhd N_{ik}$ the lightlike vector at $x_i$ normal to a face $f_k$ adjacent to $x_{ij}$, with the normalization $\langle N_{ik},X_{ik}\rangle_{\x_\Lambda}=-1$, and with distinct $i,j,k\in\{1,2,3,4\}$. Then $T_{ij}$ stabilizes the null planes intersecting along the geodesic $x_{ij}$ and preserves its orientation if and only if
  \begin{align*}
  T_{ij}\rhd (A_i\rhd N_{ik})=\sigma_{ij} A_j\rhd N_{jk},
  \qquad\qquad
  T_{ij}\rhd (A_i\rhd X_{ij})=-A_j\rhd X_{ji},
  \end{align*}
   for some $\sigma_{ij}\in\R^\times$. With the condition  $|\sigma_{ij}|=1$ this is equivalent to
  \begin{align*}
  \Big(a_{ij}\mathbb{1}+b_{ij}\Im X_{ij}\Big)N_{ik}\Big(a_{ij}\mathbb{1}-b_{ij}\Im X_{ij}\Big)=\sigma_{ij}R_{ij}^\inv N_{jk}R_{ij},\qquad a_{ij}^2-b_{ij}^2=1,
  \end{align*}
  and, again by a direct computation, one finds
  \begin{align*}
  a_{12}\pm b_{12}=a_{34}\pm b_{34}=\Big|\frac{s_\Lambda(\alpha_2)}{s_\Lambda(\alpha_1)}\Big|^{\pm 1/2},
  \qquad \qquad
  \sigma_{12}=\sigma_{34}=\sgn\Big(\frac{s_\Lambda(\alpha_2)}{s_\Lambda(\alpha_1)}\Big),
  \cr
  a_{31}\pm b_{31}=a_{24}\pm b_{24}=\Big|\frac{s_\Lambda(\alpha_1)}{s_\Lambda(\alpha_3)}\Big|^{\pm1/2},
  \qquad\qquad
  \sigma_{31}=\sigma_{24}=\sgn\Big(\frac{s_\Lambda(\alpha_1)}{s_\Lambda(\alpha_3)}\Big),
  \cr
  a_{23}\pm b_{23}=a_{14}\pm b_{14}=\Big|\frac{s_\Lambda(\alpha_3)}{s_\Lambda(\alpha_2)}\Big|^{\pm1/2},
  \qquad\qquad
  \sigma_{23}=\sigma_{14}=\sgn\Big(\frac{s_\Lambda(\alpha_3)}{s_\Lambda(\alpha_2)}\Big),
  \end{align*}
  with $a_{ij}=a_{ji}$, $b_{ij}=b_{ji}$ and $\sigma_{ij}=\sigma_{ji}$. Factoring out $-\sigma_{ij}(a_{ij}-b_{ij})e^{-\ell\theta_{ij}/2}$ and inserting $\alpha_1=\alpha$, $\alpha_2=\beta$ and $\alpha_3=\gamma$, we obtain the expressions in the proposition.
  \end{proof}

  \subsection{Ideal tetrahedra} 
  \label{sec:idealtet}
  
  Corollary \ref{cor:edgelength} shows that the edge lengths of a lightlike tetrahedron in $\X_\Lambda$ play a similar role to the dihedral angles of an ideal hyperbolic tetrahedron: up to isometries, they determine the lightlike tetrahedron completely. Indeed, the duality between lightlike planes in $\X_\Lambda$ and points on the ideal boundary $\partial_\infty \Y_\Lambda$ suggests that lightlike tetrahedra should be dual to tetrahedra in $\Y_\Lambda$ whose vertices are points in $\partial_\infty \Y_\Lambda$, pairwise connected by spacelike geodesics.

  Such tetrahedra are precisely the generalized ideal tetrahedra introduced and investigated by Danciger in \cite{DaPhD, Da2}, up to the fact that we exclude the degenerate ones. In this section we review the results on generalized ideal tetrahedra in \cite{DaPhD, Da2} that are needed in the following and relate them to the corresponding statements about lightlike tetrahedra. We then show that lightlike and ideal tetrahedra are dual under the projective duality from Sections \ref{sec:projdual1} and \ref{subsec:ideldual}.

  \begin{definition}
  An ideal tetrahedron in $\Y_\Lambda$ is a non-degenerate convex geodesic 3-simplex whose vertices are points in $\partial_\infty\Y_\Lambda$ and whose faces lie on spacelike geodesic planes.
  \end{definition}
  
  As all vertices of an ideal tetrahedron are contained in $\partial_\infty\Y_\Lambda$ and all faces lie on spacelike geodesic planes, the action of the isometry group $\PGL^+(2,\bbC_\Lambda)$ on $\partial_\infty \Y_\Lambda$ allows one to map three vertices of an ideal tetrahedron to fixed reference points in $\partial_\infty \Y_\Lambda$, as in Proposition \ref{lem3pt}. 
  It is shown in \cite[Proposition 3]{Da2} that the remaining vertex is 
   then parametrized by the cross-ratio from Definition \ref{def:crossrat}. Alternatively, this vertex  is given by two real parameters $\alpha,\beta$, which can be viewed as generalized dihedral angles.

  \begin{proposition}\label{prop:idealtet}
  Let $I$  be an ideal tetrahedron in $\Y_\Lambda$  with vertices $y_1,y_2,y_3,y_4$.  Then there is a unique isometry  $B\in \PGL^+(2,\bbC_\Lambda)$ and $\alpha,\beta,\gamma\in\R$, satisfying $\alpha+\beta+\gamma=0$, such that
  \begin{align*}
  &B\rhd y_1=\left(\begin{matrix} 1 & 0 \cr 0 & 0\end{matrix}\right),
  \qquad
  B\rhd y_2=\left(\begin{matrix} 0 & 0 \cr 0 &  1 \end{matrix}\right),
  \qquad
  B\rhd y_3=\left(\begin{matrix} 1 & 1 \cr 1 & 1 \end{matrix}\right),\nonumber\\
  &
  B\rhd y_4=\left(\begin{matrix} \frac{s_\Lambda(\beta)^2}{s_\Lambda(\alpha)^2} & -\frac{s_\Lambda(\beta)}{s_\Lambda(\alpha)}e^{\ell\gamma} \cr -\frac{s_\Lambda(\beta)}{s_\Lambda(\alpha)}e^{-\ell\gamma}& 1 \end{matrix}\right).
  \end{align*}
  For $\Lambda=1$, one can choose $0<|\alpha|,|\beta|,|\gamma|<\pi$.
  \end{proposition}
  
  \begin{proof} As $y_1,y_2,y_3\in\partial_\infty \Y_\Lambda$ lie on a spacelike geodesic plane, by Proposition \ref{lem3pt}   there is a unique isometry
   $B\in \PGL^+(2,\bbC_\Lambda)$ with
  \begin{align*}
  B\rhd y_1=\left(\begin{matrix} 1 & 0 \cr 0 & 0\end{matrix}\right),
  \qquad
  B\rhd y_2=\left(\begin{matrix} 0 & 0 \cr 0 &  1 \end{matrix}\right),
  \qquad
  B\rhd y_3=\left(\begin{matrix} 1 & 1 \cr 1 & 1 \end{matrix}\right),
  \end{align*}
   up to a permutation of the vertices. The remaining vertex 
  is then given by $B\rhd y_4=v_4v_4^\dag$ with 
  \begin{align*}
  v_4=\begin{pmatrix} |z|^2 & z\\ \bar z & 1\end{pmatrix},\qquad \qquad z\in\bbC_\Lambda^\times\setminus\{1\}.
  \end{align*}

As all faces lie on spacelike geodesic planes, by the proof of Proposition \ref{prop:idealtet} one has  $1-z\in\bbC_\Lambda^\times\setminus\{1\}$. In particular, there exists $r_1,r_2,\beta,\gamma\in\R^\times$ such that
  \begin{align*}
  z=r_1 e^{\ell\gamma},\qquad\qquad 1-z=r_2 e^{-\ell\beta}.
  \end{align*}
  Eliminating the parameters $r_1,r_2$ yields
  \begin{align}\label{eq:zpar}
  z=-\frac{s_\Lambda(\beta)}{s_\Lambda(\alpha)}e^{\ell\gamma},\qquad\qquad \alpha+\beta+\gamma=0,
  \end{align}
  and therefore
  \begin{align*}
  B\rhd y_4=\left(\begin{matrix}\frac{s_\Lambda(\beta)^2}{s_\Lambda(\alpha)^2} & -\frac{s_\Lambda(\beta)}{s_\Lambda(\alpha)}e^{\ell\gamma} \cr -\frac{s_\Lambda(\beta)}{s_\Lambda(\alpha)}e^{-\ell\gamma}& 1\end{matrix}\right).
  \end{align*}
  \end{proof}

  Equation \eqref{eq:zpar} relates the parameters $\alpha,\beta$ that parametrize an ideal tetrahedron in Proposition \ref{prop:idealtet} to the generalized cross-ratio of its vertices from Definition \ref{def:crossrat}. By considering also the images of the cross-ratio under the action of the subgroup \eqref{eq:permutmatrix}  that permutes the vertices $B\rhd y_1$, $B\rhd y_2$ and  $B\rhd y_3$, one obtains all the cross-ratios of a generalized ideal tetrahedron \cite[Section 3.1]{Da2}.

  \begin{corollary}\label{cor:crossratio}
  The cross-ratios of vertices of the ideal tetrahedron in Proposition \ref{prop:idealtet}  are given by
  \begin{align*}
  z
  =-\frac{s_\Lambda(\beta)}{s_\Lambda(\alpha)}e^{\ell\gamma},
  \qquad
  \frac 1 {1-z}
  =-\frac{s_\Lambda(\alpha)}{s_\Lambda(\gamma)}e^{\ell\beta},
  \qquad
  \frac{z-1} z
  =-\frac{s_\Lambda(\gamma)}{s_\Lambda(\beta)}e^{\ell\alpha},
  \end{align*}
  and their multiplicative inverses.
  \end{corollary}

  As for lightlike tetrahedra, using the symmetries \eqref{eq:permutmatrix}, we may always choose two of the parameters $\alpha,\beta,\gamma$ in Proposition \ref{prop:idealtet} to be positive. For $\Lambda=1$, due to periodicity, we can further choose $0<|\alpha|,|\beta|,|\gamma|<\pi$. We then obtain the following parametrization of an ideal tetrahedron that is the counterpart of Proposition \ref{prop:tetpara}.
  
  \begin{proposition} \label{prop:idealpara}
  The vertices in Proposition \ref{prop:idealtet} define an ideal tetrahedron in $\Y_\Lambda$ for all $\alpha,\beta,\gamma$ with $\alpha+\beta+\gamma=0$. Up to isometries, any ideal tetrahedron $I\subset\Y_\Lambda$ admits a global parametrization
  \begin{align*}
 I=\Big\{y(t,r,\theta)\in\Y_\Lambda \mid t\geq t(r,\theta),\; 0\leq r\leq r(\theta),\;  -\alpha\leq\theta\leq 0\, \Big\},
  \end{align*}
  where
  \begin{align*}
  &y(t,r,\theta)=\frac{1}{t}\left(\begin{matrix} t^2+|z(r,\theta)|^2 & z(r,\theta) \cr \bar z(r,\theta) & 1 \end{matrix}\right), & & z(r,\theta)=re^{\ell(\theta-\beta)}-\frac{s_\Lambda(\beta)}{s_\Lambda(\alpha)}e^{\ell\gamma},
  \\
  & t(r,\theta)=\Big(\frac{s_\Lambda(\theta-\gamma)}{s_\Lambda(\alpha)}r-r^2\Big)^{1/2}, & & r(\theta)=\frac{s_\Lambda(\beta)}{s_\Lambda(\alpha)}\frac{s_\Lambda(\gamma)}{s_\Lambda(\theta-\beta)},
  \end{align*}
  with $\alpha,\beta>0$ for all $\Lambda$ and $\alpha+\beta<\pi$ for $\Lambda=1$.
  \end{proposition}
  
  \begin{proof}
  Let $y_1,y_2,y_3,y_4\in\Y_\Lambda$ be given as in Proposition \ref{prop:idealtet}, and choose lifts $y_1',y_2',y_3',y_4'\in\R^4\subset\Mat(2,\bbC_\Lambda)$. Up to isometries (possibly reversing orientation), we can choose 
  \begin{align}\label{eq:yprimedef}
  &y'_1=\left(\begin{matrix} 1 & 0 \cr 0 & 0\end{matrix}\right),
  \qquad
  y'_2=\left(\begin{matrix} 0 & 0 \cr 0 & 1 \end{matrix}\right),
  \qquad
  y'_3=\left(\begin{matrix} 1 & 1 \cr 1 & 1 \end{matrix}\right),\\
 &y'_4=\left(\begin{matrix}\frac{s_\Lambda(\beta)^2}{s_\Lambda(\alpha)^2} & -\frac{s_\Lambda(\beta)}{s_\Lambda(\alpha)}e^{\ell\gamma} \cr -\frac{s_\Lambda(\beta)}{s_\Lambda(\alpha)}e^{-\ell\gamma} & 1 \end{matrix}\right), \qquad \alpha,\beta>0.\nonumber
  \end{align}
  
  We consider the convex cone in $\R^4$ spanned by lifts of the vertices $y_i\in \Y_\Lambda\subset\RP^3$ to vectors $y'_i\in\R^4$. This takes the form
  \begin{align} \label{eq:tetformula}
I'=\Big\{y'=\sum_{i=1}^4 b_i y'_i\mid b_i\geq 0, \sum_{i=1}^4 b_i\neq 0\Big\}.
  \end{align}
  This projects to a convex tetrahedron in $\Y_\Lambda$ if and only if $\langle y',y'\rangle_\Lambda<0$ for all $y'\in I'$, and this condition is satisfied for all $\alpha,\beta>0$ and $\gamma=-\alpha-\beta$.

  Any point $y\in \Y_\Lambda$ that is connected to $y_1$ by a spacelike geodesic can be parametrized as
  \begin{align}\label{eq:horzykpar}
  y(t,z)=\frac{1}{t}\left(\begin{matrix} t^2+|z|^2 & z \cr \bar z & 1 \end{matrix}\right),\qquad\text{with}\qquad t>0, \; z\in \bbC_\Lambda.
  \end{align}
  The points on $\partial_\infty \Y_\Lambda$ that are connected to $y_1$ by a spacelike geodesic are obtained from \eqref{eq:horzykpar} as the limit  $t\to 0$. 
  Note also that for all $z\in\bbC_\Lambda$, the map
  $g_z:\R\to \Y_\Lambda$,  $t\mapsto y(e^s, z)$ is a spacelike geodesic in $\Y_\Lambda$, parametrized by arc length and with $g_z(\infty)=y_1$. 
  This follows because  $g_z$ parametrizes the intersection of the image of a plane in $\R^4$ under the map \eqref{eq:ymatrixrp} with the set of matrices of unit determinant and because  $d(g_z(s), g_z(s'))={|s-s'|}$ by Proposition \ref{prop:geodesdist}.  
  Hence, we can view the sets
  \begin{align*}
  H_t(y_1)=\Big\{y(t,z)\mid z\in \bbC_\Lambda\Big\},
  \end{align*}
  for fixed $t>0$ as
   generalized horocycles based at $y_1\in \partial_\infty \Y_\Lambda$. For $\Lambda=1$, they coincide with the usual horocycles in $\bbH^3$.
  
  The edge geodesic through $y_1$ and $y_j$ is obtained by setting $B_k=0$ for $k\notin\{1,j\}$ in \eqref{eq:tetformula}. 
  By comparing the resulting expression with \eqref{eq:horzykpar}, one finds that this geodesic intersects each horocycle $H_t(y_1)$ in a unique point  $y(t,z_j)$ with $z_j$ given by
  \begin{align*}
  z_2=0,\qquad z_3=1,\qquad z_4=-\frac{s_\Lambda(\beta)}{s_\Lambda(\alpha)}e^{\ell\gamma}.
  \end{align*}
  More generally, a comparison of the parametrizations \eqref{eq:tetformula} and \eqref{eq:horzykpar} shows that any geodesic through $y_1$ that intersects the ideal tetrahedron $I$ intersects each horocycle $H_t(y_1)$ in a unique point $y(t,z)$ with $z$ given by
  \begin{align}\label{eq:zpara}
  z(r,\theta)=re^{\ell(\theta-\beta)}-\frac{s_\Lambda(\beta)}{s_\Lambda(\alpha)}e^{\ell\gamma},
  \end{align}
  with
  \begin{align*}
  0\leq r\leq r(\theta)=\frac{s_\Lambda(\beta)}{s_\Lambda(\alpha)}\frac{s_\Lambda(\gamma)}{s_\Lambda(\theta-\beta)},\qquad\qquad -\alpha\leq\theta\leq0.
  \end{align*}
  The intersection point of the geodesic $g_{r,\theta}: \R\to \Y_\Lambda$, $s\mapsto y(e^s, z(r,\theta))$ with the face opposite the vertex $y_1$ is obtained by setting $B_1=0$ in \eqref{eq:tetformula}.  Parameterizing $z$ as in \eqref{eq:zpara} and comparing with \eqref{eq:horzykpar}, we find that this intersection point is given by 
  \begin{align*}
  e^s=t(r,\theta)=\Big(\frac{s_\Lambda(\theta-\gamma)}{s_\Lambda(\alpha)}r-r^2\Big)^{1/2}.
  \end{align*}
  Inserting formula \eqref{eq:zpara} into the parametrization \eqref{eq:horzykpar} then completes the proof. 
  \end{proof}
  
  The parameters $\alpha,\beta,\gamma$ in Propositions \ref{prop:idealtet} and  \ref{prop:idealpara} also have a geometrical interpretation, namely as generalized dihedral angles at the edges of the ideal tetrahedron. Here, our convention for the dihedral angles uses one exterior angle, namely the biggest dihedral angle $\alpha+\beta$,  and two interior angles, $\alpha$ and $\beta$.
  For $\Lambda=1$ these are the usual dihedral angles between the faces of an ideal hyperbolic tetrahedron, up to the fact that one of them is external and given by $\pi-\theta$, where $\theta$ the usual interior dihedral angle.
  For $\Lambda=-1$ they give a {\em Lorentzian} angle between its faces, and for $\Lambda=0$ they are the length of the unique translation along the degenerate direction that relates adjacent faces. Using the global parametrization in Proposition \ref{prop:idealpara}, we obtain the analogue of Corollary \ref{cor:edgelength}.
  
  \begin{corollary} \label{cor:diheds}
  An ideal tetrahedron $I$ is determined up to isometries by its generalized dihedral angles.  If $I$ is parametrized as in Proposition \ref{prop:idealpara}, its dihedral angles are  $\alpha$, $\beta$ and $\alpha+\beta$, with opposite edges having equal dihedral angles.
  \end{corollary}
  
  \begin{figure}[t]
  \begin{center}
  \begin{tikzpicture}[scale=.5]
  \draw[color=black, line width=1pt] (0,0) circle (5);
  \draw[color=black, line width=1pt] (0,0) ellipse (5 and 2);
  \draw[color=black, line width=1pt] (-1.9,-1.9).. controls (-4.5,5.7) and (-2.5,5.7).. (1.9,1.9);
  \draw[color=black, line width=.5pt, ] (0,0)--(5,0);
  \draw[color=black, line width=.5pt,] (-1.05,4)--(.8,.8);
  \draw[color=black, ->, >=stealth, line width=1pt] (.2,0)--(1,.8);
  \draw[color=red, line width=1pt] (-1.9,-1.9)--(1.9,1.9); 
  \draw[color=red, line width=1pt] (-1.9,-1.9).. controls (0,0) and (4,0).. (5,0); 
  \draw[color=red, line width=1pt] (1.9,1.9).. controls (1,1) and (3,0).. (5,0); 
  \draw[color=red, line width=1pt] (-1.9,-1.9).. controls (0,0) and (-.5, 2).. (-1,4); 
  \draw[color=red, line width=1pt] (1.9,1.9).. controls (1,1) and (-.7,3.3).. (-1,4); 
  \draw[fill=black, color=black] (5,0) circle(.15);
  \draw[fill=black, color=black] (-1.9,-1.9) circle(.15);
  \draw[fill=black, color=black] (1.9,1.9) circle(.15);
  \draw[fill=black, color=black] (-1, 4) circle(.15);
  \draw[line width=1pt, color=black] (-3, -1.6).. controls (-2.9,-.9)..  (-2.25, -.8);
  \node at (-1.9,-1.9)[anchor=north east, color=black]{$y_1$};
  \node at (1.9,1.9)[anchor=south west, color=black]{$y_2$};
  \node at (5.1,0)[anchor=west,color=black]{$y_3$};
  \node at (-1,4)[anchor=south west, color=black]{$y_4$};
  \node at (-2.6,-.6)[anchor=east, color=black]{$\theta_{12}$};
  \node at (.7,-.1)[anchor=south west, color=black]{$\varphi_{12}$};
  \node at (-2,3)[color=black]{$f_3$};
  \node at (2.5,-1)[color=black]{$f_4$};
  \end{tikzpicture}
  \caption{ Exterior dihedral angle $\theta_{12}$ and shearing distance $\varphi_{12}$ in $\bbH^3$.}
  \end{center}
  \end{figure}
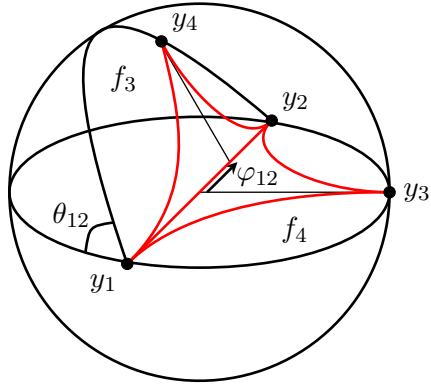
  
  Proposition \ref{prop:idealpara} and Corollary \ref{cor:diheds} show that the dihedral angles of an ideal tetrahedra play an analogous role to the edge lengths of lightlike tetrahedra. It is also possible to give a geometric interpretation for the ratios of their generalized sine functions as shearing distances along edges.

  We define the shearing distance along an edge $e_{ij}$ as the signed arc length $\varphi_{ij}$ between the orthogonal projections of $y_k$ and $y_l$ on $e_{ij}$, for all distinct $i,j,k,l\in \{1,2,3,4\}$. The sign of $\varphi_{ij}$ is taken positive (resp. negative) if the orientations of $e_{ij}$ induced (i) by the face opposite $y_k$ and (ii) by moving from $\pi_{ij}(y_k)$ to $\pi_{ij}(y_l)$ agree (resp. disagree), see Figure \ref{fig:signshearing}.

  \begin{proposition}\label{prop:sheardist} Let $I\subset \Y_\Lambda$ be an ideal tetrahedron with vertices $y_1,y_2,y_3,y_4$ parametrized as in Proposition  \ref{prop:idealtet}. Then the shearing distance $\varphi_{ij}$ at the edge $e_{ij}$ is given by
  $$
  2 \cosh(\varphi_{ij})=|z_{ij}|+|z_{ij}|^\inv,
  $$
   where $|z_{ij}|=|z_{ji}|$ and
  \begin{align*}
  &|z_{12}|=|z_{34}|=\left|\frac{s_\Lambda(\beta)}{s_\Lambda(\alpha)}\right|,
  & & |z_{31}|=|z_{34}|=\left| \frac{s_\Lambda(\alpha)}{s_\Lambda(\gamma)}\right|, 
  & & |z_{23}|=|z_{14}|=\left|\frac{s_\Lambda(\gamma)}{s_\Lambda(\beta)}\right|.
  \end{align*}
  \end{proposition}
  
  \begin{proof}  Denote by  $B_{ij}$  the unique isometry 
  with $B_{ij}\rhd\infty=y_i$, $B_{ij}\rhd 0=y_j$ and $B_{ij}\rhd 1=y_k$ from Proposition \ref{lem3pt}, where   $(y_i,y_j,y_k)$ is positively ordered with respect to the orientation of $I$. 
  Then the orthogonal projection of $y_k$ on $e_{ij}$ is given by $\pi_{ij}(y_k)=B_{ij}^{\inv }\rhd \mathbb 1$ and the orthogonal projection of the remaining vertex $y_l$ by
  $\pi_{ij}(y_l)=B_{ji}^\inv\rhd\mathbb 1$. 
  Suppose $B_{ij}$ is normalized with $|\det(B_{ij})|=1$. Then by Proposition \ref{prop:geodesdist}   the shearing distance $\varphi_{ij}$ satisfies
  \begin{align*}
  2\cosh(\varphi_{ij})=|\tr(B_{ij}^\inv B_{ji}(B_{ij}^\inv B_{ji})^\dagger)|.
  \end{align*}
  The claim then follows by computing the matrices $B_{ij}$ from the parametrization of the vertices in Proposition  \ref{prop:idealtet}. 
  \end{proof}
  
  \begin{figure}[t]
  \begin{center}
  \begin{tikzpicture}[scale=.6]
  \draw[line width=1pt, color=black] (0,-3)--(0,3);
  \draw[line width=.5pt, color=black, style=dashed] (0,1)--(-4,1);
  \draw[line width=.5pt, color=black, style=dashed] (0,-1)--(4,-1);
  \draw[line width=.7pt, ->,>=stealth, color=red] (.3,-1)--(.3,1);
  \node at (0,3)[anchor=south]{$y_j$};
  \node at (0,-3)[anchor=north]{$y_i$};
  \node at (4,-1)[anchor=west]{$y_l$};
  \node at (-4,1)[anchor=east]{$y_k$};
  \node at (0,0)[anchor=east]{$e_{ij}$};
  \node at (1.3,.1){$f_k$};
  \node at (-1.8,.1){$f_l$};
  \draw[color=black, line width=1pt] (0,-3)--(-4,1); 
  \draw[color=black, line width=1pt] (0,3)--(4,-1); 
  \draw[color=black, line width=1pt] (0,3)--(-4,1); 
  \draw[color=black, line width=1pt] (0,-3)--(4,-1); 
  \draw [red,thick,domain=0:270, <-, >=stealth] plot ({.6*cos(\x)+1.3}, {.6*sin(\x)+.1});
  \draw [black,thick,domain=0:270, <-, >=stealth] plot ({.6*cos(\x)-1.8}, {.6*sin(\x)+.1});
  \node at (0,-4.5){$\varphi_{ij}>0$};
  \begin{scope}[shift={(11.5,0)}]
  \draw[line width=1pt, color=black] (0,-3)--(0,3);
  \draw[line width=.5pt, color=black, style=dashed] (0,-1)--(-4,-1);
  \draw[line width=.5pt, color=black, style=dashed] (0,1)--(4,1);
  \draw[line width=.7pt, <-,>=stealth, color=red] (.3,-1)--(.3,1);
  \node at (0,3)[anchor=south]{$y_j$};
  \node at (0,-3)[anchor=north]{$y_i$};
  \node at (4,1)[anchor=west]{$y_l$};
  \node at (-4,-1)[anchor=east]{$y_k$};
  \node at (0,0)[anchor=east]{$e_{ij}$};
  \node at (1.3,.1){$f_k$};
  \node at (-1.8,.1){$f_l$};
  \draw[color=black, line width=1pt] (0,-3)--(-4,-1); 
  \draw[color=black, line width=1pt] (0,3)--(4,1); 
  \draw[color=black, line width=1pt] (0,3)--(-4,-1); 
  \draw[color=black, line width=1pt] (0,-3)--(4,1); 
  \draw [red,thick,domain=0:270, <-, >=stealth] plot ({.6*cos(\x)+1.3}, {.6*sin(\x)+.1});
  \draw [black,thick,domain=0:270, <-, >=stealth] plot ({.6*cos(\x)-1.8}, {.6*sin(\x)+.1});
  \node at (0,-4.5){$\varphi_{ij}<0$};
  \end{scope}
  \end{tikzpicture}
  \end{center}
  \caption{Sign conventions for the shearing distance $\varphi_{ij}$. 
  }
  \label{fig:signshearing}
  \end{figure}
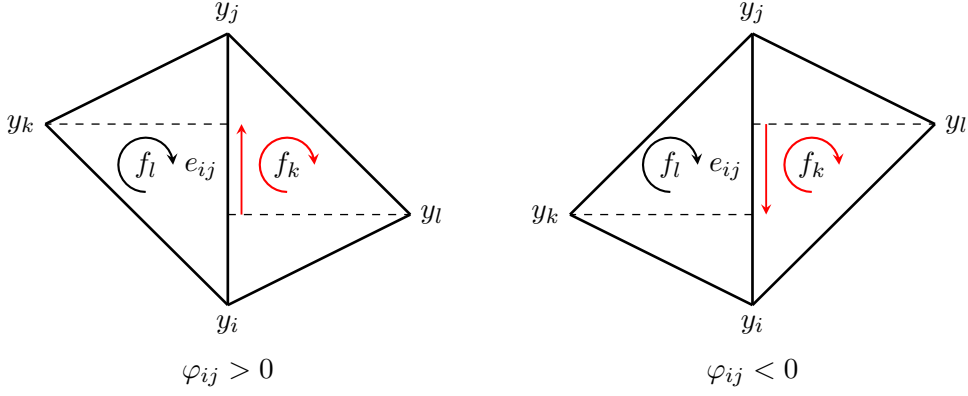
  
  As in the case of lightlike tetrahedra, the cross-ratios or shape parameters of a generalized ideal tetrahedron can also be characterized in terms of its symmetries.

  \begin{proposition}\label{prop:geomchar} Let $I\subset\Y_\Lambda$ be an ideal tetrahedron. Denote by $y_{ij}$ the geodesic segment between $y_i$ and $y_j$, oriented from $y_i$ to $y_j$. There exists a unique isometry $T_{ij}\in\PGL^+(2,\bbC_\Lambda)$ that stabilizes $y_{ij}$, together with its orientation, and maps one opposite vertex to the other. With the parametrization of Proposition \ref{prop:idealtet}, we have
  \begin{align*}
  T_{ij}=B_{ij}\begin{pmatrix} z_{ij} & 0 \cr 0 &1 \end{pmatrix}B_{ij}{}^\inv,
  \end{align*}
  where $z_{ij}=z_{ji}$ is given by
  \begin{align*}
  &z_{12}=z_{34}=-\frac{s_\Lambda(\beta)}{s_\Lambda(\alpha)}e^{\ell\gamma},
  & & z_{31}=z_{34}=-\frac{s_\Lambda(\alpha)}{s_\Lambda(\gamma)}e^{\ell\beta},
  & & z_{23}=z_{14}=-\frac{s_\Lambda(\gamma)}{s_\Lambda(\beta)}e^{\ell\alpha},
  \end{align*}
   and where $B_{ij}\in\PGL^+(2,\bbC_\Lambda)$  
    maps $\infty,0,1\in\partial_\infty\Y_\Lambda$ to  $y_i,y_j,y_k$,  respectively, with the order of $(y_i,y_j,y_k)$ induced by the orientation of $I$.
  \end{proposition}
  \begin{proof}
  Given an isometry $B_{ij}$
  with $B_{ij}\rhd\infty=y_i$, $B_{ij}\rhd 0=y_j$ and $B_{ij}\rhd 1=y_k$, define $z_{ij}=B_{ij}^{\inv}\rhd y_l\in\bbC_\Lambda \rmP^1$ as the preimage of the remaining vertex $y_l$. The projective matrix
  \begin{align*}
  \begin{pmatrix} z_{ij} & 0 \cr 0 & 1 \end{pmatrix}\in\PGL^+(2,\bbC_\Lambda)\end{align*}
  then stabilizes both $\infty$ and $0$ in $\bbC_\Lambda P^1$ and maps $1$ to $z_{ij}$.  It follows that the isometry 
  \begin{align*}
  T_{ij}=B_{ij}\begin{pmatrix} z_{ij} & 0 \cr 0 &1 \end{pmatrix}B_{ij}{}^\inv,
  \end{align*}
  stabilizes $y_i$ and $y_j$ and maps $y_k$ to $y_l$.
  
  From Proposition \ref{prop:idealtet} we obtain
  \begin{align*}
  z_{12}=-\frac{s_\Lambda(\beta)}{s_\Lambda(\alpha)}e^{\ell\gamma}.
  \end{align*}
  The other parameters $z_{ij}$ are obtained by computing the isometries $B_{kl}^{-1}\circ B_{ij}$, for instance
  \begin{align*}
   B_{31}^\inv\circ B_{12}:\left\{
    \begin{aligned}
          \infty \,&\, \mapsto \;y_1 \;\mapsto \;  0,  \\
          0 \,&\, \mapsto \;y_2 \;\mapsto \; 1,  \\
          1 \,&\, \mapsto \; y_3 \;\mapsto \; \infty,  \\
          z_{12} \,&\, \mapsto \; y_4  \;\mapsto \; z_{23}  = \tfrac{1}{1-z_{12}}, \\
      \end{aligned}
   \right.
   \qquad
   B_{34}^\inv\circ B_{12}:\left\{
    \begin{aligned}
          \infty \,&\, \mapsto \; y_1 \;\mapsto \;1,  \\
          0 \,&\, \mapsto \; y_2  \;\mapsto \;z_{34} = z_{12}, \\
          1 \,&\, \mapsto \; y_3  \;\mapsto \;\infty, & \\
          z_{12} \,&\, \mapsto \; y_4  \;\mapsto \; 0. &  \\
      \end{aligned}
   \right.
  \end{align*}
  The claim then follows from the identity  $\alpha+\beta+\gamma=0$.
  \end{proof}

  Corollary \ref{cor:diheds}  and Proposition \ref{prop:sheardist} show that the cross-ratios of an ideal tetrahedron from Corollary \ref{cor:crossratio} have a direct geometric interpretation that generalizes the one of ideal tetrahedra in $\bbH^3$.
  Their arguments are generalized dihedral angles between faces, and their moduli shearing distance along edges. 
  
  They are the counterparts of Corollary \ref{cor:edgelength} and Proposition \ref{prop:lorentz_angle} for lightlike tetrahedra in $\X_\Lambda$, which state that the arguments of their shape parameters determine the edge lengths and their moduli the Lorentzian angle between the internal planes of a lightlike tetrahedron.
  Proposition \ref{prop:geomchar}, which characterizes the cross-ratios of an ideal tetrahedron in terms of its symmetries, is the counterpart of Proposition \ref{eq:geomlightlike} for lightlike tetrahedra.
  
  We have seen in Proposition \ref{prop:tetpara} that given parameters $\alpha,\beta,\gamma$, satisfying $\alpha+\beta+\gamma=0$, there exists a lightlike tetrahedron with edge lengths $|\alpha|,|\beta|,|\gamma|$, unique up to isometries. Similarly, under the same assumptions, Proposition \ref{prop:idealtet} proves the existence of a generalized ideal tetrahedron with generalized dihedral angles $|\alpha|,|\beta|,|\gamma|$, again unique up to isometries. The following theorem gives a geometric interpretation for this correspondence between lightlike and ideal tetrahedra in terms of the projective duality of Sections \ref{sec:projdual1} and \ref{subsec:ideldual}.
  
Note, however, that this correspondence is \emph{not} given by the duality between convex sets in $\X_\Lambda$ and $\Y_\Lambda$ from \cite{Fillastre-Seppi} discussed in Section \ref{sec:projdual1}. As explained in Section \ref{sec:projdual1}, the dual of a convex set in $\X_\Lambda$ or $\Y_\Lambda$ can be characterized as the set of spacelike geodesic planes that do not intersect the convex set. Here, instead, we characterize lightlike tetrahedra in $\X_\Lambda$ as the sets of spacelike geodesic planes in $\Y_\Lambda$ that \emph{do intersect} ideal tetrahedra in two specified pairs of opposite edges. Conversely, ideal tetrahedra in $\Y_\Lambda$ correspond to spacelike geodesic planes in $\X_\Lambda$ that intersect a lightlike tetrahedron in all pairs of opposite edges except its longest edge pair.

\begin{theorem}\label{prop:dualprop} The projective duality from Section \ref{sec:projdual1} identifies 
 a   lightlike tetrahedron in $\X_\Lambda$ with the set of spacelike planes in $\Y_\Lambda$ that intersect an ideal tetrahedron along two pairs of opposite edges. It identifies 
an ideal tetrahedron in $\Y_\Lambda$ with the set of spacelike planes in $\X_\Lambda$ that intersect a lightlike tetrahedron along its shortest edges.
\end{theorem}

\begin{proof}
  This follows from the parameterization of lightlike tetrahedra and ideal tetrahedra as projections of the convex cones 
  \begin{align*}
    L'=\Big\{x'=\sum_{i=1}^4 a_ix_i'\mid a_i\geq 0,\;\sum_{i=1}^4 a_i\neq 0\Big\},
    &&
    I'=\Big\{y'=\sum_{i=1}^4 b_i y'_i\mid b_i\geq 0, \sum_{i=1}^4 b_i\neq 0\Big\},
  \end{align*}
  with  the vertices $x'_i$ and $y'_j$ given by \eqref{eq:xprimedef} and \eqref{eq:yprimedef}.
 By assumption, we have $\alpha+\beta+\gamma = 0$ with $\alpha,\beta>0$ and $\alpha+\beta<\pi$ for $\Lambda = 1$. This implies $\langle x'_i,y'_j\rangle=0$ for $i\neq j$ and
  \begin{align}\label{eq:scalar}
    \langle x'_1,y'_1\rangle = -s_\Lambda(\alpha) < 0,
    & &
    \langle x'_2,y'_2\rangle = -s_\Lambda(\beta) < 0,
    \cr 
    \langle x'_3,y'_3\rangle = -s_\Lambda(\gamma) > 0,
    & &
    \langle x'_4,y'_4\rangle = -s_\Lambda(\gamma) > 0.
  \end{align} 
In particular, the spacelike plane in $\Y_\Lambda$ dual to any point in the lightlike tetrahedron must intersect the ideal tetrahedron: given any $x'\in L'$ there exists $y'\in I'$ such that $\langle x',y'\rangle = 0$.
Such spacelike planes, however, cannot not intersect the pair of edges $e_{12}$ and $e_{34}$ in $I'$:  If $\{i,j\}=\{1,2\}$ or $\{i,j\}=\{3,4\}$ and
  $y' = b_i y'_i + b_j y'_j $  with $b_i,b_j\geq 0$ and $b_i+b_j\neq 0$ we have $\langle x', y'\rangle < 0$ for all $x'\in L'$.  For all other combinations of  $i$ and $j$, there are $b_i,b_j\geq 0$ and $b_i+b_j\neq 0$ for which $\langle x',y'\rangle=0$. By  Proposition \ref{prop:tetpara} and Corollary \ref{cor:edgelength} the edges $e_{12}$ and $e_{34}$ are the longest edges of the lightlike tetrahedron. 
The proof of the second statement is analogous.
\end{proof}

Although this correspondence is \emph{not the duality of convex sets} from Section \ref{sec:projdual1}, it still identities faces and vertices of a lightlike tetrahedron with faces and vertices of a lightlike tetrahedron.  Geodesics through two vertices or on two faces of a  lightlike tetrahedron  are identified with geodesics on the two dual faces or though the two dual vertices, respectively. In this sense, lightlike tetrahedra in $\X_\Lambda$ and ideal tetrahedra in $\Y_\Lambda$ are projectively dual.

  \section{Volumes of lightlike and ideal tetrahedra}
  In this section we derive formulas for the volumes of lightlike tetrahedra in $\X_\Lambda$ and of generalized ideal tetrahedra in $\Y_\Lambda$ as functions of their edge lengths and dihedral angles, respectively.  These formulas are obtained by direct integration of the volume forms on $\X_\Lambda$ and on $\Y_\Lambda$, defined here uniquely up to global rescaling as the $\PGL^+(2,\bbC_\Lambda)$-invariant 3-forms on each space.

  \subsection{Volumes of ideal tetrahedra} 
  
  We start with the computation of volumes of generalized ideal tetrahedra in $\Y_\Lambda$. This is technically much simpler to compute and serves as a guide for the computation of the lightlike volume below. For $\Lambda=1$, it includes the Milnor-Lobachevsky formula \cite{Mi}, which gives the volume of a hyperbolic ideal tetrahedron $I$ as
  \begin{align}\label{eq:volstandard}
  \vol(I)=\frac 1 2 \Big(\Cl(2\alpha)+\Cl(2\beta)+\Cl(2\gamma)\Big)=\Lob(\alpha)+\Lob(\beta)+\Lob(\gamma).
  \end{align}
  Here $\alpha,\beta$ and $\gamma=\pi-(\alpha+\beta)$ are the interior dihedral angles of the tetrahedron, 
  $\Cl:\R\to \R$ is the Clausen function of order two and $\Lob:\R\to \R$ the closely related Lobachevsky function.
  
  Note that taking the exterior dihedral angle for $\gamma$ instead and setting $\gamma=-(\alpha+\beta)$ in \eqref{eq:volstandard} gives the same result due to periodicity. Hence,  \eqref{eq:volstandard} remains valid for our conventions on dihedral angles, where $\gamma=-(\alpha+\beta)$ (see Proposition \ref{prop:geomchar}). 
  
  We will now show that the volume formulas for generalized ideal tetrahedra $I\subset\Y_\Lambda$ 
  can be computed for all values of $\Lambda$ simultaneously and are simple generalizations of formula \eqref{eq:volstandard}, in which $\Lambda$ appears as a deformation parameter. 
  
  The standard computation of the volume for an ideal hyperbolic tetrahedron, due to Milnor \cite{Mi} and based on the work by Lobachevsky, proceeds by subdividing the ideal tetrahedron in three sub-tetrahedra with a higher degree of symmetry. This method can be extended to generalized ideal tetrahedra. 
  However, for simplicity and to exhibit the analogies with the computation of the volume of lightlike tetrahedra in $\X_\Lambda$, we compute the volume
  by a different method that does not require a subdivision, namely with the parametrization from Proposition \ref{prop:idealpara}.

  \begin{theorem}\label{prop:volideal}
  The volume of an ideal tetrahedron $I\subset\Y_\Lambda$  is given by
  \begin{align*}
  \vol(I)=\frac{1}{2}\Big(\Cl_\Lambda(2\alpha)+\Cl_\Lambda(2\beta)+\Cl_\Lambda(2\gamma)\Big),
  \end{align*}
  where $\alpha,\beta$ and $\gamma=-(\alpha+\beta)$ are its generalized dihedral angles from Proposition \ref{prop:idealtet} and $\Cl_\Lambda$ is the generalized Clausen function defined by
  \begin{align*}
  \Cl_\Lambda(\alpha):=-\int_0^\alpha d\theta\log\left |2s_\Lambda(\tfrac \theta 2)\right|.
  \end{align*}
  \end{theorem}
  \begin{proof} 
   To compute the volume, we express the volume form on $I$ in terms of the coordinates $r,\theta, t$ from Proposition \ref{prop:idealpara} and use the identification \eqref{eq:ymatrixrp} of $\R^4$ with the set of matrices $Y\in \mathrm{Mat}(2,\bbC_\Lambda)$ satisfying $Y^\dag=Y$.  For $\Lambda=\pm 1$,   the volume form on $ I $ is then induced by the semi-Riemann metric \eqref{eq:yform} on $\R^4$ via \eqref{eq:ymatrixrp} and the parametrization in Proposition \ref{prop:idealpara}.  A direct computation shows that it is 
  \begin{align}\label{eq:volform}
  d\vol=\frac{r}{t^3}\,dt\wedge dr\wedge d\theta.
  \end{align}
  For $\Lambda=0$ the bilinear form \eqref{eq:yform}  is degenerate and does not induce a volume form on $\Y_\Lambda$. Nevertheless, the volume form on $\Y_\Lambda$  can be defined, up to real rescaling, as the unique 3-form on $\Y_\Lambda$  invariant under the action of  $\PGL^+(2,\bbC_\Lambda)$.   It is again given by  \eqref{eq:volform}.
  The volume of  $ I $ is then obtained from \eqref{eq:volform} and the parametrization in Proposition \ref{prop:idealpara}
  \begin{align*}
  \vol( I )&=\int_0^\alpha d\theta \int_{0}^{r(\theta)}\!\!\!\!dr\int_{t(r,\theta)}^\infty \!\!\!\!dt\; \frac{r}{t^3} 
  =-\frac{1}{2}\int_0^\alpha d\theta\int_{0}^{r(\theta)} \!\! \frac{dr}{r-\frac{s_\Lambda(\alpha+\beta-\theta)}{s_\Lambda(\beta)}}
  \cr
  &=-\frac{1}{2}\int_0^\alpha d\theta\;\log\Big|\frac{s_\Lambda(\theta)}{s_\Lambda(\alpha+\beta-\theta)}\frac{s_\Lambda(\alpha-\theta)}{s_\Lambda(\theta+\beta)}\Big|
  \cr
  &=-\int_0^\alpha d\theta\log|2s_\Lambda(\theta)|
  -\int_{0}^{\beta} d\theta\log|2s_\Lambda(\theta)|
  +\int_0^{\alpha+\beta} \!\!\!\!d\theta\log|2s_\Lambda(\theta)|
  \cr
  &=\frac{1}{2}\Big(\Cl_\Lambda(2\alpha)+\Cl_\Lambda(2\beta)-\Cl_\Lambda(2(\alpha+\beta))\Big).
  \end{align*}
  \end{proof}

  \subsection{Volumes of lightlike tetrahedra}

  We now consider the volumes of lightlike tetrahedra $ L \subset\X_\Lambda$. These volumes can be computed in a similar way
  from the parametrization in Proposition \ref{prop:tetpara}.  By a straightforward change of coordinates, this yields a parametrization in which both, the lightlike tetrahedron and its volume form become particularly simple.

  \begin{theorem}\label{prop:dualvol}
  The volume of a lightlike tetrahedron $ L \subset\X_\Lambda$  is  
  \begin{align*}
  \vol( L )=&\frac{1}{2\Lambda}\Big(\Cl_\Lambda(2\alpha)+\Cl_\Lambda(2\beta)+\Cl_\Lambda(2\gamma)\Big)
  \cr
  &\qquad+\frac 1 {\Lambda}\;\Big( \alpha\log|s_\Lambda(\alpha)|+\beta\log|s_\Lambda(\beta)|+\gamma \log|s_\Lambda(\gamma)|\Big),\quad & &\Lambda=\pm 1,\\[+2ex]
  \vol( L )=&-\frac 1 3 \alpha\beta\gamma, & &\Lambda=0,
  \end{align*}
  where $\alpha$, $\beta$ and $-\gamma=\alpha+\beta$  are the edge lengths of $ L $ and 
  $\Cl_\Lambda$ is the generalized Clausen function from Theorem \ref{prop:volideal}.
  \end{theorem}
  
  \begin{proof}  Starting from the parametrization in Proposition \ref{prop:tetpara} and setting
  \begin{align*}
  A=\frac{\sin(s)-\sin(t)}{2\cos(t)},\qquad\qquad B=\frac{\sin(s)+\sin(t)}{2\cos(t)},
  \end{align*}
  we can rewrite the matrix $X(A,B)$ in Proposition \ref{prop:tetpara}  as 
  \begin{align*}
  &X(s,t)=\frac{\sin(s)-\sin(t)}{2\cos(t)}X_{41}+\frac{\sin(s)+\sin(t)}{2\cos(t)}X_{42}+\frac{\sin(s)-\cos(t)}{\cos(t)}X_{43}.
  \end{align*}
  This yields the global parametrization
  \begin{align}\label{eq:lighttetpar}
  & L =\Big\{x(r,s,t)\mid 0\leq r\leq r(s,t)\leq\pi,\; |t|\leq s\leq \tfrac{\pi}{2}-|t|,\; -\tfrac{\pi}{4}\leq t\leq \tfrac{\pi}{4}\Big\}
  \end{align}
  with
  \begin{align}\label{eq:geodinttvol}
  &x(r,s,t)= \left(\begin{matrix}c_\Lambda(r)+\ell\frac{\cos(t)}{\cos(s)}s_\Lambda(r) & \ell\frac{\sin(t)-\sin(s)}{\cos(s)}s_\Lambda(r) \cr \ell\frac{\sin(s)+\sin(t)}{\cos(s)}s_\Lambda(r) & c_\Lambda(r)-\ell\frac{\cos(t)}{\cos(s)}s_\Lambda(r) \end{matrix}\right),\\
  \label{eq:rint}
  &r(s,t)=ct_\Lambda^\inv\left(\frac{a\sin(t)+b\cos(t)+c\sin(s)}{d\cos(s)}\right),
  \end{align}
  and
  \begin{align}\label{eq:rint2}
  &a=\frac 1 2\left(\frac{s_\Lambda(\alpha)}{s_\Lambda(\beta)}-\frac{s_\Lambda(\beta)}{s_\Lambda(\alpha)}\right),  & &c=\frac 1 2\left(\frac{s_\Lambda(\alpha)}{s_\Lambda(\beta)}+\frac{s_\Lambda(\beta)}{s_\Lambda(\alpha)}\right),\\
  &b=c_\Lambda(\alpha+\beta), & &d=s_\Lambda(\alpha+\beta).\nonumber
  \end{align}
  To express the volume form on $ L $ in terms of the coordinates $r,s,t$, we use the identification \eqref{eq:xmatrixrp} of $\R^4$ with the set of matrices $X\in \mathrm{Mat}(2,\bbC_\Lambda)$ satisfying $X^\circ=X$. For $\Lambda=\pm 1$ the volume form on $\X_\Lambda$ is the 3-form on $\mathrm{AdS}_3$ or $\mathrm{dS}_3$ induced by the semi-Riemannian metric $\langle\cdot,\cdot \rangle_{2,0,2}$ or $\langle\cdot,\cdot \rangle_{1,0,3}$  on $\R^4$, respectively.  For $\Lambda=0$, it is the standard 3-form on $\R^3$. In all three cases, the induced volume form on $ L $ is obtained from the identification \eqref{eq:xmatrixrp} and the parametrization \eqref{eq:geodinttvol} and reads
  \begin{align*}
  d\vol=\frac{s_\Lambda(r)^2}{\cos(s)^2} \, dt\wedge ds\wedge dr.
  \end{align*}
  To compute the volume of the lightlike tetrahedron $ L $, we integrate this volume form over the parameter range in \eqref{eq:lighttetpar}. 
  For $\Lambda=0$, this is a direct and simple computation 
  \begin{align*}
  \vol( L )&=\int_{-\tfrac \pi 4}^{\tfrac \pi 4} dt \int_{|t|}^{\frac{\pi}{2}-|t|}\!\!\!\!\!\! ds\int_0^{r(s,t)}\!\!\!\!\!\! dr\; \frac{s_\Lambda(r)^2}{\cos(s)^2}=\frac 1 3 \int_{-\tfrac \pi 4}^{\tfrac \pi 4}dt \int_{|t|}^{\tfrac \pi 2-|t|} \!\!\!\!\!\!ds\; \frac{r(s,t)^3}{\cos(s)^2}\\
  &=\frac 1 3 \int_{0}^{\tfrac \pi 4} dt \int_{t}^{\tfrac \pi 2-t} \!\!\!\!\!\!ds\; \frac{r(s,t)^3+r(s,-t)^3}{\cos(s)^2}.
  \end{align*}
  Inserting expression \eqref{eq:rint} for $r(s,t)$ with $t_\Lambda(x)=x$ for $\Lambda=0$, we obtain
  \begin{align*}
  \vol( L )&=\frac {d^3} 3 \int_0^{\tfrac \pi 4} dt\int_t^{\tfrac\pi 2-t} \!\!\!\!\!\!ds\; \frac{\cos(s)}{(a\sin(t)+b\cos(t)+ c\sin(s))^3}\\
  &+\frac {d^3} 3 \int_0^{\tfrac \pi 4} dt\int_t^{\tfrac\pi 2-t} \!\!\!\!\!\!ds\;\frac{\cos(s)}{(-a\sin(t)+b\cos(t)+ c\sin(s))^3}\\
  &=\frac{d^3} {6c} \int_0^{\tfrac \pi 4} \frac{dt}{\cos(t)^2} \; \Bigg(\frac 1 {((a+c)\tan(t)+b)^2}+\frac 1 {((c-a)\tan(t)+b)^2}\Bigg)\\
  &-\frac{d^3} {6c} \int_0^{\tfrac \pi 4} \frac{dt}{\cos(t)^2} \; \Bigg(\frac 1 {(a\tan(t)+b+c)^2} +\frac 1 {(-a\tan(t)+b+c)^2}\Bigg)\\
  &=\frac 1 3 \alpha\beta(\alpha+\beta),
  \end{align*}
  where we used the substitution rule twice and in the last step inserted the expressions for $a,b,c,d$  from \eqref{eq:rint2} with $s_\Lambda(x)=x$ and $c_\Lambda(x)=1$ for 
  $\Lambda=0$.
  
  For $\Lambda=\pm 1$ the computation of the volume is more involved.
  Performing the integration over $r$ and splitting the integral over $t$ we obtain
  \begin{align*}
  &\vol( L )=\int_{-\tfrac \pi 4}^{\tfrac \pi 4}dt \int_{|t|}^{\frac{\pi}{2}-|t|}\!\!\!\! ds\int_0^{r(s,t)}\!\!\!\! dr\;\; \frac{s_\Lambda(r)^2}{\cos(s)^2}
  \\
  &=\frac{1}{4\ell^2}\int_{-\tfrac \pi 4}^{\tfrac \pi 4}dt \int_{|t|}^{\frac{\pi}{2}-|t|}\!\!\!\! ds\;\;\frac{s_\Lambda(2r(s,t))-2r(s,t)}{\cos(s)^2}\nonumber\\
  &=\frac{1}{4\ell^2}\int_{0}^{\tfrac \pi 4}dt \int_{t}^{\frac{\pi}{2}-t}\!\!\!\! ds\;\;\left(\frac{s_\Lambda(2r(s,t))-2r(s,t)}{\cos(s)^2}+ \frac{s_\Lambda(2r(s,-t))-2r(s,-t)}{\cos(s)^2}\right).\nonumber
  \end{align*}
  To integrate over $s$, we now use the indefinite integral
  \begin{align*}
  \int ds\;\;&\frac{s_\Lambda(2r(s,t))-2r(s,t)}{\cos(s)^2}=
  -2ct_\Lambda^\inv\Bigg(\frac{a\sin(t)+b\cos(t)+c\sin(s)}{d\cos(s)}\Bigg) \tan(s)
  \cr
  &
  \qquad\qquad-2ct_\Lambda^\inv\Bigg(\frac{(a\cos(t)-b\sin(t))^2-c\sin(s)(a\sin (t)+b\cos (t))-c^2}{d\sin (s) (a \cos(t)-b\sin(t))}\Bigg) \frac{a\tan(t)+b}{a-b\tan(t)}, \nonumber
  \end{align*}
  where $ct_\Lambda^\inv$ is the generalized inverse cotangent given by \eqref{eq:tandef}. 
  That the derivative of the right hand side with respect to $s$ is indeed the integrand of the left hand side follows by a direct but lengthy computation. The derivative of the term $\tan(s)$ on the right hand side gives the second term on the left. The first term on the left is obtained from the derivatives of the inverse generalized cotangents on the right hand side with the formulas
  \begin{align*}
  \frac {d} {dx}  ct_\Lambda^\inv(x) =-\frac 1 {x^2-\ell^2},\qquad \qquad s_\Lambda(2ct_\Lambda^\inv(x))=\frac{2x}{x^2-\ell^2},
  \end{align*}
   that follow from  \eqref{eq:trigids},  \eqref{eq:trigdev}   and \eqref{eq:tandef}. After some computations using trigonometric identities and inserting 
   expressions \eqref{eq:rint} and \eqref{eq:rint2} for $r(s,t)$  and $a,b,c,d$,  one then obtains the first term in the integrand on the left.

  To perform the integration over $s$, we insert this indefinite integral into the expression for $\vol( L )$. Simplifying the resulting terms with the addition formulas 
  \begin{align*}
  ct_\Lambda^\inv(x)+ct_\Lambda^\inv(y)=ct_\Lambda^\inv\Big(\frac{xy+\ell^2}{x+y}\Big),
  \end{align*}
  derived from  \eqref{eq:trigids} and \eqref{eq:tandef}, then yields
  \begin{align*}
  \vol( L )=\frac{1}{2\ell^2}\int_{0}^{\tfrac \pi 4}dt
   \Bigg[&\Big(\frac{a\tan(t)+b}{a-b\tan(t)}+\tan(t)\Big)ct_\Lambda^\inv\Big(\frac{(a+c)\tan(t)+b}{d}\Big)
  \cr
  -&\Big(\frac{a\tan(t)-b}{a+b\tan(t)}+\tan(t)\Big)ct_\Lambda^\inv\Big(\frac{a+c+b\tan(t)}{d\tan(t)}\Big)
  \cr
  +&\Big(\frac{a\tan(t)-b}{a+b\tan(t)}-\cot(t)\Big)ct_\Lambda^\inv\Big(\frac{(a+c)(1-\tan(t))+b(1+\tan(t)}{d(1+\tan(t))}\Big)
  \cr
  -&\Big(\frac{a\tan(t)+b}{a-b\tan(t)}-\cot(t)\Big)ct_\Lambda^\inv\Big(\frac{(a+c)(1+\tan(t))+b(1-\tan(t))}{d(1-\tan(t))}\Big)
  \cr
  -&\Big(\frac{a\tan(t)-b}{a+b\tan(t)}-\tan(t)\Big)ct_\Lambda^\inv\Big(\frac{b}{d}\Big)\\
  -&\Big(\frac{a\tan(t)+b}{a-b\tan(t)}-\frac{a\tan(t)-b}{a+b\tan(t)}\Big)ct_\Lambda^\inv\Big(\frac{a+b+c}{d}\Big)\Bigg].
  \end{align*}
  To simplify this integral further, we apply a change of variables,
  \begin{align*}
  \tan(s)=\frac{1-\tan(t)}{1+\tan(t)},
  \end{align*}
  to the third and fourth term to combine them with the first and second term, respectively.  After some further computations involving trigonometric identities we then obtain
  \begin{align*}
  \vol( L )=\frac{1}{2\ell^2}\int_{0}^{\tfrac \pi 4}\frac{dt}{\cos^2(t)}
  \Bigg[&\Big(\frac{1}{\tan(t)-1}-\frac{1}{\tan(t)-\frac{a}{b}}+\frac{1}{\tan(t)+\frac{a+b}{a-b}}\Big)
  ct_\Lambda^\inv\Big(\frac{(a+c)\tan(t)+b}{d}\Big)
  \cr
  -&\Big(\frac{1}{\tan(t)-1}-\frac{1}{\tan(t)+\frac{a}{b}}+\frac{1}{\tan(t)+\frac{a-b}{a+b}}\Big)ct_\Lambda^\inv\Big(\frac{a+c+b\tan(t)}{d\tan(t)}\Big)
  \cr
  +&\;\;\frac{1}{\tan(t)-\frac{a}{b}}ct_\Lambda^\inv\Big(\frac{a+b+c}{d}\Big)
  -\frac{1}{\tan(t)+\frac{a}{b}}ct_\Lambda^\inv\Big(\frac{a-b-c}{d}\Big)\Bigg].
  \end{align*}
  To perform the integration over $t$ we apply the changes of variables
  \begin{align*}
  ct_\Lambda(\theta)=\frac {(a+c)\tan(t)+b} d,\qquad ct_\Lambda(\theta)=\frac{a+c+b\tan(t)} {d\tan(t)}
  \end{align*}
  to the first and third terms and to the second and fourth terms in this expression, respectively. We then combine the resulting expressions, insert formulas \eqref{eq:rint2} for the variables $a,b,c,d$  and use the definition of the generalized trigonometric functions in terms of the exponential and the identities  \eqref{eq:trigids}.   After some computations this yields
  \begin{align*}
  \vol( L )
  &=\frac{1}{\ell}\int_{\beta}^{\alpha+\beta}\!\!\!\!\!\! d\theta \Bigg(\frac{\theta+\beta}{1-e^{-2\ell\theta}}-\frac{\theta}{1-e^{2\ell(\beta-\theta)}}\Bigg)
  \cr
  &+\frac{1}{\ell}\int_{0}^{\beta}\;d\theta \Bigg(\frac{\theta-\alpha}{1-e^{2\ell(\alpha+\beta-\theta)}}-\frac{\theta}{1-e^{2\ell(\beta-\theta)}}\Bigg)
  \cr
  &+\frac{1}{\ell}\int_{\beta}^{\alpha+\beta}\!\!\!\!\!\!d\theta \Bigg(\frac{\theta-\beta}{1-e^{-2\ell\theta}\frac{1+z^2}{1+\bar z^2}}-\frac{\theta}{1-e^{-2\ell(\beta+\theta)}\frac{1+z^2}{1+\bar z^2}}\Bigg)
  \cr
  &+\frac{1}{\ell}\int_{0}^{\beta}\;d\theta \Bigg(\frac{\theta+\alpha}{1-e^{-2\ell(\alpha+\beta+\theta)}\frac{1+z^2}{1+\bar z^2}}-\frac{\theta}{1-e^{-2\ell(\beta+\theta)}\frac{1+z^2}{1+\bar z^2}}\Bigg),
  \end{align*}
  where $z$ is the cross-ratio from Corollary \ref{cor:crossratio}. 
  
  The terms in the third and fourth line cancel, and the remaining terms can be recombined to
  \begin{align*}
  \vol( L )&=\frac{1}{\ell}\int_{0}^{\alpha+\beta}d\theta\frac{\theta+\beta}{1-e^{-2\ell\theta}}-\frac{1}{\ell}\int_{0}^{\alpha}d\theta \frac{\theta+\beta}{1-e^{-2\ell\theta}}-\frac{1}{\ell}\int_{0}^{\beta}d\theta\frac{\theta+\beta}{1-e^{-2\ell\theta}}
  \cr
  &+\frac{1}{\ell}\int^{\alpha+\beta}_{0}d\theta\frac{-\theta+\beta}{1-e^{2\ell\theta}}-\frac{1}{\ell}\int^{\alpha}_0 d\theta\frac{-\theta+\beta}{1-e^{2\ell\theta}}-\frac{1}{\ell}\int^{\beta}_{0}d\theta \frac{-\theta+\beta}{1-e^{2\ell\theta}}
  \cr
  &=\frac{1}{\ell^2}\int_{0}^{\alpha+\beta}d\theta\;  \theta \, ct_\Lambda(\theta)-\frac{1}{\ell^2}\int_{0}^{\alpha}d\theta\;  \theta\, ct_\Lambda(\theta)-\frac{1}{\ell^2}\int_{0}^{\beta}d\theta\;  \theta\, ct_\Lambda(\theta).
  \end{align*}
  To complete the computation of the volume it is now sufficient to note that
  \begin{align}\label{eq:llclausen}
  &\int_{0}^{\alpha}d\theta\;  \theta\; ct_\Lambda(\theta)
  =\int_{0}^{\alpha}d\theta \Bigg[\frac d {d\theta}\Big(\theta\log|2s_\Lambda(\theta)|\Big)-\log|2s_\Lambda(\theta)|\Bigg]\\
  &=\alpha\log|2s_\Lambda(\alpha)|-\int_{0}^{\alpha}d\theta\log|2s_\Lambda(\theta)|
  =\alpha\log|2s_\Lambda(\alpha)|+\frac{1}{2}\Cl_\Lambda(2\alpha),\nonumber
  \end{align}
  where $\Cl_\Lambda$ is the generalized Clausen function defined in Theorem \ref{prop:volideal}. Inserting this identity in the expression for the volume yields the volume formula for $\Lambda=\pm 1$ in Theorem \ref{prop:dualvol}.
  \end{proof}
  
  Note that the volume of the lightlike tetrahedron $ L \subset\X_\Lambda$   for $\Lambda=0$  is also obtained from the volume formula for a 3-simplex in 3d Minkowski space. Omitting the coordinate $x_2$ in the identification \eqref{eq:xmatrixrp} we can identify the vertices of $ L $  with   points in $\R^3$. The volume is then given by the Minkowski bilinear form $\langle\cdot,\cdot\rangle_{1,0,2}$ and the Lorentzian wedge product on $\R^3$ as
  \begin{align*}
  \vol( L )=\frac 1 6\, \Big |\Big\langle x_3-x_4, (x_1-x_4)\wedge (x_2-x_4)\Big\rangle\Big|=\frac 1 3 \alpha\beta(\alpha+\beta).
  \end{align*}
  
  It remains to clarify the relation between the volume formulas for a lightlike tetrahedron for $\Lambda=0$ and $\Lambda=\pm 1$. 
  For $\Lambda=-\ell^2=0$  the division by $\ell^2$ in the volume formula for $\Lambda=\pm 1$  is ill-defined.  However, in this case we have $\Cl_\Lambda(x)=-x\log|2s_\Lambda(x)|+x$ and hence the numerator of the volume formula for $\Lambda=\pm 1$ also vanishes.  In fact, we can obtain the volume formula for $\Lambda=0$ as a limit of the formula for $\Lambda=\pm 1$ if we extend the latter to $\Lambda\in \R$ by considering its expansion as a power series in $\ell$.

  \begin{corollary} The volume of a lightlike tetrahedron $ L \subset\X_\Lambda$ is given as a power series in its shortest edge lengths $\alpha,\beta$ and in $\Lambda$  by
  \begin{align}\label{eq:volexp}
  \vol( L )
  &=\sum_{k=1}^\infty \frac{4^{k} (-1)^{k-1}  \Lambda^{k-1} B_{2k}}{(2k+1)!} \sum_{j=1}^{k} { \binom{k+1}{j}} \alpha^j \beta^{k+1-j}\\
  &=\frac 1 3 \alpha\beta(\alpha+\beta) +O(\Lambda), \nonumber
  \end{align}
  where $B_{2k}$ is the $2k$th Bernoulli number.
  \end{corollary}
  
  \begin{proof}
  Using expression \eqref{eq:expellsincos}  for the generalized trigonometric functions in terms of the exponential map, which extends to general $\Lambda=-\ell^2\in\R$,  and the well-known Laurent series expansion of the cotangent and hyperbolic cotangent, we 
  obtain the power series 
   \begin{align*}
  \frac x {t_\Lambda(x)}=\sum_{k=0}^\infty\frac{4^{k}  B_{2k} (-1)^{k} \Lambda^k} {(2k)!} x^{2k}=1-\frac{\Lambda}{3} x^2-\frac {\Lambda^2}{45} x^4+ \ldots,
  \end{align*}
   for general  $\Lambda=-\ell^2\in\R$.
  Integrating this expression as in \eqref{eq:llclausen} yields
  \begin{align}\label{eq:helpclaus}
  &\frac 1 2 \Cl_\Lambda(2 y)+y\log|2s_\Lambda(y)|=\int_0^y dx\frac x{t_\Lambda(x)}\\
  =& \sum_{k=0}^\infty \frac{4^{k}B_{2k}(-1)^k\Lambda^k}{(2k+1)!} y^{2k+1}=y-\frac {\Lambda y^3}{9}-\frac {\Lambda^2 y^5}{225}+\ldots .\nonumber
  \end{align}
   Subtracting expression \eqref{eq:helpclaus} for $y=\alpha$ and $y=\beta$ from the one for $y=\alpha+\beta$ 
   annihilates the linear term. After dividing by $\ell^2=-\Lambda$ and applying the binomial formula one obtains the first line in  \eqref{eq:volexp}. 
  \end{proof}

  \paragraph{Acknowledgments}
  CS thanks Prof. Jinsung Park (KIAS) for comments and discussions during the preparation of this paper. 
  CM thanks Prof. Andreas Knauf (FAU) for comments on a draft of this paper.

  \paragraph{Funding statement}
  This work was supported in part by the National Research Foundation of Korea (NRF) grant funded by the Korea government (MSIT) (2019R1F1A1060827).
  
  This research was supported in part by Perimeter Institute for Theoretical Physics. Research at Perimeter Institute is supported by the Government of Canada through the Department of Innovation, Science and Economic Development and by the Province of Ontario through the Ministry of Research and Innovation.
  
  CS was supported by a KIAS Individual Grant (SP036102) via the Center for Mathematical Challenges at Korea Institute for Advanced Study.

\end{document}